%% file: Moeglin_s_parametrization_revised_2.tex
\documentclass[11pt]{amsart}
\usepackage{amsmath, amssymb}
\usepackage{amsfonts}
\usepackage{mathrsfs}
\usepackage[arrow,matrix,curve,cmtip,ps]{xy}

\usepackage{amsthm}
\usepackage{mathtools} 
\usepackage{array}
\usepackage{mathdots}

\usepackage[top= 2.25cm, bottom=2cm, left = 2cm, right= 2cm]{geometry}
\usepackage{hyperref}
\usepackage{setspace}

\allowdisplaybreaks

\numberwithin{equation}{section}

\input{Macros_reps}

\begin{document}
\title{On M{\oe}glin's parametrization of Arthur packets for p-adic quasisplit $Sp(N)$ and $SO(N)$}

\author{Bin Xu}

\address{Department of Mathematics and Statistics\\University of Calgary\\2500 University Dr. NW
Calgary\\Alberta\\Canada\\T2N 1N4}
\email{bin.xu2@ucalgary.ca}

\subjclass[2010]{22E50 (primary); 11F70 (secondary)}
\keywords{symplectic and orthogonal group, Arthur packet, endoscopy}


\maketitle

\begin{abstract}
We give a survey on M{\oe}glin's construction of representations in the Arthur packets for $p$-adic quasisplit symplectic and orthogonal groups. The emphasis is on comparing M{\oe}glin's parametrization of elements in the Arthur packets with that of Arthur. 

\end{abstract}

\section{Introduction}
\label{sec: introduction}

Let $F$ be a number field and $G$ be a quasisplit connected reductive group over $F$. The local components of the automorphic representations of $G$ belong to a very special class of irreducible smooth representations, which is usually referred to as the ``Arthur class". In the archimedean case, there is a geometric theory of irreducible smooth representations (see \cite{ABV:1992}), which suggests a possible way to characterize the Arthur class. In the $p$-adic case, the general characterization of the Arthur class remains a mystery. Nonetheless, when $G$ is a general linear group, the Arthur class is known in both cases due to M{\oe}glin-Waldspurger's classification of the discrete spectrum of automorphic representations of general linear groups \cite{MW:1989}. In this paper, we will only consider the $p$-adic case. So from now on, let us assume $F$ is a $p$-adic field, and we will also denote $G(F)$ by $G$, which should not cause any confusion in the context. To describe the Arthur class for general linear groups, we need to introduce some notations first. If $G = GL(n)$, let us take $B$ to be the group of upper-triangular matrices and $T$ to be the group of diagonal matrices, then the standard Levi subgroup $M$ can be identified with 
\[
GL(n_{1}) \times \cdots \times GL(n_{r})
\]
for any partition of $n = n_{1} + \cdots + n_{r}$ as follows
\[
\begin{pmatrix}
GL(n_{1})&& \\
&\ddots & \\
&& GL(n_{r})\\
\end{pmatrix}
\]
\begin{align*}
(g_{1}, \cdots, g_{r}) \longrightarrow \text{diag}\{g_{1}, \cdots, g_{r}\}.
\end{align*}
For $\r = \r_{1} \otimes \cdots \otimes \r_{r}$, where $\r_{i}$ is a finite-length smooth representation of $GL(n_{i})$ for $1 \leqslant i \leqslant r$, we denote the normalized parabolic induction $\Ind_{P}^{G} (\r)$ by 
\[
\r_{1} \times \cdots \times \r_{r}.
\] 
Moreover, we denote the direct sum of its irreducible subrepresentations by $<\r_{1} \times \cdots \times \r_{r}>$. An irreducible supercuspidal representation of a general linear group can always be written in a unique way as $\rho ||^{x} : = \rho \otimes |\det(\cdot)|^{x}$ for an irreducible unitary supercuspidal representation $\rho$ and a real number $x$. To fix notations, we will always denote by $\rho$ an irreducible unitary supercuspidal representation of $GL(d_{\rho})$. Now for a finite length arithmetic progression of real numbers of common length $1$ or $-1$
\[
x, \cdots, y
\]
and an irreducible unitary supercuspidal representation $\rho$ of $GL(d_{\rho})$, it is a general fact that 
\[
\rho||^{x} \times \cdots \times \rho||^{y}
\]
has a unique irreducible subrepresentation, denoted by $<\rho; x, \cdots, y>$ or $<x, \cdots, y>$. If $x \geqslant y$, it is called a Steinberg representation; if $x < y$, it is called a Speh representation. Such sequence of ordered numbers is called a {\bf segment}, and we denote it by $[x, y]$ or $\{x, \cdots, y\}$. 
In particular, when $x = -y > 0$, we can let $a = 2x + 1 \in \mathbb{Z}$ and write 
\[
St(\rho, a) := <\frac{a-1}{2}, \cdots, -\frac{a-1}{2}>,
\]
which is an irreducible smooth representation of $GL(ad_{\rho})$. It follows from Zelevinksy's classification theory that all discrete series of $GL(n)$ can be given by $St(\rho, a)$ for pairs $(\rho, a)$ satisfying $n = ad_{\rho}$, and this is a bijection. We define a {\bf generalized segment} to be a matrix 
\begin{align*}
  \begin{bmatrix}
   x_{11} & \cdots & x_{1n} \\
   \vdots &  & \vdots \\
   x_{m1} & \cdots & x_{mn}
  \end{bmatrix}
\end{align*}
such that each row is a decreasing (resp. increasing) segment and each column is an increasing (resp. decreasing) segment. The normalized induction
\[
\times_{i \in [1, m]} <\rho; x_{i1}, \cdots, x_{in} >
\]
has a unique irreducible subrepresentation, and we denote it by $<\rho; \{x_{ij}\}_{m \times n}>$. If there is no ambiguity with $\rho$, we will also write it as $<\{x_{ij}\}_{m \times n}>$ or
\begin{align*}
  \begin{pmatrix}
   x_{11} & \cdots & x_{1n} \\
   \vdots &  & \vdots \\
   x_{m1} & \cdots & x_{mn}
  \end{pmatrix}.
\end{align*}
Moreover, 
\[
<\rho; \{x_{ij}\}_{m \times n}> \cong <\rho; \{x_{ij}\}^{T}_{m \times n}>
\] 
where $\{x_{ij}\}^{T}_{m \times n}$ is the transpose of $\{x_{ij}\}_{m \times n}$. Let $a, b$ be positive integers, we define $Sp(St(\rho, a), b)$ to be the unique irreducible subrepresentation of 
\[
St(\rho, a)||^{-(b-1)/2} \times St(\rho, a)||^{-(b-3)/2} \times \cdots \times St(\rho, a)||^{(b-1)/2}.
\]
Then one can see $Sp(St(\rho, a), b)$ is given by the following generalized segment
\begin{align*}
  \begin{bmatrix}
   (a-b)/2 & \cdots & 1-(a+b)/2 \\
   \vdots &  & \vdots \\
   (a+b)/2-1 & \cdots & -(a-b)/2
  \end{bmatrix}.
\end{align*}
The Arthur class for $GL(n)$ consists of irreducible representations 
\begin{align}
\label{eq: Arthur class GL(N)}
\times_{i = 1}^{r} \Big(\underbrace{Sp(St(\rho_{i}, a_{i}), b_{i}) \times \cdots \times Sp(St(\rho_{i}, a_{i}), b_{i})}_{l_{i}}\Big)
\end{align}
for any set of triples $(\rho_{i}, a_{i}, b_{i})$ with multiplicities $l_{i}$ such that $\sum_{i = 1}^{r} l_{i} a_{i}b_{i}d_{\rho_{i}} = n$. In particular, it contains all the discrete series. The local Langlands correspondence for general linear groups gives a bijection between the set of equivalence classes of irreducible unitary supercuspidal representations of $GL(d)$ with the equivalence classes of $d$-dimensional irreducible unitary representations of the Weil group $W_{F}$. If we identify $\rho_{i}$ in \eqref{eq: Arthur class GL(N)} with the corresponding $d_{\rho_{i}}$-dimensional representations of $W_{F}$, then we get an equivalence class of $n$-dimensional representations of $W_{F} \times SL(2, \C) \times SL(2, \C)$ by taking 
\[
\bigoplus_{i = 1}^{r} l_{i} (\rho_{i} \otimes \nu_{a_{i}} \otimes \nu_{b_{i}}),
\]
where $\nu_{a_{i}}$ (resp. $\nu_{b_{i}}$) is the $(a_{i} - 1)$-th (resp. $(b_{i} - 1)$-th) symmetric power representation of $SL(2, \C)$. So the Arthur class for $GL(n)$ can be parameterized by the set of equivalence classes of $n$-dimensional representations of 
\[
\q: W_{F} \times SL(2, \C) \times SL(2, \C) \rightarrow GL(n, \C)
\]
such that $\q|_{W_{F}}$ is unitary and $\q|_{SL(2, \C) \times SL(2, \C)}$ is algebraic. We call such $\q$ an Arthur parameter for $GL(n)$.
The two copies of $SL(2, \C)$ in the definition of Arthur parameters have their own meanings. The first one introduced by Deligne, corresponds to some monodromy operator, and is usually integrated with the Weil group as $L_{F} : = W_{F} \times SL(2, \C)$, named Weil-Deligne group (or local Langlands group). The second $SL(2, \C)$ is introduced by Arthur, and it corresponds to the non-temperedness of the associated irreducible smooth representation of $GL(n)$ (cf. \eqref{eq: Arthur class GL(N)}).

For general $G$, we can define an Arthur parameter to be a $\D{G}$-conjugacy class of admissible homomorphisms from $L_{F} \times SL(2, \C)$ to $\L{G}$, which are bounded on their restrictions to $W_{F}$. And we denote the set of Arthur parameters by $\Q{G}$. It is conjectured that the Arthur class for $G$ should be parameterized by $\Q{G}$. 
To be more precise, for any $\q \in \Q{G}$, we are expecting to be able to associate it with a finite set $\Pkt{\q}$ of irreducible smooth representations of $G$, which is called an Arthur packet. The structure of $\Pkt{\q}$ can be very delicate in general, for example, we would expect these packets to have nontrivial intersections with each other. When $G$ is a classical group, M{\oe}glin has developed a theory to characterize the elements in $\Pkt{\q}$ (cf. \cite{Moeglin:2006}, \cite{Moeglin:2009}, etc.). The main goal of this paper is to present her results in the case of quasisplit symplectic and orthogonal groups. First of all, we need to give the definition of $\Pkt{\q}$ in these cases.

To simplify the discussion in the introduction, we assume $G = Sp(2n)$ if not specified. We should point out all the theorems and propositions that we state for symplectic groups below also have  their analogues for orthogonal groups. For $\q \in \Q{G}$, there is a natural $GL(N, \C)$-conjugacy class of embeddings $\L{G} \hookrightarrow GL(N, \C)$ for $N = 2n+1$. So we can view $\q$ as an equivalence class of representations of $L_{F} \times SL(2, \C)$, or an Arthur parameter for $GL(N)$. Moreover, such $\q$ is necessarily self-dual. So by the previous discussion we can associate it with an irreducible smooth representation $\r_{\q}$ of $GL(N)$ (cf. \eqref{eq: Arthur class GL(N)}), which is also self-dual. Arthur \cite{Arthur:2013} showed one can associate $\q$ with a ``multi-set" $\Pkt{\q}$ of irreducible smooth representations of $G$ such that the spectral transfer of some linear combination of characters in $\Pkt{\q}$ is the twisted character of $\r_{\q}$. If we define $\S{\q}$ to be the component group of the centralizer of the image of $\q$ in $\D{G}$ (which can be made independent of the choice of representatives of $\q$, and shown to be abelian), then Arthur further showed there is a ``canonical" map from $\Pkt{\q}$ to the characters $\D{\S{\q}}$ of $\S{\q}$. So for any element $\e \in \D{\S{\q}}$, we can write $\r(\q, \e)$ for the direct sum of elements in $\Pkt{\q}$ which are associated with $\e$, then $\r(\q, \e)$ is a finite-length smooth representation of $G$. 
The possibility for $\Pkt{\q}$ being a multi-set rather than a set suggests the irreducible constituents in $\r(\q, \e)$ may have multiplicities, and also $\r(\q, \e)$ may have common irreducible constituents for different $\e \in \D{\S{\q}}$. But these possibilities are all ruled out by the following deep theorem of M{\oe}glin.

\begin{theorem}[M{\oe}glin, \cite{Moeglin1:2011}]
\label{thm: main theorem 1}
For $G = Sp(2n)$ and $\q \in \Q{G}$, $\Pkt{\q}$ is multiplicity free.
\end{theorem}

In fact, for $\q \in \Q{G}$ and $\e \in \D{\S{\q}}$, M{\oe}glin constructed a finite-length semisimple smooth representation $\r_{M}(\q, \e)$ of $G$. She showed $\Pkt{\q}$ consists of $\r_{M}(\q, \e)$ for all $\e \in \D{\S{\q}}$, and by studying their properties she is able to conclude Theroem~\ref{thm: main theorem 1}. A subtle point here is $\r(\q, \e)$ in Arthur's parametrization can be different from $\r_{M}(\q, \e)$. This point has been emphasized in various works of M{\oe}glin, and she also gave the relation between these two. Our second goal in this paper is to make that relation more transparent, and in the meantime we are able to clarify the fact that the representations $\r_{M}(\q, \e)$ constructed by M{\oe}glin are indeed elements in the Arthur packet $\Pkt{\q}$. For this purpose, we would like to rewrite Arthur's parametrization $\r(\q, \e)$ by $\r_{W}(\q, \e)$ to emphasize its dependence on certain kind of Whittaker normalization (see Section~\ref{sec: Arthur packet}). And the relation between $\r_{W}(\q, \e)$ and $\r_{M}(\q, \e)$ can be given in the following theorem.

\begin{theorem}
\label{thm: main theorem 2}
For $G = Sp(2n)$ and $\q \in \Q{G}$, there exists a character $\e_{\q}^{M/W} \in \D{\S{\q}}$, such that for any $\e \in \D{\S{\q}}$
\[
\r_{M}(\q, \e) = \r_{W}(\q, \e \e_{\q}^{M/W}).
\]
\end{theorem}
For the statement in this theorem to be true, we have implicitly put some restrictions on M{\oe}glin's parametrization $\r_{M}(\q, \e)$. 
In the most general setting, we will attach $\r_{M}(\q, \e)$ to characters $\e$ in $\D{\S{\q^{>}}}$ (see Section~\ref{sec: Arthur parameter}) which contains $\D{\S{\q}}$, and we will also define $\e_{\q}^{M/W}$ in $\D{\S{\q^{>}}}$. 
The starting point of this comparison theorem is in the case of discrete series. Let us define
\[
\Pdt{G} := \{\p \in \Q{G}: \p = \+_{i = 1}^{r} \rho_{i} \otimes \nu_{a_{i}} \otimes \nu_{1}, \text{ and } \rho_{i}^{\vee} = \rho_{i}\}.
\] 
Then the following theorem of Arthur showed $\Pdt{G}$ parametrizes the discrete series of $G$.

\begin{theorem}[Arthur]
\label{thm: discrete series}
For $G = Sp(2n)$, the set of irreducible discrete series representations of $G$ admits a disjoint decomposition
\[
\Pkt{2}(G) = \bigsqcup_{\p \in \Pdt{G}} \Pkt{\p}.
\]
Moreover, for any $\p \in \Pdt{G}$ and $\e \in \D{\S{\p}}$, $\r_{W}(\p, \e)$ is an irreducible representation.
\end{theorem}

For $\p \in \Pdt{G}$ and $\e \in \D{\S{\p}}$, we can simply define 
\[
\r_{M}(\p, \e) := \r_{W}(\p, \e).
\] 
To justify this definition, we need to recall M{\oe}glin's construction (joint with Tadi{\'c}) of discrete series of $G$. We start by introducing some more notations, and here we will also include the case of special orthogonal groups.

If $G = Sp(2n)$, let us define it with respect to 
\[
\begin{pmatrix} 
      0 & -J_{n} \\
      J_{n} &  0 
\end{pmatrix},
\]
where 
\[
J_{n} = 
\begin{pmatrix}
        &&1\\
        &\iddots&\\
        1&&&
\end{pmatrix}.
\]
Let us take $B$ to be subgroup of upper-triangular matrices in $G$ and $T$ to be subgroup of diagonal matrices in $G$,  then the standard Levi subgroup $M$ can be identified with 
\[
GL(n_{1}) \times \cdots \times GL(n_{r}) \times G_{-}
\] 
for any partition $n = n_{1} + \cdots + n_{r} + n_{-}$ and $G_{-} = Sp(2n_{-})$ as follows
\[
\begin{pmatrix}
GL(n_{1})&&&&&&0 \\
&\ddots &&&&& \\
&& GL(n_{r})&&&&\\
&&&G_{-} &&&\\
&&&&GL(n_{r})&& \\
&&&&&\ddots&\\
0&&&&&&GL(n_{1})
\end{pmatrix}
\]
\begin{align}
\label{eq: embedding}
(g_{1}, \cdots g_{r}, g) \longrightarrow \text{diag}\{g_{1}, \cdots, g_{r}, g, {}_tg^{-1}_{r}, \cdots, {}_tg^{-1}_{1}\},
\end{align}
where ${}_tg_{i} = J_{n_{i}}{}^tg_{i}J^{-1}_{n_{i}}$ for $1 \leqslant i \leqslant r$. Note $n_{-}$ can be $0$, in which case we simply write $Sp(0) = 1$. For $\r = \r_{1} \otimes \cdots \otimes \r_{r} \otimes \sigma$, where $\r_{i}$ is a finite-length smooth representation of $GL(n_{i})$ for $1 \leqslant i \leqslant r$ and $\sigma$ is a finite-length smooth representation of $G_{-}$, we denote the normalized parabolic induction $\Ind_{P}^{G}(\r)$ by 
\[
\r_{1} \times \cdots \times \r_{r} \rtimes \sigma.
\] 
Moreover, we denote the direct sum of its irreducible subrepresentations by $<\r_{1} \times \cdots \times \r_{r} \rtimes \sigma>$. These notations can be easily extended to special orthogonal groups. If $G = SO(N)$ split, we define it with respect to $J_{N}$. When $N$ is odd, the situation is exactly the same as the symplectic case. When $N = 2n$, there are two distinctions. First, the standard Levi subgroups given through the embedding \eqref{eq: embedding} do not exhaust all standard Levi subgroups of $SO(2n)$. To get all of them, we need to take the $\theta_{0}$-conjugate of $M$ given in \eqref{eq: embedding}, where 
\[
\theta_{0} = 
\begin{pmatrix}
1 &&&&& \\
& \ddots &&&& \\
&&& 1 && \\
&& 1 &&& \\
&&&& \ddots & \\
&&&&& 1
\end{pmatrix}.
\]
Note $M^{\theta_{0}} \neq M$ only when $n_{-} = 0$ and $n_{r} > 1$. In order to distinguish the $\theta_{0}$-conjugate standard Levi subgroups of $SO(2n)$, we will only identify those Levi subgroups $M$ in \eqref{eq: embedding} with $GL(n_{1}) \times \cdots \times GL(n_{r}) \times G_{-}$, and we denote the other one simply by $M^{\theta_{0}}$. Second, if the partition $n = n_{1} + \cdots + n_{r} + n_{-}$ satisfies $n_{r} =1$ and $n_{-} = 0$, then we can rewrite it as $n = n_{1} + \cdots + n_{r - 1} + n'_{-}$ with $n'_{-} = 1$, and the corresponding Levi subgroup is the same. This is because $GL(1) \cong SO(2)$. To fix notation, we will always write it as $SO(2)$. In this paper, we will also consider $G = SO(2n, \eta)$, which is the outer form of the split $SO(2n)$ with respect to a quadratic extension $E/F$ and $\theta_{0}$. Here $\eta$ is the associated quadratic character of $E/F$ by the local class field theory. Then the standard Levi subgroups of $SO(2n, \eta)$ will be the outer form of those $\theta_{0}$-stable standard Levi subgroups of $SO(2n)$. In particular, they can be identified with $GL(n_{1}) \times \cdots \times GL(n_{r}) \times SO(n_{-}, \eta)$ and $n_{-} \neq 0$. Note in the case of $SO(8)$, there is another outer form, but we will not consider it in this paper.


Now we are back to the case $G = Sp(2n)$. For $\p = \+_{i = 1}^{r} \rho_{i} \otimes \nu_{a_{i}} \otimes \nu_{1} \in \Pdt{G}$, we define 
\[
Jord(\p) := \{(\rho_{i}, a_{i}): 1 \leqslant i \leqslant q \},
\] 
and 
\[
Jord_{\rho}(\p) := \{ a_{i} : \rho = \rho_{i}\}.
\]
Then we can identify $\D{\S{\p}}$ with the subspace of $\Two$-valued functions $\e(\cdot)$ on $Jord(\p)$ such that 
\[
\prod_{(\rho, a) \in Jord(\p)} \e(\rho, a) = 1
\]
(see Section~\ref{sec: Arthur parameter}). The following theorem gives a parametrization of irreducible supercuspidal representations of $G$.

\begin{theorem}[\cite{Moeglin:2011}, Theorem 1.5.1]
\label{thm: supercuspidal parametrization}
For $G = Sp(2n)$, the irreducible supercuspidal representations of $G$ are parametrized by $\p \in \Pdt{G}$, and $\e \in \D{\S{\p}}$ satisfying the following properties:

\begin{enumerate}

\item if $(\rho, a) \in Jord(\p)$, then $(\rho, a-2) \in Jord(\p)$ as long as $a - 2 > 0$;

\item if $(\rho, a), (\rho, a-2) \in Jord(\p)$, then $\e(\rho, a) \e(\rho, a - 2) = -1$;

\item if $(\rho, 2) \in Jord(\p)$, then $\e(\rho, 2) = -1$.

\end{enumerate}

\end{theorem}

For non-supercuspidal irreducible representations of $G$, we can characterize their cuspidal supports by the following proposition.

\begin{proposition}[\cite{Xu:preprint3}, Proposition 9.3]
\label{prop: parabolic reduction}
For $G = Sp(2n)$, suppose $\p \in \Pdt{G}$, and $\e \in \D{\S{\p}}$. For any $(\rho, a) \in Jord(\p)$, we denote by $a_{-}$ the biggest positive integer smaller than $a$ in $Jord_{\rho}(\p)$. And we also write $a_{min}$ for the minimum of $Jord_{\rho}(\p)$. If $a = a_{min}$, we let $a_{-} = 0$ if $a$ is even, and $-1$ otherwise. In this case, we always assume $\e(\rho, a)\e(\rho, a_{-}) = -1$.

\begin{enumerate}

\item If $\e(\rho, a)\e(\rho, a_{-}) = -1$ and $a_{-} < a - 2$, then 
\begin{align}
\label{eq: parabolic reduction 1}
\r_{W}(\p, \e) \hookrightarrow <(a-1)/2, \cdots, (a_{-} + 3)/2> \rtimes \r_{W}(\p', \e')
\end{align}
as the unique irreducible subrepresentation, where 
\[
Jord(\p') = Jord(\p) \cup \{(\rho, a_{-} + 2)\} \backslash \{(\rho, a)\},
\]
and 
\[
\e'(\cdot)= \e(\cdot) \text{ over } Jord(\p) \backslash \{(\rho, a)\}, \quad \quad \e'(\rho, a_{-}+2) = \e(\rho, a).
\]

\item If $\e(\rho, a)\e(\rho, a_{-}) = 1$, then 
\begin{align}
\label{eq: parabolic reduction 2}
\r_{W}(\p, \e) \hookrightarrow <(a-1)/2, \cdots, -(a_{-} -1)/2> \rtimes \r_{W}(\p', \e'),
\end{align}
where 
\[
Jord(\p') = Jord(\p) \backslash \{(\rho, a), (\rho, a_{-})\},
\]
and $\e'(\cdot)$ is the restriction of $\e(\cdot)$. In particular, suppose $\e_{1} \in \D{\S{\p}}$ satisfying $\e_{1}(\cdot) = \e(\cdot)$ over $Jord(\p')$ and
\[
\e_{1}(\rho, a) = -\e(\rho, a), \quad \quad \e_{1}(\rho, a_{-}) = -\e(\rho, a_{-}).
\]
Then the induced representation in \eqref{eq: parabolic reduction 2} has two irreducible subrepresentations, namely
\[
\r_{W}(\p, \e) \+ \r_{W}(\p, \e_{1}).
\]

\item If $\e(\rho, a_{min}) = 1$ and $a_{min}$ is even, then 
\begin{align}
\label{eq: parabolic reduction 3}
\r_{W}(\p, \e) \hookrightarrow <(a_{min}-1)/2, \cdots, a_{0}> \rtimes \r_{W}(\p', \e')
\end{align}
as the unique irreducible subrepresentation, where 
\[
Jord(\p') = Jord(\p) \backslash \{(\rho, a_{min})\},
\] 
and $\e'(\cdot)$ is the restriction of $\e(\cdot)$.

\end{enumerate}

\end{proposition}

The construction of discrete series by M{\oe}glin and Tadi{\'c} can be obtained by reversing the steps \eqref{eq: parabolic reduction 1}, \eqref{eq: parabolic reduction 2} and \eqref{eq: parabolic reduction 3} in this proposition. 
Finally, in the general construction of $\r_{M}(\q, \e)$, one requires various reducibility results, which are all based on the following basic criterion.

\begin{proposition}[\cite{Xu:preprint3}, Corollary 9.1]
\label{prop: cuspidal reducibility}
For $G = Sp(2n)$, suppose $\r$ is a supercuspidal representation of $G$ and $\r \in \Pkt{\p}$ for some $\p \in \Pdt{G}$. Then for any unitary irreducible supercuspidal representation $\rho$ of $GL(d_{\rho})$, the parabolic induction 
\[
\rho||^{\pm(a_{\rho} +1)/2} \rtimes \r
\]
reduces exactly for
\begin{align}
\label{eq: cuspidal reducibility full orthogonal group}
a_{\rho} = \begin{cases}
                 \text{ max } Jord_{\rho}(\p), & \text{ if } Jord_{\rho}(\p) \neq \emptyset, \\
                                    0, & \text{ if $Jord_{\rho}(\p) = \emptyset$, $\rho$ is self-dual and is of opposite type to $\D{G}$}, \\
                                    -1, & \text{ otherwise. }
                 \end{cases}
\end{align}

\end{proposition}

The main tool in M{\oe}glin's construction of elements in the Arthur packets of classical groups is the Jacquet module. Here we would like to summarize the relevant notations about Jacquet modules used in her work. For general $G$, we denote by $\Rep(G)$ the category of finite-length smooth representations of $G$. We include the zero space in $\Rep(G)$, and by an irreducible representation we always mean it is nonzero. Now let $G$ be a quasisplit symplectic or special orthogonal group of $\bar{F}$-rank $n$.  We fix a unitary irreducible supercuspidal representation $\rho$ of $GL(d_{\rho})$, and we assume $M = GL(d_{\rho}) \times G_{-}$ is the Levi component of a standard maximal parabolic subgroup $P$ of $G$. Note in case $G_{-} = 1$ and $G$ is special even orthogonal, we require $P$ to be contained in the standard parabolic subgroup of $GL(2n)$ by our convention. Then for $\r \in \Rep(G)$, we can decompose the semisimplification of the Jacquet module
\[
s.s. \Jac_{P}(\r) = \bigoplus_{i} \tau_{i} \otimes \sigma_{i},
\] 
where $\tau_{i} \in \Rep(GL(d_{\rho}))$ and $\sigma_{i} \in \Rep(G_{-})$, both of which are irreducible. We define $\Jac_{x}\r$ for any real number $x$ to be 
\[
\Jac_{x}(\r) = \bigoplus_{\tau_{i} = \rho||^{x}} \sigma_{i}.
\]
If we have an ordered sequence of real numbers $\{x_{1}, \cdots, x_{s}\}$, we can define
\[
\Jac_{x_{1}, \cdots, x_{s}}\r = \Jac_{x_{s}} \circ \cdots \circ \Jac_{x_{1}} \r.
\]
Moreover, let 
\[
\bar{\Jac}_{x} = \begin{cases}
                         \Jac_{x} + \Jac_{x} \circ \theta_{0},  & \text{ if $G = SO(2n)$ and $n = d_{\rho} \neq 1$, } \\
                         \Jac_{x},         & \text{ otherwise. }
                         \end{cases}
\]
Then $\bar{\Jac}_{x}$ defines a functor on the category of $O(2n)$-conjugacy classes of finite-length smooth representations of $SO(2n)$. It is not hard to see $\Jac_{x}$ can be defined for $GL(n)$ in a similar way by replacing $G_{-}$ by $GL(n_{-})$. Furthermore, we can define $\Jac^{op}_{x}$ analogous to $\Jac_{x}$ but with respect to $\rho^{\vee}$ and the standard Levi subgroup $GL(n_{-}) \times GL(d_{\rho^{\vee}})$. So let us define $\Jac^{\theta}_{x} = \Jac_{x} \circ \Jac^{op}_{-x}$ for $GL(n)$. There are some explicit formulas for computing these Jacquet modules, and we refer the readers to (\cite{Xu:preprint3}, Section 5).

{\bf Acknowledgements}: The author wants to thank the Institute for Advanced Study (Princeton), where this work was carried out. He also gratefully acknowledges the support from University of Calgary and the Pacific Institute for the Mathematical Sciences (PIMS), when this work took its final form. During his time at IAS, he was supported by the NSF under Grant No. DMS-1128155 and DMS-1252158.


\section{Arthur parameter}
\label{sec: Arthur parameter}

Let $F$ be a $p$-adic field and $G$ be a quasisplit symplectic or special orthogonal group. We define the local Langlands group as $L_{F} = W_{F} \times SL(2, \C)$, where $W_{F}$ is the usual Weil group. We write $\Gal{F} = \Gal{\bar{F}/F}$ for the absolute Galois group over $F$. Let $\D{G}$ be the complex dual group of $G$, and $\L{G}$ be the Langlands dual group of $G$. An Arthur parameter of $G$ is a $\D{G}$-conjugacy class of admissible homomorphisms
\(
\underline{\q}: L_{F} \times SL(2, \C) \longrightarrow \L{G},
\)
such that $\underline{\q}|_{W_{F}}$ is bounded. If $\underline{\q}|_{SL(2, \C)} = 1$, we say the parameter is tempered. We denote by $\Q{G}$ the set of Arthur parameters of $G$. Here we can simplify the Langlands dual groups as in the following table:
\begin{center}
\begin{spacing}{1.5}
\begin{tabular}{| c | c | }
     \hline
      $G$                        &      $\L{G}$              \\
     \hline
      $Sp(2n)$                &      $SO(2n+1, \C)$  \\
     \hline
      $SO(2n+1)$           &      $Sp(2n, \C)$ \\
     \hline
      $SO(2n, \eta)$       &      $SO(2n, \C) \rtimes \Gal{E/F}$     \\     
     \hline
\end{tabular}
\end{spacing}
\end{center}
In the last case, $\eta$ is a quadratic character associated with a quadratic extension $E/F$ and $\Gal{E/F}$ is the associated Galois group. We fix an isomorphism $SO(2n, \C) \rtimes \Gal{E/F} \cong O(2n, \C)$. So in either of these cases, there is a natural embedding $\xi_{N}$ of $\L{G}$ into $GL(N, \C)$ up to $GL(N, \C)$-conjugacy, where $N = 2n+1$ if $G = Sp(2n)$ or $N = 2n$ otherwise. We fix an outer automorphism $\theta_{0}$ of $G$ preserving an $F$-splitting. If $G$ is symplectic or special odd orthogonal, we let $\theta_{0} = id$. If $G$ is special even orthogonal, we let $\theta_{0}$ be induced from the conjugate action of the nonconnected component of the full orthogonal group. Let $\D{\theta}_{0}$ be the dual automorphism of $\theta_{0}$. We write $\Sigma_{0} = <\theta_{0}>$, $G^{\Sigma_{0}} = G \rtimes <\theta_{0}>$, and $\D{G}^{\Sigma_{0}} = \D{G} \rtimes <\D{\theta}_{0}>$. So in the special even orthogonal case, $G^{\Sigma_{0}}$ (resp. $\D{G}^{\Sigma_{0}}$) is isomorphic to the full (resp. complex) orthogonal group. Let $\x_{0}$ be the character of $G^{\Sigma_{0}}/G$, which is nontrivial when $G$ is special even orthogonal. 

By composing $\q$ with $\xi_{N}$, we can view $\q$ as an equivalence class of $N$-dimensional self-dual representation of $L_{F} \times SL(2, \C)$. So we can decompose $\q$ as follows
\begin{align}
\label{eq: Arthur parameter}
\q = \bigoplus_{i=1}^{r} l_{i} \q_{i} = \bigoplus_{i=1}^{r}l_{i}(\rho_{i} \otimes \nu_{a_{i}} \otimes \nu_{b_{i}}).
\end{align}
Here $\rho_{i}$ are equivalence classes of irreducible unitary representations of $W_{F}$, which can be identified with irreducible unitary supercuspidal representations of $GL(d_{\rho_{i}})$ under the local Langlands correspondence (cf. \cite{HarrisTaylor:2001}, \cite{Henniart:2000}, and \cite{Scholze:2013}). And $\nu_{a_{i}}$ (resp. $\nu_{b_{i}}$) are the $(a_{i}-1)$-th (resp. $(b_{i}-1)$-th) symmetric power representations of $SL(2, \C)$. The irreducible constituent $\rho_{i} \otimes \nu_{a_{i}} \otimes \nu_{b_{i}}$ has dimension $n_{i} = n_{(\rho_{i}, a_{i}, b_{i})}$ and multiplicity $l_{i}$. We define the multi-set of Jordan blocks for $\q$ as follows,
\[
Jord(\q) := \{(\rho_{i}, a_{i}, b_{i}) \text{ with multiplicity } l_{i}: 1 \leqslant i \leqslant r \}.
\]
For any $\rho$, let us define
\[
Jord_{\rho}(\q) := \{(\rho', a', b') \in Jord(\q): \rho' = \rho\}.
\]
Fix a representative $\underline{\q}$, we define for any subgroup $\Sigma \subseteq \Sigma_{0}$
\[
S^{\Sigma}_{\underline{\q}} = \Cent(\Im \underline{\q}, \D{G}^{\Sigma}),
\]
\[
\cS{\underline{\q}}^{\Sigma} = S^{\Sigma}_{\underline{\q}} / Z(\D{G})^{\Gal{F}},
\]
\[
\S{\underline{\q}}^{\Sigma} = \cS{\underline{\q}}^{\Sigma} / \cS{\underline{\q}}^{0} = S_{\underline{\q}}^{\Sigma} / S_{\underline{\q}}^{0}Z(\D{G})^{\Gal{F}}.
\]
We denote by $s_{\underline{\q}}$ the image of the nontrivial central element of $SL(2, \C)$ in $\S{\underline{\q}}$.

To characterize the centralizer groups $\S{\underline{\q}}$ and $\S{\underline{\q}}^{\Sigma_{0}}$, we need to introduce a parity condition on the set of Jordan blocks $Jord(\q)$. There is a common way to define the parity for self-dual irreducible representations $\rho$ of $W_{F}$ (see \cite{Xu:preprint3}, Section 3). We say $(\rho_{i}, a_{i}, b_{i}) $ is of {\bf orthogonal type} if $\rho_{i} \otimes \nu_{a_{i}} \otimes \nu_{b_{i}}$ factors through an orthogonal group, or equivalently $a_{i} + b_{i}$ is even when $\rho_{i}$ is of orthogonal type and $a_{i} + b_{i}$ is odd when $\rho_{i}$ is of symplectic type. Similarly we say $(\rho_{i}, a_{i}, b_{i})$ is of {\bf symplectic type} if $\rho_{i} \otimes \nu_{a_{i}} \otimes \nu_{b_{i}}$ factors through a symplectic group, or equivalently $a_{i} + b_{i}$ is odd when $\rho_{i}$ is of orthogonal type and $a_{i} + b_{i}$ is even when $\rho_{i}$ is of symplectic type. Let $\q_{p}$ be the parameter whose Jordan blocks consists of those in $Jord(\q)$ with the same parity as $\D{G}$, and 
let $\q_{np}$ be any Arthur parameter of general linear group such that 
\[
\q = \q_{np} \+ \q_{p} \+ \q_{np}^{\vee}.
\]
We denote by $Jord(\q)_{p}$ the set of Jordan blocks in $Jord(\q_{p})$ without multiplicity. After this preparation, we can identify those centralizer groups above with certain quotient space of $\Two$-valued functions on $Jord(\q)_{p}$. To be more precise, let $s_{0} = (s_{0,i}) \in \Two^{Jord(\q)_{p}}$ be defined as $s_{0,i} = 1$ if $l_{i}$ is even and $s_{0,i} = -1$ if $l_{i}$ is odd. Then
\[
\S{\underline{\q}}^{\Sigma_{0}} \cong \{s  = (s_{i}) \in \Two^{Jord(\q)_{p}} \} / <s_{0}>,
\]
and
\[
\S{\underline{\q}} \cong \{ s  = (s_{i}) \in \Two^{Jord(\q)_{p}} : \prod_{i}(s_{i})^{n_{i}} = 1\} / <s_{0}>
\] 
if $G$ is special even orthogonal. Under these identifications, $s_{\underline{\q}} = s_{\q} := (s_{\q, i}) \in \Two^{Jord(\q)_{p}}$ with $s_{\q, i} = (-1)^{l_{i}}$ if $b_{i}$ is even and $s_{\q, i} =1$ if $b_{i}$ is odd. Let us denote by $\S{\q}$ (resp. $\S{\q}^{\Sigma_{0}}$) the corresponding quotient space of $\Two$-valued functions on $Jord(\q)_{p}$ such that $\S{\q} \cong \S{\underline{\q}}$ (resp. $\S{\q}^{\Sigma_{0}} \cong \S{\underline{\q}}^{\Sigma_{0}}$).

There is a natural inner product on $\Two^{Jord(\q)_{p}}$ which identify its dual with itself. Let $\e = (\e_{i})$ and $s = (s_{i})$ be two elements in $\Two^{Jord(\q)_{p}}$, then their inner product is defined by $\e(s) = \prod_{i}(\e_{i} \ast s_{i})$, where
\[
\e_{i} \ast s_{i} = \begin{cases}
                           -1, & \text{ if } \e_{i} = s_{i} = -1, \\
                           1,  & \text{ otherwise. } 
                           \end{cases}
\]
So on the dual side, 
\[
\D{\S{\q}^{\Sigma_{0}}} = \{\e = (\e_{i}) \in \Two^{Jord(\q)_{p}}: \prod_{i} \e_{i}^{l_{i}} = 1\}.
\]
When $G$ is special even orthogonal, let $\e_{0} = (\e_{0,i}) \in \Two^{Jord(\q)_{p}}$ be defined as $\e_{0,i} = 1$ if $n_{i}$ is even, or $\e_{0,i} = -1$ if $n_{i}$ is odd, then $\e_{0} \in \D{\S{\q}^{\Sigma_{0}}}$ is always trivial when restricted to $\S{\q}$, and  
\[
\D{\S{\q}} = \{\e = (\e_{i}) \in \Two^{Jord(\q)_{p}}: \prod_{i} \e_{i}^{l_{i}} = 1\} / <\e_{0}>.
\]
In general, we can let $\e_{0} = 1$ if $G$ is not special even orthogonal. In this paper, we will always denote elements in $\D{\S{\q}^{\Sigma_{0}}}$ by $\e$ and denote its image in $\D{\S{\q}}$ by $\bar{\e}$.

For computational purpose, it is more convenient to view $\S{\q}^{\Sigma_{0}}$ as functions on $Jord(\q_{p})$. In fact there is a natural projection  
\begin{align}
\label{eq: projection}
\xymatrix{ \Two^{Jord(\q_{p})} \ar[rr]^{Cont} && \Two^{Jord(\q)_{p}} \\
                 s \ar@{|->}[rr] && s'}
\end{align}
such that 
\[
s'(\rho, a, b) = \prod_{\substack{(\rho', a', b') \in Jord(\q_{p}) \\ (\rho', a', b') = (\rho, a, b) \text{ in } Jord(\q)_{p} }} s(\rho', a', b')
\]
for $(\rho, a, b) \in Jord(\q)_{p}$. In particular, $s_{0}$ has a natural representative $s_{0}^{>}$ in $\Two^{Jord(\q_{p})}$ given by $s_{0}^{>}(\rho, a, b) = -1$ for all $(\rho, a, b) \in Jord(\q_{p})$. When $G$ is special even orthogonal, the determinant condition for defining $\S{\q}$ becomes 
\begin{align}
\label{eq: orthogonal group condition}
\prod_{(\rho, a, b) \in Jord(\q_{p})} s(\rho, a, b)^{n_{(\rho, a, b)}} = 1. 
\end{align}
Moreover, $s_{\q}$ also has a natural representative $s^{>}_{\q}$ in $\Two^{Jord(\q_{p})}$ given by $s_{\q}^{>}(\rho, a, b) = -1$ if $b$ is even or $1$ if $b$ is odd. 
We define
\[
\S{\q^{>}}^{\Sigma_{0}} = \{s(\cdot) \in \Two^{Jord(\q_{p})} \} / <s_{0}^{>}>,
\]
and 
\[
\S{\q^{>}} = \{ s(\cdot) \in \Two^{Jord(\q_{p})} : \prod_{(\rho, a, b) \in Jord(\q_{p})} s(\rho, a, b)^{n_{(\rho, a, b)}} = 1\} / <s_{0}^{>}>
\] 
if $G$ is special even orthogonal. Then there are surjections $\S{\q^{>}}^{\Sigma_{0}} \rightarrow \S{\q}^{\Sigma_{0}}$ and $\S{\q^{>}} \rightarrow \S{\q}$.

On the dual side, we have a natural inclusion
\[
\xymatrix{ \Two^{Jord(\q)_{p}} \ar@{^{(}->}[rr]^{Ext} && \Two^{Jord(\q_{p})} \\
                 \e \ar@{|->}[rr] && \e'}
\]
such that 
\[
\e'(\rho, a, b) = \e(\rho, a, b)
\]
for $(\rho, a, b) \in Jord(\q_{p})$. We can define an inner product on $\Two^{Jord(\q_{p})}$ as for $\Two^{Jord(\q)_{p}}$. Then this inclusion is adjoint to the previous projection in the sense that
\[
\e(Cont(s)) = Ext(\e)(s)
\]
for $\e \in \Two^{Jord(\q)_{p}}$ and $s \in \Two^{Jord(\q_{p})}$. Therefore $\e_{0}$ can also be viewed as a function on $Jord(\q_{p})$ through the inclusion map, and the condition imposed on defining $\D{\S{\q}^{\Sigma_{0}}}$ becomes
\[
\prod_{(\rho, a, b) \in Jord(\q_{p})} \e(\rho, a, b) = 1.
\]
We also define
\[
\D{\S{\q^{>}}^{\Sigma_{0}}} = \{\e(\cdot) \in \Two^{Jord(\q_{p})} : \prod_{(\rho, a, b) \in Jord(\q_{p})} \e(\rho, a, b) = 1\}, 
\]
and 
\[
\D{\S{\q^{>}}} = \{\e(\cdot) \in \Two^{Jord(\q_{p})} : \prod_{(\rho, a, b) \in Jord(\q_{p})} \e(\rho, a, b) = 1\} / <\e_{0}>
\] 
if $G$ is special even orthogonal. 
Then there are inclusions $\D{\S{\q}^{\Sigma_{0}}} \hookrightarrow \D{\S{\q^{>}}^{\Sigma_{0}}}$ and $\D{\S{\q}} \hookrightarrow \D{\S{\q^{>}}}$. For $\e \in \D{\S{\q^{>}}^{\Sigma_{0}}}$, we denote its image in $\D{\S{\q^{>}}}$ by $\bar{\e}$.

In the end, we are going to associate any Arthur parameter $\q \in \Q{G}$ with two Langlands parameters of $G$ naturally. For the first one, we define
\[
\p_{\q}(u) = \q \begin{pmatrix} u,  & \begin{pmatrix} |u|^{\frac{1}{2}} & 0 \\ 0 & |u|^{-\frac{1}{2}}  \end{pmatrix} \end{pmatrix},  \,\,\,\,\,\,\, u \in L_{F}.
\]
Note $\p_{\q} \in \P{G}$ is nontempered, and in the notation of \eqref{eq: Arthur parameter} we can write it as
\[
\p_{\q} = \bigoplus_{i=1}^{r} l_{i}\Big(\bigoplus_{j = 0}^{b_{i}-1}(\rho_{i}||^{(b_{i}-1)/2 - j} \otimes \nu_{a_{i}})\Big).
\]
For the second one, we can compose $\q$ with 
\[
\Delta: W_{F} \times SL(2, \C) \rightarrow W_{F} \times SL(2, \C) \times SL(2, \C),
\]
which is the diagonal embedding of $SL(2, \C)$ into $SL(2, \C) \times SL(2, \C)$ when restricted to $SL(2, \C)$, and is identity on $W_{F}$. Note the composition $\q_{d} := \q \circ \Delta \in \Pbd{G}$. To expand $\q_{d}$, we need to introduce some more notations. For $(\rho, a, b) \in Jord(\q)$, let us write $A = (a+b)/2-1$, $B = |a-b|/2$, and set $\zeta = \zeta_{a, b} = \text{Sign}(a - b)$ if $a \neq b$ and arbitrary otherwise. Then we can replace $(\rho, a, b)$ by $(\rho, A, B, \zeta)$. Under this new notation, we have
\[
\q_{d} = \bigoplus_{i=1}^{r} l_{i} \Big(\bigoplus_{j \in [A_{i}, B_{i}]} \rho_{i} \otimes \nu_{2j+1} \Big),
\]
where $j$ is taken over half-integers in the segment $[A_{i}, B_{i}]$.

Finally, $\Sigma_{0}$ acts on $\Q{G}$ through $\D{\theta}_{0}$, and we denote the corresponding set of $\Sigma_{0}$-orbits by $\cQ{G}$. It is clear that for $\q \in \Q{G}$, $Jord(\q)$ only depends on its image in $\cQ{G}$. It is because of this reason, we will also denote the elements in $\cQ{G}$ by $\q$. Moreover, through the natural embedding $\xi_{N}$, we can view $\cQ{G}$ as a subset of equivalence classes of $N$-dimensional self-dual representations of $L_{F} \times SL(2, \C)$.

\section{Endoscopy}
\label{sec: endoscopy}

Before we can introduce the Arthur packets, we need to talk about the relevant cases of endoscopy in this paper. The discussion here will be parallel with that in (\cite{Xu:preprint3}, Section 4). Suppose $\q \in \Q{G}$ and $s \in \cS{\underline{\q}}$ is semisimple. In our case, there is a quasisplit reductive group $H$ with the property that 
\[
\D{H} \cong \Cent(s, \D{G})^{0},
\] and the isomorphism extends to an embedding 
\[
\xi: \L{H} \rightarrow \L{G}
\] 
such that $\xi(\L{H}) \subseteq \Cent(s, \L{G})$ and $\underline{\q}$ factors through $\L{H}$. So from $\q$ we get a parameter $\q_{H} \in \Q{H}$. We say $(H, \underline{\q}_{H})$ corresponds to $(\underline{\q}, s)$ through $\xi$, and denote this relation by $(H, \underline{\q}_{H}) \rightarrow (\underline{\q}, s)$. Such $H$ is called an {\bf endoscopic group} of $G$. In the following examples we will always assume $\q = \q_{p}$. 

\begin{example}
\label{eg: endoscopy}
\begin{enumerate}

\item If $G = Sp(2n)$, then $\L{G} = SO(2n + 1, \C)$. For $s \in \S{\q^{>}}$, it gives a partition on $Jord(\q)$ depending on $s(\rho, a, b) = 1$ or $-1$, i.e.,
\[
Jord(\q) = Jord_{+} \sqcup Jord_{-}.
\]
Without loss of generality, let us assume 
\[
\sum_{(\rho, a, b) \in Jord_{+}} n_{(\rho, a, b)} = 2n_{I} + 1 = N_{I} \text{ and } \sum_{(\rho, a, b) \in Jord_{-}} n_{(\rho, a, b)} = 2n_{II} = N_{II}.
\] 
Define
\[
\eta_{I} = \eta_{II} = \prod_{(\rho, a, b) \in Jord_{-}}\eta_{(\rho, a, b)}
\]
where $\eta_{(\rho, a, b)}$ is the quadratic character dual to $\det (\rho \otimes \nu_{a} \otimes \nu_{b})$. Let 
\[
G_{I} = Sp(2n_{I}) \text{ and } G_{II} = SO(2n_{II}, \eta_{II}).
\] 
Then we have
\[
H = G_{I} \times G_{II} \text{ and } \L{H} = (\D{G}_{I} \times \D{G}_{II}) \rtimes \Gal{E_{II}/F},
\] 
where $E_{II}$ is the quadratic extension of $F$ associated with $\eta_{II}$. Let
\[
\xi_{i}: \L{G}_{i} \hookrightarrow GL(N_{i}, \C)
\]
be the natural embedding for $i = I, II$. Then 
\[
\xi := (\xi_{I} \otimes \eta_{I}) \+ \xi_{II}
\] 
factors through $\L{G}$ and defines an embedding $\L{H} \hookrightarrow \L{G}$. We define $\q_{I} \in \cQ{G_{I}}$ by 
\[
Jord(\q_{I}) := \{(\rho \otimes \eta_{I}, a, b):  (\rho, a, b) \in Jord_{+}\},
\] 
and $\q_{II} \in \cQ{G_{II}}$ by 
\[
Jord(\q_{II}) := \{(\rho, a, b) \in Jord_{-} \}.
\]
Let $\q_{H} = \q_{I} \times \q_{II}$.

\item If $G = SO(2n+1)$, then $\L{G} = Sp(2n, \C)$. For $s \in \S{\q^{>}}$, it gives a partition on $Jord(\q)$ depending on $s(\rho, a, b) = 1$ or $-1$, i.e.,
\[
Jord(\q) = Jord_{+} \sqcup Jord_{-}.
\]
We can assume 
\[
\sum_{(\rho, a, b) \in Jord_{+}} n_{(\rho, a, b)} = 2n_{I} = N_{I} \text{ and } \sum_{(\rho, a, b) \in Jord_{-}} n_{(\rho, a, b)} = 2n_{II} = N_{II}.
\] 
Define $\eta_{I} = \eta_{II} = 1$. Let 
\[
G_{I} = SO(2n_{I}+1) \text{ and } G_{II} = SO(2n_{II}+1).
\] 
Then we have 
\[
H = G_{I} \times G_{II} \text{ and } \L{H} = \D{G}_{I} \times \D{G}_{II}
\]
Let
\[
\xi_{i}: \L{G}_{i} \hookrightarrow GL(N_{i}, \C)
\]
be the natural embedding for $i = I, II$. Then 
\[
\xi := \xi_{I} \+ \xi_{II}
\] 
factors through $\L{G}$ and defines an embedding $\L{H} \hookrightarrow \L{G}$. We define $\q_{I} \in \cQ{G_{I}}$ by 
\[
Jord(\q_{I}) := \{(\rho, a, b) \in Jord_{+}\},
\] 
and $\q_{II} \in \cQ{G_{II}}$ by 
\[
Jord(\q_{II}) := \{(\rho, a, b) \in Jord_{-} \}.
\]
Let $\q_{H} = \q_{I} \times \q_{II}$.

\item If $G = SO(2n, \eta)$, then $\L{G} = SO(2n, \C) \rtimes \Gal{E/F}$. For $s \in \S{\q^{>}}$, it gives a partition on $Jord(\q)$ depending on $s(\rho, a, b) = 1$ or $-1$, i.e.,
\[
Jord(\q) = Jord_{+} \sqcup Jord_{-}.
\]
By the condition \eqref{eq: orthogonal group condition}, we can assume 
\[
\sum_{(\rho, a, b) \in Jord_{+}} n_{(\rho, a, b)} = 2n_{I} = N_{I} \text{ and } \sum_{(\rho, a, b) \in Jord_{-}} n_{(\rho, a, b)} = 2n_{II} = N_{II}.
\] 
Define
\[
\eta_{I} = \prod_{(\rho, a, b) \in Jord_{+}}\eta_{(\rho, a, b)} \text{ and } \eta_{II} = \prod_{(\rho, a, b) \in Jord_{-}}\eta_{(\rho, a, b)},
\]
where $\eta_{(\rho, a, b)}$ is the quadratic character dual to $\det (\rho \otimes \nu_{a} \otimes \nu_{b})$. We also denote by $E_{i}$ the quadratic extension of $F$ associated with $\eta_{i}$ for $i = I, II$.
Let 
\[
G_{I} = SO(2n_{I}, \eta_{I}) \text{ and } G_{II} = SO(2n_{II}, \eta_{II}).
\]
Then we have 
\[
H = G_{I} \times G_{II} \text{ and } \L{H}  = (\D{G}_{I} \times \D{G}_{II}) \rtimes \Gal{L/F} 
\]
where $L = E_{I} E_{II}$. Let
\[
\xi_{i}: \L{G}_{i} \hookrightarrow GL(N_{i}, \C)
\]
be the natural embedding for $i = I, II$. Then 
\[
\xi := \xi_{I} \+ \xi_{II}
\] 
factors through $\L{G}$ and defines an embedding $\L{H} \hookrightarrow \L{G}$. We define $\q_{I} \in \cQ{G_{I}}$ by 
\[
Jord(\q_{I}) := \{(\rho, a, b) \in Jord_{+}\},
\] 
and $\q_{II} \in \cQ{G_{II}}$ by 
\[
Jord(\q_{II}) := \{(\rho, a, b) \in Jord_{-} \}.
\]
Let $\q_{H} = \q_{I} \times \q_{II}$.

\end{enumerate}
\end{example}

In the examples above, $H$ is called an {\bf elliptic endoscopic group} of $G$. We can define $\cQ{H} = \cQ{G_{I}} \times \cQ{G_{II}}$, then $\q_{H} \in \cQ{H}$. For $s \in \S{\q^{>}}$, we still say $(H, \q_{H})$ correspond to $(\q, s)$ through $\xi$, and denote this relation by $(H, \q_{H}) \rightarrow (\q, s)$. 

In part (3), it is possible to also choose $s \in \S{\q^{>}}^{\Sigma_{0}}$ but not in $\S{\q^{>}}$, and then we get a partition on $Jord(\q)$, i.e., 
\[
Jord(\q) = Jord_{+} \sqcup Jord_{-}
\]
so that 
\[
\sum_{(\rho, a, b) \in Jord_{+}} n_{(\rho, a, b)} = 2n_{I} + 1= N_{I} \text{ and } \sum_{(\rho, a, b) \in Jord_{-}} n_{(\rho, a, b)} = 2n_{II} + 1 = N_{II}.
\] 
Define
\[
\eta_{I} = \prod_{(\rho, a, b) \in Jord_{+}}\eta_{(\rho, a, b)} \text{ and } \eta_{II} = \prod_{(\rho, a, b) \in Jord_{-}}\eta_{(\rho, a, b)},
\]
where $\eta_{(\rho, a, b)}$ is the quadratic character dual to $\det (\rho \otimes \nu_{a} \otimes \nu_{b})$.
Let 
\[
G_{I} = Sp(2n_{I}) \text{ and } G_{II} = Sp(2n_{II})
\] 
Then we can define $\q_{I} \in \cQ{G_{I}}$ by 
\[
Jord(\q_{I}) := \{(\rho \otimes \eta_{I}, a, b) \in Jord_{+}\},
\] 
and $\q_{II} \in \cQ{G_{II}}$ by 
\[
Jord(\q_{II}) := \{(\rho \otimes \eta_{II}, a, b) \in Jord_{-} \}.
\]
Let
\[
H = G_{I} \times G_{II} \text{ and } \L{H} = \D{G}_{I} \times \D{G}_{II}.
\]
In this case, $H$ is called a {\bf twisted elliptic endoscopic group} of $G$. Let 
\[
\xi_{i}: \L{G}_{i} \hookrightarrow GL(N_{i}, \C)
\]
be the natural embedding for $i = I, II$. Then 
\[
\xi := (\xi_{I} \otimes \eta_{I}) \+ (\xi_{II} \otimes \eta_{II})
\] 
factors through $\L{G}$ and defines an embedding $\L{H} \hookrightarrow \L{G}$. Let 
\[
\q_{H} = \q_{I} \times \q_{II}.
\] 
We say $(H, \q_{H})$ corresponds to $(\q, s)$ through $\xi$, and write $(H, \q_{H}) \rightarrow (\q, s)$.

In this paper, we also want to consider the twisted elliptic endoscopic groups of $GL(N)$, but we will only need the simplest case here. Recall for $\q \in \Q{G}$, we can view $\underline{\q}$ as a self-dual $N$-dimensional representation through the natural embedding 
\[
\xi_{N}: \L{G} \rightarrow GL(N, \C),
\]
and in this way we get a self-dual Arthur parameter for $GL(N)$. We fix an outer automorphism $\theta_{N}$ of $GL(N)$ preserving an $F$-splitting, and let $\D{\theta}_{N}$ be the dual automorphism on $GL(N, \C)$, then $\xi_{N}(\L{G}) \subseteq \Cent(s, GL(N, \C))$ and $\D{G} = \Cent(s, GL(N, \C))^{0}$ for some semisimple $s \in GL(N, \C) \rtimes \D{\theta}_{N}$. So we call $G$ a twisted elliptic endoscopic group of $GL(N)$.

What lies in the heart of endoscopy theory is a (twisted) transfer map on the spaces of smooth compactly supported functions from $G$ to its (twisted) elliptic endoscopic group $H$ (similarly from $GL(N)$ to its twisted elliptic endoscopic group $G$). The existence of the (twisted) transfer map is quite deep, and it was conjectured by Langlands, Shelstad and Kottwitz. In a series of papers Waldspurger \cite{Waldspurger:1995} \cite{Waldspurger:1997} \cite{Waldspurger:2006} \cite{Waldspurger:2008} was able to reduce it to the {\bf Fundamental Lemma} for Lie algebras over the function fields. Finally it is in this particular form of the fundamental lemma, Ngo \cite{Ngo:2010} gave his celebrated proof. Let us denote such transfers by 
\begin{align}
\label{eq: endoscopic transfer}
\xymatrix{C^{\infty}_{c}(G) \ar[r]  & C^{\infty}_{c}(H) \\
                f \ar[r] & f^{H}}
\end{align}
and similarly 
\begin{align}
\label{eq: twisted endoscopic transfer}
\xymatrix{C^{\infty}_{c}(GL(N)) \ar[r]  & C^{\infty}_{c}(G) \\
                f \ar[r] & f^{G}.}
\end{align}
In the definition of the (twisted) transfer maps, there is a normalization issue. To resolve that, we will always fix a $\Sigma_{0}$-stable (resp. $\theta_{N}$-stable) Whittaker datum for $G$ (resp. $GL(N)$) in this paper, and we will take the so-called Whittaker normalization on the transfer maps. We should also point out these transfer maps are only well defined after we pass to the space of (twisted) {\bf orbital integrals} on the source and the space {\bf stable orbital integrals} on the target. Note the space of (twisted) (resp. stable) orbital integrals are dual to the space of (twisted) (resp. stable) invariant distributions on $G$, i.e. one can view the (twisted) (resp. stable) invariant distributions of $G$ as linear functionals of the space of (twisted) (resp. stable) orbital integrals. So dual to these transfer maps, the stable distributions on $H$ (resp. $G$) will map to (twisted) invariant distributions on $G$ (resp. $GL(N)$). We call this map the (twisted) {\bf spectral endoscopic transfer}. Since we can identity $C^{\infty}_{c}(G \rtimes \theta_{0})$ (resp. $C^{\infty}_{c}(GL(N) \rtimes \theta_{N})$) with $C^{\infty}_{c}(G)$ (resp. $C^{\infty}_{c}(GL(N))$) by sending $g \rtimes \theta_{0}$ (resp. $g_{N} \rtimes \theta_{N}$) to $g$ (resp. $g_{N}$), we can define the twisted transfer map also for $C^{\infty}_{c}(G \rtimes \theta_{0})$ (resp. $C^{\infty}_{c}(GL(N) \rtimes \theta_{N})$).

If $\r$ is an irreducible smooth representation of $G$, then it defines an invariant distribution on $G$ by the trace of 
\[
\r(f) = \int_{G}f(g)\r(g)dg
\] 
for $f \in C^{\infty}_{c}(G)$. We call this the character of $\r$ and denote it by $f_{G}(\r)$. For any irreducible representation $\r^{\Sigma_{0}}$ of $G^{\Sigma_{0}}$, which contains $\r$ in its restriction to $G$, we define a twisted invariant distribution on $G$ by the trace of 
\[
\r^{\Sigma_{0}}(f) = \int_{G \rtimes \theta_{0}} f(g)\r^{\Sigma_{0}}(g)dg
\] 
for $f \in C^{\infty}_{c}(G \rtimes \theta_{0})$. We call this the twisted character of $G$, and denote it by $f_{G}(\r^{\Sigma_{0}})$. We can also define the twisted characters for $GL(N)$ similarly, but we will write it in a slightly different way. Let $\r$ be a self-dual irreducible smooth representation of $GL(N)$, we can define a twisted invariant distribution on $GL(N)$ by taking the trace of 
\[
\r(f) \circ A_{\r}(\theta_{N})
\] 
for $f \in C^{\infty}_{c}(GL(N))$, where $A_{\r}(\theta_{N})$ is an intertwining operator between $\r$ and $\r^{\theta_{N}}$. We call this the twisted character of $\r$ and denote it by $f_{N^{\theta}}(\r)$. 

Since the (twisted) elliptic endoscopic groups $H$ in our case are all products of quasisplit symplectic and special orthogonal groups, we can define a group of automorphisms of $H$ by taking the product of $\Sigma_{0}$ on each factor, and we denote this group again by $\Sigma_{0}$. Let $\sH(G)$ (resp. $\sH(H)$) be the subspace of $\Sigma_{0}$-invariant functions in $C_{c}^{\infty}(G)$ (resp. $C_{c}^{\infty}(H)$). Then it follows from a simple property of the transfer map (which we will not explain here) that we can restrict both \eqref{eq: endoscopic transfer} and \eqref{eq: twisted endoscopic transfer} to $\sH(G)$ and $\sH(H)$.

\section{Arthur packet}
\label{sec: Arthur packet}

For $\q \in \cQ{G}$, we define 
\[
\r_{\q} = \times_{(\rho, a, b) \in Jord(\q)} Sp(St(\rho, a), b).
\]
From \cite{Tadic:1986}, we know $\r_{\q}$ is a unitary self-dual irreducible representation of $GL(N)$, and there is a Whittaker normalization of the intertwining operator $A_{\r_{\q}}(\theta_{N})$ on $\r_{\q}$ (see \cite{Arthur:2013}, 2.2). Now we can state Arthur's local theory for $G$.

\begin{theorem}[Arthur]
\label{thm: Arthur packet}

For any $\q \in \cQ{G}$ and $\bar{\e} \in \D{\S{\q}}$, there is a canonical way to associate a finite-length semisimple unitary representation viewed as $\sH(G)$-module $\r(\q, \bar{\e})$ (which can be zero), satisfying the following properties:

\begin{enumerate}

\item
\[
f(\q) := \sum_{\bar{\e} \in \D{\S{\q}}} \bar{\e}(s_{\q})f_{G}(\r(\q, \bar{\e}))
\]
defines a stable distribution for $f \in \sH(G)$. Moreover,
\begin{align}
\label{eq: nontempered character relation GL(N)}
f^{G}(\q) = f_{N^{\theta}}(\r_{\q}), \,\,\,\,\,\,\,\,\,\,\,\,\,\,   f \in C^{\infty}_{c}(GL(N)),
\end{align}
after we normalize the Haar measures on $G$ and $GL(N)$ in a compatible way.

\item Suppose $\q = \q_{p}$ and $s \in \S{\q^{>}}$. Let $(H, \q_{H}) \rightarrow (\q, s)$, and we define a stable distribution $f(\q_{H})$ for $f \in \sH(H)$ as in (1), then after we normalize the Haar measures on $G$ and $H$ in a compatible way the following identity holds
\begin{align}
\label{eq: nontempered character relation}
f^{H}(\q_{H}) = \sum_{\bar{\e} \in \D{\S{\q}}} \bar{\e}(ss_{\q})f_{G}(\r(\q, \bar{\e})), \,\,\,\,\,\,\,\,\,\,\, f \in \sH(G),
\end{align}
where we denote the image of $s$ in $\S{\q}$ again by $s$.

\end{enumerate}
\end{theorem}

When $G$ is special even orthogonal, we have an additional character relation. 

\begin{theorem}[Arthur]
\label{thm: character relation full orthogonal group}
Suppose $G$ is special even orthogonal, $\q = \q_{p} \in \cQ{G}$ and $\e \in \D{\S{\q}^{\Sigma_{0}}}$, for any irreducible representation $\r$ viewed as $\sH(G)$-module $[\r]$ in $\r(\q, \bar{\e})$ such that $\r^{\theta_{0}} \cong \r$, one can associate it with an extension $\r^{\Sigma_{0}}$ to $G^{\Sigma_{0}}$. Then for any $s \in \S{\q^{>}}^{\Sigma_{0}}$ but not in $\S{\q^{>}}$ and $(H, \q_{H}) \rightarrow (\q, s)$ the following identity holds
\begin{align}
\label{eq: character relation full orthogonal group}
f^{H}(\q_{H}) = \sum_{\substack{\bar{\e} \in \D{\S{\q}}, \, [\r] \in \r(\q, \bar{\e}) : \\ \r^{\theta_{0}} \cong \r  }} \e(ss_{\q}) f_{G}(\r^{\Sigma_{0}}), \,\,\,\,\,\,\,\,\,\,\, f \in C^{\infty}_{c}(G \rtimes \theta_{0}),
\end{align}
where $\e \in \D{\S{\q}^{\Sigma_{0}}}$ is in the preimage of $\bar{\e}$, and it depends on the extension $\r^{\Sigma_{0}}$. We denote the image of $s$ in $\S{\q}^{\Sigma_{0}}$ again by $s$, and we normalize the Haar measures on $G$ and $H$ in a compatible way. 
\end{theorem}

We denote the set of $\sH(G)$-modules $\r(\q, \bar{\e})$ for fixed $\q \in \cQ{G}$ and all $\bar{\e} \in \D{\S{\q}}$ by $\cPkt{\q}$. 
One can see from both \eqref{eq: nontempered character relation GL(N)} and \eqref{eq: nontempered character relation} that the parametrization inside $\cPkt{\q}$ by $\D{\S{\q}}$ depends on the normalization of $A_{\r_{\q}}(\theta_{N})$ and also those of intertwining operators related to $\q_{H}$ (i.e., $A_{\r_{\q_{i}}}(\theta_{N_{i}})$ for $i = I, II$). In Arthur's theory, we always use the Whittaker normalization, as it is the most natural normalization from the global point of view, and it is in this sense that we say the association of $\r(\q, \bar{\e})$ with $\bar{\e} \in \D{\S{\q}}$ is canonical. But as it has been pointed out in \cite{MW:2006}, locally there is no reason to privilege the Whittaker normalization. Later on we will discuss another normalization used by M{\oe}glin and Waldspurger in \cite{MW:2006}, which is critical for studying the structure of $\r(\q, \bar{\e})$. So in order to distinguish different parametrizations with respect to various normalizations, we will denote $\r(\q, \bar{\e})$ in Arthur's theory by $\r_{W}(\q, \bar{\e})$, and similarly denote $f(\q)$ by $f_{W}(\q)$ and denote $f_{N^{\theta}}(\r_{\q})$ by $f_{N^{\theta}, W}(\r_{\q})$. 

Unlike the tempered case where all $\r_{W}(\q, \bar{\e})$ are distinct and irreducible (see Theorem~\ref{thm: discrete series} and \cite{Xu:preprint3}, Theorem 2.2), Arthur's theory says little about $\r_{W}(\q, \bar{\e})$ except for its unitarity. In fact, $\r_{W}(\q, \bar{\e})$ can be reducible or even zero in general, and it is the main goal of this paper to explore the inner structure of $\r_{W}(\q, \bar{\e})$. To do so, we will mainly follow \cite{Moeglin:2006}, \cite{Moeglin:2009} and \cite{MW:2006}. 

As a consequence of M{\oe}glin's results about $\r_{W}(\q, \bar{\e})$, we will be able to define the Arthur packet for $G^{\Sigma_{0}}$ and describe its structure (see Section~\ref{sec: general case}). 
To begin with, we define $\Pkt{\p}^{\Sigma_{0}}$ for $\p \in \cPdt{G}$ to be set of irreducible representations of $G^{\Sigma_{0}}$, whose restriction to $G$ belong to $\cPkt{\p}$. Then Theorem~\ref{thm: character relation full orthogonal group} allows us to parametrize $\Pkt{\p}^{\Sigma_{0}}$ by $\D{\S{\p}^{\Sigma_{0}}}$, and we have the following result.

\begin{theorem}[Arthur]
\label{thm: discrete series full orthogonal group}
Suppose $\p \in \cPdt{G}$, there is a canonical bijection between $\Pkt{\p}^{\Sigma_{0}}$ and $\D{\S{\p}^{\Sigma_{0}}}$
\[
\xymatrix{\D{\S{\p}^{\Sigma_{0}}}  \ar[r] & \Pkt{\p}^{\Sigma_{0}} \\
                 \e \ar@{{|}->}[r]  & \r^{\Sigma_{0}}_{W}(\p, \e),}
\]
such that 
\begin{itemize}

\item
\begin{align}
\label{eq: discrete series full orthogonal character}
\r^{\Sigma_{0}}_{W}(\p, \e \e_{0}) \cong \r^{\Sigma_{0}}_{W}(\p, \e) \otimes \x_{0}.
\end{align}

\item
$[\r^{\Sigma_{0}}_{W}(\p, \e)|_{G}] = 2 \r_{W}(\p , \bar{\e})$ if $G$ is special even orthogonal and $\S{\p}^{\Sigma_{0}} = \S{\p}$, or $\r_{W}(\p ,\bar{\e})$ otherwise.

\item For any $s \in \S{\p}^{\Sigma_{0}}$ but not in $\S{\p}$ and $(H, \p_{H}) \rightarrow (\p, s)$, the following identity holds

\begin{align*}
f^{H}_{W}(\p_{H}) = \sum_{\bar{\e} \in \D{\S{\p}}} \e(s) f_{G}(\r_{W}^{\Sigma_{0}}(\p, \e)), \,\,\,\,\,\,\,\,\,\,\, f \in C^{\infty}_{c}(G \rtimes \theta_{0}).
\end{align*}

\end{itemize}

\end{theorem}


\section{M{\oe}glin-Waldspurger's normalization}
\label{sec: M-W normalization}

The main reference for this section is \cite{MW:2006}. Suppose $\q \in \cQ{G}$, we denote the normalized action of $\theta_{N}$ on $\r_{\q}$ by $\theta(\q)$ for simplicity. If it is the Whittaker normalization, we denote it by $\theta_{W}(\q)$. Our aim is to introduce the normalization used by M{\oe}glin and Waldspurger, which we denote by $\theta_{MW}(\q)$, and to calculate explicitly the difference $\theta_{MW}(\q)/\theta_{W}(\q)$. 

To give the definition, we need to specify a class of parameters in $\cQ{G}$ called parameters with ``discrete diagonal restriction". To be more precise, $\q \in \cQ{G}$ is said to have {\bf discrete diagonal restriction} if $\q_{d} \in \cPdt{G}$. It is an easy exercise to see that this is equivalent to require $\q = \q_{p}$ and for any fixed $\rho$, the segments $[A, B]$ for $(\rho, A, B, \zeta) \in Jord_{\rho}(\q)$ are disjoint. In particular this implies $Jord(\q)$ is multiplicity free. Among this class of parameters, we call $\q$ is {\bf elementary} if $A = B$ for all $(\rho, A, B, \zeta) \in Jord(\q)$, or equivalently $inf(a, b) = 1$ for all $(\rho, a, b) \in Jord(\q)$. Note in the original terminology of M{\oe}glin and Waldspurger, elementary parameters are not required to have discrete diagonal restriction, nevertheless whenever they treat the elementary parameters, they include the condition of discrete diagonal restriction. This is the reason that we include the condition of discrete diagonal restriction in our definition of elementary parameters. For simplicity, if $\q$ is elementary we also denote by $Jord_{\rho}(\q_{d})$ the set of integers $\alpha$ such that $(\rho, \alpha, 1) \in Jord(\q_{d})$, and we write $(\rho, \alpha, \delta_{\alpha})$ for $(\rho, (\alpha-1)/2, (\alpha-1)/2, \delta_{\alpha}) \in Jord(\q)$.

We first give the definition of $\theta_{MW}(\q)$ for those elementary parameters. Suppose for all $(\rho, B, B, \zeta) \in Jord(\q)$, we have $B=0$, then simply let $\theta_{MW}(\q) = \theta_{W}(\q)$. Otherwise, we fix $\rho$ and let $B_{0}$ be the smallest number with $(\rho, B_{0}, B_{0}, \zeta_{0}) \in Jord(\q)$. If $B_{0} \neq 0$, we have 
\[
\r_{\q} \hookrightarrow \rho||^{\zeta_{0} B_{0}} \times \r_{\q'} \times \rho||^{-\zeta_{0} B_{0}}
\]
as the unique irreducible subrepresentation, where $Jord(\q')$ is obtained from $Jord(\q)$ by changing $(\rho, B_{0}, B_{0}, \zeta_{0})$ to $(\rho, B_{0} - 1, B_{0} - 1, \zeta_{0})$ when $B_{0} \geqslant 1$, or removing $(\rho, B_{0}, B_{0}, \zeta_{0})$ otherwise. Then we take $\theta_{MW}(\q)$ to be induced from $\theta_{MW}(\q')$. If $B_{0} = 0$, let $B_{1}$ be the next smallest number with $(\rho, B_{1}, B_{1}, \zeta_{1}) \in Jord(\q)$, and we have
\[
\r_{\q} \hookrightarrow <\zeta_{1}B_{1}, \cdots, 0>  \times \r_{\q'} \times <0, \cdots, -\zeta_{1}B_{1}>
\]
where $Jord(\q')$ is obtained from $Jord(\q)$ by removing $(\rho, B_{0}, B_{0}, \zeta_{0})$ and $(\rho, B_{1}, B_{1}, \zeta_{1})$. Note $\r_{\q}$ appears with multiplicity one in the induced representation, then again we take $\theta_{MW}(\q)$ to be induced from $\theta_{MW}(\q')$. This finishes the case of elementary parameters.

Next we consider the case of parameters with discrete diagonal restriction. We choose $(\rho, A, B, \zeta)$ with  $A > B$, then 
\[
\r_{\q} \hookrightarrow <\zeta B, \cdots, -\zeta A> \times \r_{\q'} \times <\zeta A, \cdots, -\zeta B>,
\]
as the unique irreducible subrepresentation, where $Jord(\q') = Jord(\q) \cup \{(\rho, A-1, B+1, \zeta)\} \backslash \{(\rho, A, B, \zeta)\}$. Then we take $\theta_{MW}(\q)$ to be induced from $\theta_{MW}(\q')$.

\begin{lemma}
In the set up above, $\theta_{MW}(\q)$ is independent of the choice of $(\rho, A, B, \zeta)$.
\end{lemma}

The proof of this Lemma can be found in (\cite{MW:2006}, Lemma 1.12.1 and Lemma 1.12.2).

Now we can consider the general case. If $\q \neq \q_{p}$, we can write
\[
\r_{\q} \cong \Big(\times_{(\rho, a, b)}Sp(St(\rho, a), b)\Big) \times \r_{\q_{p}} \times \Big(\times_{(\rho, a, b)}Sp(St(\rho, a), b)^{\vee}\Big),
\]
where $(\rho, a, b)$ are taken over $Jord(\q_{np})$, and hence define $\theta_{MW}(\q)$ to be induced from $\theta_{MW}(\q_{p})$. So without loss of generality, we may assume $\q = \q_{p}$. The general case requires us to put some total order $>_{\q}$ on $Jord(\q_{p})$ satisfying the following condition. 

($\mathcal{P}$): $\forall (\rho, A, B, \zeta), (\rho, A', B', \zeta') \in Jord(\q)$ with $A > A', B > B'$ and $\zeta = \zeta'$, then $(\rho, A, B, \zeta) >_{\q} (\rho, A', B', \zeta')$.

The necessity of this condition will be discussed in a moment. The point is there are many orders satisfying this condition and we do not have a privileged one in general. Nonetheless, for parameters with discrete diagonal restriction, we can always choose an order such that for any $\rho$, $(\rho, A, B, \zeta) >_{\q} (\rho, A', B', \zeta')$ if and only if $A > A'$. We call such orders the {\bf natural orders} for parameters with discrete diagonal restriction. For $\q \in \cQ{G}$ with order $>_{\q}$, we call $\q_{\gg} \in \cQ{G_{\gg}}$ with order $>_{\q_{\gg}}$ dominates $\q$ with respect to $>_{\q}$, if there is an order preserving bijection between $Jord(\q_{\gg})$ and $Jord(\q)$, which sends $(\rho, A_{\gg}, B_{\gg}, \zeta_{\gg})$ to $(\rho, A, B, \zeta)$ satisfying $A_{\gg} - A = B_{\gg} - B \geqslant 0$ and $\zeta_{\gg} = \zeta$. 

Suppose $(\q_{\gg}, >_{\q_{\gg}})$ dominates $(\q, >_{\q})$ with both orders satisfying condition $(\mathcal{P})$, and $\q_{\gg}$ has discrete diagonal restriction, we have 
\begin{align}
\label{eq: general representation GL}
\r_{\q} = \circ _{(\rho, A, B, \zeta) \in Jord(\q)} \Jac^{\theta}_{(\rho, A_{\gg}, B_{\gg}, \zeta) \mapsto (\rho, A, B, \zeta)} \r_{\q_{\gg}}
\end{align}
where the composition is taken in the decreasing order with respect to $>_{\q}$. Note if the condition ($\mathcal{P}$) is not satisfied, this may not be true. To describe the Jacquet functor in \eqref{eq: general representation GL}, we consider the following generalized segment:
\begin{align}
X^{\gg}_{(\rho, A, B, \zeta)} =
\label{eq: shift matrix}
  \begin{bmatrix}
   \zeta B_{\gg} & \cdots & \zeta (B + 1) \\
   \vdots &  & \vdots \\
   \zeta A_{\gg} & \cdots & \zeta (A+1)
  \end{bmatrix}.
\end{align}
Then the Jacquet functor in \eqref{eq: general representation GL} means applying $\Jac^{\theta}_{x}$ consecutively for $x$ ranges over $X^{\gg}_{(\rho, A, B, \zeta)}$ from top to bottom and from left to right.
Now it is clear how to define $\theta_{MW}(\q)$. We first choose an order $>_{\q}$ satisfying condition $(\mathcal{P})$, and then choose a dominating parameter $\q_{\gg}$ with discrete diagonal restriction and natural order. We define $\theta_{MW}(\q)$ to be the one induced from $\theta_{MW}(\q_{\gg})$ through the Jacquet module. The upshot is $\theta_{MW}(\q)$ only depends on the order $>_{\q}$, but not on the dominating parameter $\q_{\gg}$. This is explained in \cite{MW:2006}, and one can also see this when we derive the formula for $\theta_{MW}(\q) / \theta_{W}(\q)$.

Suppose $\q \in \cQ{G}$ and we fix an order $>_{\q}$ on $Jord(\q_{p})$ satisfying $(\mathcal{P})$, then we can define a set $\mathcal{Z}_{MW/W}(\q)$ of {\bf unordered pairs} of Jordan blocks from $Jord(\q_{p})$ as follows. 

\begin{definition}
\label{def: MW/W}
A pair $\{(\rho, a, b), (\rho', a', b') \in Jord(\q_{p})\}$ is contained in $\mathcal{Z}_{MW/W}(\q)$ if and only if $\rho = \rho'$, and it is in one of the following situations.

\begin{enumerate}

\item Case: $a, b$ are even and $a', b'$ are odd. 

   \begin{enumerate}
   
   \item If $\zeta_{a, b} = -1$ and 
             \(
            \begin{cases}
                      & \zeta_{a', b'} = -1 \Rightarrow (\rho, a, b) >_{\q} (\rho, a', b'), a > a'. \\
                      & \zeta_{a', b'} = +1 \Rightarrow a > a'.          
             \end{cases}
             \)
   \item If $\zeta_{a, b} = \zeta_{a', b'} = +1$ and \(
            \begin{cases}
                      & (\rho, a, b) >_{\q} (\rho, a', b') \Rightarrow a' > a, b > b'. \\
                      & (\rho, a, b) <_{\q} (\rho, a', b') \Rightarrow a > a', b > b'.          
             \end{cases}
             \)
                
   \end{enumerate}

\item Case : $a$ is odd, $b$ is even and $a'$ is even, $b'$ is odd. 

   \begin{enumerate}
   
   \item If $\zeta_{a, b} = -1$ and 
             \(
            \begin{cases}
                      & \zeta_{a', b'} = -1 \Rightarrow (\rho, a, b) >_{\q} (\rho, a', b'), a < a'. \\
                      & \zeta_{a', b'} = +1 \text{ and } \begin{cases}
                                                        (\rho, a, b) >_{\q} (\rho, a', b') \Rightarrow a < a'. \\        
                                                        (\rho, a, b) <_{\q} (\rho, a', b') \Rightarrow a > a'.
                                                      \end{cases}
             \end{cases}
             \)
   
   \item If $\zeta_{a, b} = \zeta_{a', b'} = +1$ and \(
            \begin{cases}
                      & (\rho, a, b) >_{\q} (\rho, a', b') \Rightarrow a < a', b > b'. \\
                      & (\rho, a, b) <_{\q} (\rho, a', b') \Rightarrow a > a', b > b'.          
             \end{cases}
             \)
                
   \end{enumerate}

\end{enumerate}

\end{definition}

\begin{theorem}
\label{thm: sign}
For $\q \in \cQ{G}$, $\theta_{MW}(\q)/\theta_{W}(\q) = (-1)^{|\mathcal{Z}_{MW/W}(\q)|}$.
\end{theorem}

\begin{proof}
By our definition it suffices to prove the theorem for $\q = \q_{p}$, so we will assume $\q = \q_{p}$ from now on. The proof we give here is incomplete for we will need to refer to (\cite{MW:2006}, Section 5) for several ingredients. First, we would like to assume this theorem for $\q$ having discrete diagonal restriction and natural order, and we refer interested readers to (\cite{MW:2006}, Theorem 5.6.1). Secondly, we need to use the ``unipotent normalization" $\theta_{U}(\q)$ introduced in (\cite{MW:2006}, Section 5), and we will recall two of its most important properties as follows. 

The first property of $\theta_{U}(\q)$ is parallel with a similar property for the Whittaker normalization $\theta_{W}(\q)$. Let $(\rho, A, B, \zeta) \in Jord(\q)$, and we get $\q_{\gg}$ simply by changing $(\rho, A, B, \zeta)$ to $(\rho, A_{\gg}, B_{\gg}, \zeta)$ with $A_{\gg} - A = B_{\gg} - B \geqslant 0$ and $\zeta_{\gg} = \zeta$. Suppose $\r_{\q} = \Jac^{\theta}_{(\rho, A_{\gg}, B_{\gg}, \zeta) \mapsto (\rho, A, B, \zeta)} \r_{\q_{\gg}}$ with an action $\theta(\q)$ induced from some $\theta(\q_{\gg})$. Then if $\zeta = -1$ and $\theta(\q_{\gg}) = \theta_{W}(\q_{\gg})$, then $\theta(\q) = \theta_{W}(\q)$; if $\zeta = +1$ and $\theta(\q_{\gg}) = \theta_{U}(\q_{\gg})$, then $\theta(\q) = \theta_{U}(\q)$ (see \cite{MW:2006}, Proposition 5.4.1).

To state the second property, let us define $\mathcal{Z}(\q)$ to be the set of {\bf unordered pairs} $\{(\rho, a, b), (\rho, a', b')\}$ in $Jord(\q_{p})$ such that $sup(b, b')$ and $sup(a, a')$ are both even, and $inf(b, b')$ and $inf(a, a')$ are both odd. Then we have $\theta_{W}(\q) / \theta_{U}(\q) = (-1)^{|\mathcal{Z}(\q)|}$ (see \cite{MW:2006}, Theorem 5.5.7).

Now we can start the proof. Let us index the Jordan blocks in $Jord(\q)$ according to the order $>_{\q}$, i.e., $(\rho_{i}, a_{i}, b_{i}) >_{\q} (\rho_{i-1}, a_{i-1}, b_{i-1})$. And we assume $Jord(\q) = \{(\rho_{i}, a_{i}, b_{i})\}_{i=1}^{l}$. Let $\q_{\gg}$ be a dominating parameter with discrete diagonal restriction and natural order. Then we can also obtain $\q^{k}$ from $\q_{\gg}$ by changing $(\rho_{i}, a_{\gg, i}, b_{\gg, i})$ to $(\rho_{i}, a_{i}, b_{i})$ for $1 \leqslant i \leqslant k$. In particular, we can set $\q^{0} = \q_{\gg}$. Let $\Jac^{k} := \Jac^{\theta}_{(\rho_{k}, a_{\gg, k}, b_{\gg, k}) \mapsto (\rho_{k}, a_{k}, b_{k})}$. Then we have the following sequence:
\[
\xymatrix{ \r_{\q_{\gg}} = \r_{\q^{0}} \ar[r]^{\quad \Jac^{1}} & \cdots \ar[r]^{\Jac^{k}} & \r_{\q^{k}} \ar[r]^{\Jac^{k+1}} & \cdots \ar[r]^{\Jac^{l} \quad} & \r_{\q^{l}} = \r_{\q}. }
\]
From the properties of $\theta_{W}(\q)$ and $\theta_{U}(\q)$ that we have recalled above, we can compute $\theta_{MW}(\q^{k}) / \theta_{W}(\q^{k})$. If $\zeta_{k} = -1$, we have $\theta_{MW}(\q^{k}) / \theta_{W}(\q^{k}) = \theta_{MW}(\q^{k-1}) / \theta_{W}(\q^{k-1})$. If $\zeta_{k} = +1$, we have
\begin{align*}
\theta_{MW}(\q^{k}) / \theta_{W}(\q^{k}) & = \theta_{MW}(\q^{k}) / \theta_{U}(\q^{k}) \cdot \theta_{U}(\q^{k}) / \theta_{W}(\q^{k}) = \theta_{MW}(\q^{k-1}) / \theta_{U}(\q^{k-1}) \cdot \theta_{U}(\q^{k}) / \theta_{W}(\q^{k}) \\
& = \theta_{MW}(\q^{k-1}) / \theta_{W}(\q^{k-1}) \cdot \theta_{W}(\q^{k-1}) / \theta_{U}(\q^{k-1}) \cdot \theta_{U}(\q^{k}) / \theta_{W}(\q^{k}) \\
& = \theta_{MW}(\q^{k-1}) / \theta_{W}(\q^{k-1}) \cdot (-1)^{|\mathcal{Z}(\q^{k-1})|} \cdot (-1)^{|\mathcal{Z}(\q^{k})|}.
\end{align*}
Moreover, let $\mathcal{Z}_{k}(\q^{k-1})$ (resp. $\mathcal{Z}_{k}(\q^{k})$) be the subset of pairs in $\mathcal{Z}(\q^{k-1})$ (resp. $\mathcal{Z}(\q^{k})$) containing $(\rho_{k}, a_{\gg, k}, b_{\gg, k})$ (resp. $(\rho_{k}, a_{k}, b_{k})$), then
\begin{align*}
\theta_{MW}(\q^{k}) / \theta_{W}(\q^{k}) & = \theta_{MW}(\q^{k-1}) / \theta_{W}(\q^{k-1}) \cdot (-1)^{|\mathcal{Z}_{k}(\q^{k-1})| + |\mathcal{Z}_{k}(\q^{k})|} \\
& = \theta_{MW}(\q^{k-1}) / \theta_{W}(\q^{k-1}) \cdot (-1)^{|(\mathcal{Z}_{k}(\q^{k-1}) \cup \mathcal{Z}_{k}(\q^{k})) \backslash (\mathcal{Z}_{k}(\q^{k-1}) \cap \mathcal{Z}_{k}(\q^{k}))|},
\end{align*}
where we identify $(\rho_{k}, a_{\gg, k}, b_{\gg, k})$ with $(\rho_{k}, a_{k}, b_{k})$ in taking the intersection and union. To simplify the formula above, let us denote by $\mathcal{Z}_{k}(\q^{k-1}, \q^{k})$ the set $(\mathcal{Z}_{k}(\q^{k-1}) \cup \mathcal{Z}_{k}(\q^{k})) \backslash (\mathcal{Z}_{k}(\q^{k-1}) \cap \mathcal{Z}_{k}(\q^{k}))$. 

The proof is given by induction on $k$. So let us assume the theorem is valid for $\theta_{MW}(\q_{k}) / \theta_{W}(\q_{k})$ with $0 \leqslant k \leqslant s$. Note when $k=0$, this is our assumption at the beginning. We need to prove the theorem for $k = s+1$. According to our formula, we need to divide into two cases with respect to the parity of $a_{s+1} + b_{s+1}$. Here we will only treat the case when $a_{s+1} + b_{s+1}$ is even, while the other case is similar. Let $\rho = \rho_{s+1}$. From our previous discussion, we have 
\[
\theta_{MW}(\q^{s+1}) / \theta_{W}(\q^{s+1}) = \begin{cases}
                                                                   \theta_{MW}(\q^{s}) / \theta_{W}(\q^{s}), \text{ if $\zeta_{s+1} = -1$, } \\
                                                                   \theta_{MW}(\q^{s}) / \theta_{W}(\q^{s}) \cdot (-1)^{|\mathcal{Z}_{s+1}(\q^{s}, \q^{s+1})|}, \text{ if $\zeta_{s+1} = +1$.}
                                                                   \end{cases}
\]

We first consider the case when $\zeta_{s+1} = -1$. Suppose $\{(\rho, a_{\gg, s+1}, b_{\gg, s+1}), (\rho, a, b)\}$ belongs to $\mathcal{Z}_{MW/W}(\q^{s})$, then by our definition we are in one of the following situations.

\begin{enumerate}

\item If $(\rho, a_{\gg, s+1}, b_{\gg, s+1}) >_{\q^{s}} (\rho, a, b), 
         \begin{cases}
         a_{\gg, s+1} \text{ even; } a, b \text{ odd } \Rightarrow a_{\gg, s+1} > a. \\
         a_{\gg, s+1} \text{ odd; } a, b \text{ even } \Rightarrow \text{ impossible. }                                                                                       
         \end{cases}$

\item If $(\rho, a_{\gg, s+1}, b_{\gg, s+1}) <_{\q^{s}} (\rho, a, b), 
         \begin{cases}
         a_{\gg, s+1} \text{ even; } a, b \text{ odd } \Rightarrow a_{\gg, s+1} > a, \zeta_{a, b} = +1. \\
         a_{\gg, s+1} \text{ odd; } a, b \text{ even } \Rightarrow a_{\gg, s+1} < a, \zeta_{a, b} = -1.                                                                                     
         \end{cases}$                                                                                     

\end{enumerate}
Note $a_{\gg, s+1} = a_{s+1}$, so in all the situations we have $\{(\rho, a_{s+1}, b_{s+1}), (\rho, a, b)\}$ belonging to $\mathcal{Z}_{MW/W}(\q^{s+1})$ as well. In the same way, one can show 
\[
\{(\rho, a_{s+1}, b_{s+1}), (\rho, a, b)\} \in \mathcal{Z}_{MW/W}(\q^{s+1}) \Rightarrow \{(\rho, a_{\gg, s+1}, b_{\gg, s+1}), (\rho, a, b)\} \in \mathcal{Z}_{MW/W}(\q^{s}).
\]
This means our formula is valid for $k = s+1$ in this case.

Next we come to the more difficult case $\zeta_{s+1} = +1$. Similarly, we first suppose $\{(\rho, a_{\gg, s+1}, b_{\gg, s+1}), (\rho, a, b)\}$ belongs to $\mathcal{Z}_{MW/W}(\q^{s})$, and we will be in one of the following situations.

\begin{enumerate}

\item If $(\rho, a_{\gg, s+1}, b_{\gg, s+1}) >_{\q^{s}} (\rho, a, b), 
         \begin{cases}
         a_{\gg, s+1} \text{ even; } a, b \text{ odd } \Rightarrow a_{\gg, s+1} < a,  b_{\gg, s+1} > b. \\
         a_{\gg, s+1} \text{ odd; } a, b \text{ even } \Rightarrow \begin{cases}
                                                                                                                    a_{\gg, s+1} < a, \zeta_{a, b} = -1. \\
                                                                                                                    a_{\gg, s+1} < a, b_{\gg, s+1} < b, \zeta_{a, b} = +1.                                                                                                                  
                                                                                                                    \end{cases}
         \end{cases}$

\item If $(\rho, a_{\gg, s+1}, b_{\gg, s+1}) <_{\q^{s}} (\rho, a, b), 
         \begin{cases}
         a_{\gg, s+1} \text{ even; } a, b \text{ odd } \Rightarrow a_{\gg, s+1} > a, b_{\gg, s+1} > b, \zeta_{a, b} = +1. \,\,\,\, (*-1)\\
         a_{\gg, s+1} \text{ odd; } a, b \text{ even } \Rightarrow \begin{cases}
                                                                                                                    a_{\gg, s+1} < a, \zeta_{a, b} = -1. \\
                                                                                                                    a_{\gg, s+1} > a, b_{\gg, s+1} < b, \zeta_{a, b} = +1. (*-2)                                                                                                                   
                                                                                                                    \end{cases}                                                                                     
         \end{cases}$                                                                                     

\end{enumerate}
Note $a_{s+1} < a_{\gg, s+1}$ and $b_{s+1} = b_{\gg, s+1}$, so $\{(\rho, a_{s+1}, b_{s+1}), (\rho, a, b)\} \in \mathcal{Z}_{MW/W}(\q^{s+1})$ in all the situations except for $(*-1)$ and $(*-2)$ with the additional condition $a_{s+1} < a$. It is easy to check in the exceptional cases, either $\{(\rho, a_{\gg, s+1}, b_{\gg, s+1}), (\rho, a, b)\}$ or $\{(\rho, a_{s+1}, b_{s+1}), (\rho, a, b)\}$ belongs to $\mathcal{Z}_{s+1}(\q^{s}, \q^{s+1})$.

Conversely, if we suppose $\{(\rho, a_{s+1}, b_{s+1}), (\rho, a, b)\}$ belongs to $\mathcal{Z}_{MW/W}(\q^{s+1})$, then we will be in one of the following situations.

\begin{enumerate}

\item If $(\rho, a_{s+1}, b_{s+1}) >_{\q^{s+1}} (\rho, a, b), 
         \begin{cases}
         a_{s+1} \text{ even; } a, b \text{ odd } \Rightarrow a_{s+1} < a, b_{s+1} > b. \quad \quad \quad \quad \quad \quad \, (*-3) \\
         a_{s+1} \text{ odd; } a, b \text{ even } \Rightarrow \begin{cases}
                                                                                                                    a_{s+1} < a, \zeta_{a, b} = -1. \quad \quad \quad \quad \quad (*-4) \\
                                                                                                                    a_{s+1} < a, b_{s+1} < b, \zeta_{a, b} = +1. \quad (*-5)                                                                                                                   
                                                                                                                    \end{cases}
         \end{cases}$

\item If $(\rho, a_{s+1}, b_{s+1}) <_{\q^{s+1}} (\rho, a, b), 
         \begin{cases}
         a_{s+1} \text{ even; } a, b \text{ odd } \Rightarrow a_{s+1} > a, b_{s+1} > b, \zeta_{a, b} = +1. \\
         a_{s+1} \text{ odd; } a, b \text{ even } \Rightarrow \begin{cases}
                                                                                                                    a_{s+1} < a, \zeta_{a, b} = -1. \quad \quad \quad \quad \quad (*-6)\\
                                                                                                                    a_{s+1} > a, b_{s+1} < b, \zeta_{a, b} = +1.                                                                                                                    
                                                                                                                    \end{cases}                                                                                     
         \end{cases}$                                                                                     

\end{enumerate}
We find $\{(\rho, a_{\gg, s+1}, b_{\gg, s+1}), (\rho, a, b)\} \notin \mathcal{Z}_{MW/W}(\q^{s+1})$ only for $(*-3), (*-4), (*-5), (*-6)$ with the additional condition $a_{\gg, s+1} > a$. Again, it is easy to check in these cases, either $\{(\rho, a_{\gg, s+1}, b_{\gg, s+1}), (\rho, a, b)\}$ or $\{(\rho, a_{s+1}, b_{s+1}), (\rho, a, b)\}$ belongs to $\mathcal{Z}_{s+1}(\q^{s}, \q^{s+1})$.

Finally, it suffices to figure out the set $\mathcal{Z}_{s+1}(\q^{s}, \q^{s+1})$, and show it consists of exactly those pairs that we have encountered in $(*-1)$-$(*-6)$ with their additional conditions respectively. So let us suppose either $\{(\rho, a_{\gg, s+1}, b_{\gg, s+1}), (\rho, a, b)\}$ or $\{(\rho, a_{s+1}, b_{s+1}), (\rho, a, b)\}$ belongs to $\mathcal{Z}_{s+1}(\q^{s}, \q^{s+1})$, and we list all the possibilities.

\begin{enumerate}

\item If $(\rho, a_{\gg, s+1}, b_{\gg, s+1}) >_{\q^{s}} (\rho, a, b), \\
        \begin{cases}
        a_{\gg, s+1} \text{ even; } a, b \text{ odd } \Rightarrow a_{\gg, s+1} > a > a_{s+1}, b_{\gg, s+1} > b.  \text{\quad \quad \quad \quad \quad \quad \, $(*-3)$ with $a_{\gg, s+1} > a$ } \\
        a_{\gg, s+1} \text{ odd; } a, b \text{ even } \Rightarrow \begin{cases}
                                                                                                                   a_{\gg, s+1} > a > a_{s+1}, b_{\gg, s+1} < b, \zeta_{a, b} = +1. \text{ \quad $(*-5)$ with $a_{\gg, s+1} > a$ } \\                                                                                                                  
                                                                                                                   a_{\gg, s+1} > a > a_{s+1}, b_{\gg, s+1} < b, \zeta_{a, b} = -1. \text{ \quad $(*-4)$ with $a_{\gg, s+1} > a$ }                                                                                                                                                                                                                                      
                                                                                                                   \end{cases}       
        \end{cases}$
          
\item If $(\rho, a_{\gg, s+1}, b_{\gg, s+1}) <_{\q^{s}} (\rho, a, b), \\
        \begin{cases}
        a_{\gg, s+1} \text{ even; } a, b \text{ odd } \Rightarrow a_{\gg, s+1} > a > a_{s+1},  b_{\gg, s+1} > b.  \text{ \quad \quad \quad \quad \quad \quad $\, (*-1)$ with $a_{s+1} < a$ } \\
        a_{\gg, s+1} \text{ odd; } a, b \text{ even } \Rightarrow \begin{cases}
                                                                                                                   a_{\gg, s+1} > a > a_{s+1}, b_{\gg, s+1} < b, \zeta_{a, b} = +1. \text{ \quad $(*-2)$ with $a_{s+1} < a$ } \\
                                                                                                                   a_{\gg, s+1} > a > a_{s+1}, b_{\gg, s+1} < b, \zeta_{a, b} = -1. \text{ \quad $(*-6)$ with $a_{\gg, s+1} > a$ }                                                                                                                   
                                                                                                                   \end{cases}       
        \end{cases}$

\end{enumerate}
Note each case here corresponds exactly to one of $(*-1)$-$(*-6)$ with the required additional conditions, as we indicate on their right. This finishes the proof.

\end{proof}

\begin{remark}
There is a slight difference between our definition of $\mathcal{Z}_{MW/W}(\q)$ (also $\mathcal{Z}(\q)$) and that in \cite{MW:2006}, namely they use ordered pairs rather than unordered pairs. Moreover, this theorem slightly generalizes the formula in \cite{MW:2006} in the sense that we only require $>_{\q}$ satisfies $(\mathcal{P})$.
\end{remark}

We would also like to see the effect of M{\oe}glin-Waldspurger's normalization on the parametrizations of representations inside Arthur packets. To do so, we need the following definition.

\begin{definition}
\label{def: MW/W}
For $\q \in \cQ{G}$ and $(\rho, a, b) \in Jord(\q_{p})$, 
\(
\mathcal{Z}_{MW/W}(\q)_{(\rho, a, b)} := \{(\rho', a', b') \in Jord(\q_{p}) : \text{ the pair of } (\rho, a, b) \text { and } (\rho', a', b') \text { lies in } \mathcal{Z}_{MW/W}(\q)\},
\)
and
\(
\e^{MW/W}_{\q}(\rho, a, b) := (-1)^{|\mathcal{Z}_{MW/W}(\q)_{(\rho, a, b)}|}.
\)
\end{definition}




\begin{proposition}
\label{prop: MW/W discrete}

Suppose $\q \in \cQ{G}$ has discrete diagonal restriction.


\begin{enumerate}

\item $\e^{MW/W}_{\q} \in \D{\S{\q}^{\Sigma_{0}}}$ and $\e^{MW/W}_{\q}(s_{\q}) = \theta_{MW}(\q) / \theta_{W}(\q)$.

\item If we write $\r_{MW}(\q, \bar{\e}) := \r_{W}(\q, \bar{\e} \bar{\e}^{MW/W}_{\q})$ for $\bar{\e} \in \D{\S{\q}}$, then the character identities in Theorem~\ref{thm: Arthur packet} can be rewritten as follows.
        
        \begin{enumerate}
        
        \item Let 
        \[
        f_{MW}(\q) := \sum_{\bar{\e} \in \D{\S{\q}}} \bar{\e}(s_{\q})f_{G}(\r_{MW}(\q, \bar{\e})), \,\,\,\,\,\,\,\,\, f \in \sH(G).
        \]
        Then 
        \begin{align}
        \label{eq: MW character relation GL(N)}
        f^{G}_{MW}(\q) = f_{N^{\theta}, MW}(\r_{\q}), \,\,\,\,\,\,\,\,\,\,\,\,\,\,   f \in C^{\infty}_{c}(GL(N)).
        \end{align}
        
        \item  If $s \in \S{\q}$ and $(H, \q_{H}) \rightarrow (\q, s)$, 
        then we can define a stable distribution $f_{MW}(\q_{H})$ on $H$ as in (a), and the
        following identity holds
        \begin{align}
        \label{eq: MW character relation}
        f^{H}_{MW}(\q_{H}) = \sum_{\bar{\e} \in \D{\S{\q}}} \bar{\e}(ss_{\q})f_{G}(\r_{MW}(\q, \bar{\e})), \,\,\,\,\,\,\,\,\,\,\, f \in \sH(G).
        \end{align}
       
        \end{enumerate}

\end{enumerate}

\end{proposition}

\begin{proof}
For part (1), we have
\begin{align*}
\prod_{(\rho, a, b) \in Jord(\q)} \e_{\q}^{MW/W}(\rho, a, b) & = \prod_{(\rho, a, b) \in Jord(\q)} (-1)^{|\mathcal{Z}_{MW/W}(\q)_{(\rho, a, b)}|} \\
& = (-1)^{\sum_{(\rho, a, b) \in Jord(\q)}|\mathcal{Z}_{MW/W}(\q)_{(\rho, a, b)}|} = (-1)^{2|\mathcal{Z}_{MW/W}(\q)|} = 1,
\end{align*}
and hence $\e_{\q}^{MW/W}$ defines a character of $\S{\q}^{\Sigma_{0}}$. To compute $\e_{\q}^{MW/W}(s_{\q})$, let us recall
\[
s_{\q}(\rho, a, b) = \begin{cases}
                              -1, & \text{ if $b$ is even,} \\
                               1, & \text{ if $b$ is odd,} 
                              \end{cases}
\]
for $(\rho, a, b) \in Jord(\q)$. Then
\begin{align*}
\e_{\q}^{MW/W}(s_{\q}) &= \prod_{\substack{(\rho, a, b) \in Jord(\q) \\ b \text{ is even}}} \e_{\q}^{MW/W}(\rho, a, b) = (-1)^{\sum_{\substack{(\rho, a, b) \in Jord(\q) \\ b \text{ is even}}} |\mathcal{Z}_{MW/W}(\q)_{(\rho, a, b)}|} \\
&= (-1)^{|\mathcal{Z}_{MW/W}(\q)|} = \theta_{MW}(\q) / \theta_{W}(\q).
\end{align*}

Now we consider part (2). First by definition we have for $f \in \sH(G)$
\begin{align*}
f_{MW}(\q) & = \sum_{\bar{\e} \in \D{\S{\q}}} \bar{\e}(s_{\q})f_{G}(\r_{MW}(\q, \bar{\e})) = \sum_{\bar{\e} \in \D{\S{\q}}} \bar{\e}(s_{\q})f_{G}(\r_{W}(\q, \bar{\e} \bar{\e}^{MW/W}_{\q})) \\
& = \sum_{\bar{\e} \in \D{\S{\q}}} \bar{\e} \bar{\e}^{MW/W}_{\q}(s_{\q}) f_{G}(\r_{W}(\q, \bar{\e})) = \sum_{\bar{\e} \in \D{\S{\q}}} \bar{\e}(s_{\q}) \bar{\e}^{MW/W}_{\q}(s_{\q}) f_{G}(\r_{W}(\q, \bar{\e})) \\
& = \bar{\e}^{MW/W}_{\q}(s_{\q}) \sum_{\bar{\e} \in \D{\S{\q}}} \bar{\e}(s_{\q}) f_{G}(\r_{W}(\q, \bar{\e})) = \bar{\e}^{MW/W}_{\q}(s_{\q}) f_{W}(\q).
\end{align*}
Combined with part (1) and \eqref{eq: nontempered character relation GL(N)}, we then get
\[
f^{G}_{MW}(\q) = \theta_{MW}(\q) / \theta_{W}(\q) f_{N^{\theta}, W}(\r_{\q}) = f_{N^{\theta}, MW}(\r_{\q})
\]
for $f \in C^{\infty}_{c}(GL(N))$. Next, for any $s \in \S{\q}$ and $(H, \q_{H}) \rightarrow (\q, s)$, let $\q_{H} = \q_{I} \times \q_{II}$ (see Example~\ref{eg: endoscopy}). Then by \eqref{eq: nontempered character relation} we have
\[
f^{H}_{W}(\q_{H}) = \sum_{\bar{\e} \in \D{\S{\q}}} \bar{\e}(s s_{\q})f(\r_{W}(\q, \bar{\e})).
\]
Also note the right hand side of \eqref{eq: MW character relation} is
\begin{align*}
\text{RHS} & = \sum_{\bar{\e} \in \D{\S{\q}}} \bar{\e}(s s_{\q})f(\r_{W}(\q, \bar{\e} \bar{\e}^{MW/W}_{\q})) \\
& = \sum_{\bar{\e} \in \D{\S{\q}}} \bar{\e} \bar{\e}^{MW/W}_{\q}(s s_{\q})f(\r_{W}(\q, \bar{\e})) \\
& = \bar{\e}^{MW/W}_{\q}(s s_{\q}) \sum_{\bar{\e} \in \D{\S{\q}}} \bar{\e}(s s_{\q})f(\r_{W}(\q, \bar{\e})),
\end{align*}
and the left hand side of \eqref{eq: MW character relation} is
\[
\text{LHS} = \bar{\e}^{MW/W}_{\q_{H}}(s_{\q_{H}})f^{H}_{W}(\q_{H}),
\]
where $s_{\q_{H}} = s_{\q_{I}} \times s_{\q_{II}}$ and $\e^{MW/W}_{\q_{H}} = \e^{MW/W}_{\q_{I}} \otimes \e^{MW/W}_{\q_{II}}$. So it suffices to show 
\[
\bar{\e}^{MW/W}_{\q_{H}}(s_{\q_{H}}) = \bar{\e}^{MW/W}_{\q}(s s_{\q}).
\]
Moreover, by using part (1) this equality can be reduced to 
\begin{align}
\label{eq: sign}
\bar{\e}^{MW/W}_{\q}(s) = \theta_{MW}(\q_{H})/\theta_{W}(\q_{H}) \cdot \theta_{MW}(\q) / \theta_{W}(\q),
\end{align}
where 
\[
\theta_{MW}(\q_{H})/\theta_{W}(\q_{H}) = \theta_{MW}(\q_{I})/\theta_{W}(\q_{I}) \cdot \theta_{MW}(\q_{II})/\theta_{W}(\q_{II}).
\]
To show \eqref{eq: sign}, one considers the partition $Jord(\q) = Jord_{+} \sqcup Jord_{-}$ (see Example~\ref{eg: endoscopy}). Then 
\[
\bar{\e}^{MW/W}_{\q}(s) = (-1)^{m},
\]
where
\[
m = \sharp \Big\{ \{(\rho, a, b), (\rho', a', b') \} \in \mathcal{Z}_{MW/W}(\q): (\rho, a, b) \in Jord_{+}, (\rho', a', b') \in Jord_{-}  \Big\}.
\]
By Theorem~\ref{thm: sign}, we can write the other side of \eqref{eq: sign} as
\(
(-1)^{|\mathcal{Z}_{MW/W}(\q)| - |\mathcal{Z}_{MW/W}(\q_{I})| - |\mathcal{Z}_{MW/W}(\q_{II})|},
\) 
and hence the validity of \eqref{eq: sign} is clear.

\end{proof}

For $\q = \q_{p} \in \cQ{G}$, we fix an order $>_{\q}$ on $Jord(\q)$ satisfying condition $(\mathcal{P})$. We also choose $\q_{\gg}$ dominating $\q$ with discrete diagonal restriction and natural order. We identify $\S{\q_{\gg}}$ with $\S{\q^{>}}$ and then $s^{>}_{\q} = s_{\q_{\gg}}$.
For $\bar{\e} \in \D{\S{\q^{>}}}$, we define
\begin{align}
\label{eq: MW general}
\r_{MW}(\q, \bar{\e}):= \circ _{(\rho, A, B, \zeta) \in Jord(\q)} \bar{\Jac}_{(\rho, A_{\gg}, B_{\gg}, \zeta) \mapsto (\rho, A, B, \zeta)} \r_{MW}(\q_{\gg}, \bar{\e}),
\end{align}
where the Jacquet functor is defined as in \eqref{eq: shift matrix} and the composition is taken in the decreasing order. For these $\sH(G)$-modules, we have the following proposition.

\begin{proposition}
\label{prop: MW/W}
Suppose $\q = \q_{p} \in \cQ{G}$, and $>_{\q}$ is an order on $Jord(\q)$ satisfying condition $(\mathcal{P})$. Suppose $\q_{\gg}$ has discrete diagonal restriction and dominates $\q$. Then 
\begin{enumerate}

\item $\e^{MW/W}_{\q} \in \D{\S{\q^{>}}^{\Sigma_{0}}}$ and $\e^{MW/W}_{\q}(s^{>}_{\q}) = \theta_{MW}(\q) / \theta_{W}(\q)$.

\item For $\bar{\e} \in \D{\S{\q^{>}}}$,
\[
\r_{MW}(\q, \bar{\e}) = \begin{cases}
                           \r_{W}(\q, \bar{\e} \bar{\e}^{MW/W}_{\q}), & \text{ if $\bar{\e} \bar{\e}^{MW/W}_{\q} \in \D{\S{\q}}$,} \\
                            0, & \text{ otherwise.}
                          \end{cases}
\]

\end{enumerate}
\end{proposition}

\begin{proof}
The proof of part (1) is the same as that in Proposition~\ref{prop: MW/W discrete}. So we will only show part (2) here. For $s \in \S{\q^{>}}$, we denote its image in $\S{\q}$ again by $s$. Let
\begin{align*}
\cPkt{MW, s}(\q_{\gg}) & = \sum_{\bar{\e} \in \D{\S{\q_{\gg}}}} \bar{\e}(ss_{\q_{\gg}}) \r_{MW}(\q_{\gg}, \bar{\e}), \\
\cPkt{W, s}(\q) & = \sum_{\bar{\e} \in \D{\S{\q}}} \bar{\e}(ss_{\q}) \r_{W}(\q, \bar{\e}).
\end{align*}
It follows for $\bar{\e} \in \D{\S{\q^{>}}}$,
\[
\r_{MW}(\q_{\gg}, \bar{\e}) = \frac{\bar{\e}(s_{\q_{\gg}})}{|\S{\q_{\gg}}|} \sum_{s \in \S{\q_{\gg}}} \bar{\e}(s) \cPkt{MW, s}(\q_{\gg}).
\]
Suppose $(H_{\gg}, \q_{H_{\gg}}) \rightarrow (\q_{\gg}, s)$ and $(H, \q_{H}) \rightarrow (\q, s)$, then $\q_{H_{\gg}}$ dominates $\q_{H}$. By \eqref{eq: nontempered character relation} and \eqref{eq: MW character relation} we have
\begin{align}
\label{eq: MW/W}
\circ _{(\rho, A, B, \zeta) \in Jord(\q)} \bar{\Jac}_{(\rho, A_{\gg}, B_{\gg}, \zeta) \mapsto (\rho, A, B, \zeta)} \cPkt{MW, s}(\q_{\gg}) = \theta_{MW}(\q_{H})/\theta_{W}(\q_{H}) \cPkt{W, s}(\q).
\end{align}
Analogous to \eqref{eq: sign}, one can show
\[
\theta_{MW}(\q_{H})/\theta_{W}(\q_{H}) = \bar{\e}^{MW/W}_{\q}(ss^{>}_{\q}).
\]
Therefore
\[
\r_{MW}(\q, \bar{\e}) = \frac{\bar{\e}(s^{>}_{\q})}{|\S{\q^{>}}|} \sum_{s \in \S{\q^{>}}} \bar{\e}(s) \bar{\e}^{MW/W}_{\q}(ss^{>}_{\q}) \cPkt{W, s}(\q).
\]
We rewrite it as 
\begin{align*}
\r_{MW}(\q, \bar{\e} \bar{\e}^{MW/W}_{\q}) & = \frac{\bar{\e} \bar{\e}^{MW/W}_{\q}(s^{>}_{\q})}{|\S{\q^{>}}|} \sum_{s \in \S{\q^{>}}} \bar{\e} \bar{\e}^{MW/W}_{\q}(s) \bar{\e}^{MW/W}_{\q}(ss^{>}_{\q}) \cPkt{W, s}(\q) \\
& = \frac{\bar{\e}(s^{>}_{\q})}{|\S{\q^{>}}|} \sum_{s \in \S{\q^{>}}} \bar{\e}(s) \cPkt{W, s}(\q).
\end{align*}
Note $\cPkt{W, s}(\q)$ only depends on the image of $s$ in $\S{\q}$, so 
\[
\sum_{s \in \S{\q^{>}}} \bar{\e}(s) \cPkt{W, s}(\q) = \begin{cases}
                                                                        \frac{|\S{\q^{>}}|}{|\S{\q}|} \sum_{s \in \S{\q}} \bar{\e}(s) \cPkt{W, s}(\q),  & \text{ if $\bar{\e} \in \D{\S{\q}}$,} \\
                                                                        0,  & \text{ otherwise.}
                                                                        \end{cases}
\]
If $\bar{\e} \in \D{\S{\q}}$, then $\bar{\e}(s_{\q}) = \bar{\e}(s^{>}_{\q})$, and it follows that
\[
\r_{MW}(\q, \bar{\e} \bar{\e}^{MW/W}_{\q}) =  \frac{\bar{\e}(s_{\q})}{|\S{\q}|} \sum_{s \in \S{\q}} \bar{\e}(s) \cPkt{W, s}(\q) = \r_{W}(\q, \bar{\e}).
\]
If $\bar{\e} \notin \D{\S{\q}}$, $\r_{MW}(\q, \bar{\e} \bar{\e}^{MW/W}_{\q}) = 0$. This finishes the proof.

\end{proof}

In general, for $\q \in \cQ{G}$, we define 
\[
\r_{MW}(\q, \bar{\e}) = \r_{\q_{np}} \rtimes \r_{MW}(\q_{p}, \bar{\e}),
\]
for $\bar{\e} \in \D{\S{\q^{>}}}$. Since 
\[
\r_{W}(\q, \bar{\e}) =  \r_{\q_{np}} \rtimes \r_{W}(\q_{p}, \bar{\e}) 
\]
for $\bar{\e} \in \D{\S{\q}}$, we again have 
\[
\r_{MW}(\q, \bar{\e}) = \begin{cases}
                           \r_{W}(\q, \bar{\e} \bar{\e}^{MW/W}_{\q}), & \text{ if $\bar{\e} \bar{\e}^{MW/W}_{\q} \in \D{\S{\q}}$,} \\
                            0, & \text{ otherwise.}
                          \end{cases}
\]

The main purpose of introducing M{\oe}glin-Waldspurger's normalization is that one will have a recursive formula for $f_{N^{\theta}, MW}(\r_{\q})$ with $\q \in \cQ{G}$ having discrete diagonal restriction. Here we will occasionally write $\r(\q)$ for $\r_{\q}$. To introduce the formula, let us fix $(\rho, a, b) \in Jord(\q)$ such that $inf(a, b) > 1$. Recall we also put $A = (a+b)/2 - 1$, $B = |a-b|/2$, and $\zeta = \zeta_{a, b} = \text{Sign}(a - b)$ if $a \neq b$ and arbitrary otherwise. Then it is the same to require $A \neq B$ for the fixed Jordan block. Let $\q'$ be obtained from $\q$ by removing $(\rho, a, b)$. Then we can define an element in the Grothendieck group of representations of $GL(N)$ as follows.
\begin{align*}
\r(\q)_{(\rho, A, B, \zeta)} := & \+_{C \in ]B, A]} (-1)^{A-C} \begin{pmatrix}<\zeta B, \cdots, -\zeta C> \times \Jac^{\theta}_{\zeta (B+2), \cdots, \zeta C} \r(\q', (\rho, A, B+2, \zeta)) \times <\zeta C, \cdots, -\zeta B>\end{pmatrix} \\
& \+ (-1)^{[(A-B+1)/2]} \r(\q', (\rho, A, B+1, \zeta), (\rho, B, B, \zeta)).
\end{align*} 
We impose the normalized actions of M{\oe}glin-Waldspurger on $\r(\q', (\rho, A, B+2, \zeta))$ and $\r(\q', (\rho, A, B+1, \zeta), (\rho, B, B, \zeta))$, and we denote the resulting action on $\r(\q)_{(\rho, A, B, \zeta)}$ by $\theta_{MW}(\q)_{(\rho, A, B, \zeta)}$. The next theorem shows the relation between $\r(\q)_{(\rho, A, B, \zeta)}$ and $\r(\q)$.

\begin{theorem}
\label{thm: MW formula}
Suppose $\q \in \cQ{G}$ has discrete diagonal restriction, then 
\[
f_{N^{\theta}, MW}(\r_{\q}) = f_{N^{\theta}, MW}(\r(\q)_{(\rho, A, B, \zeta)}).
\]
\end{theorem}

The proof of this theorem (see \cite{MW:2006}) involves some complicated computations of Jacquet modules, and it is fair to say that M{\oe}glin-Waldspurger's normalization is somehow artificially made for this theorem. This theorem has an immediate consequence on the Arthur packets for $G$.

For $\q \in \cQ{G}$ having discrete diagonal restriction, we write
\begin{align}
\label{eq: stable distribution MW}
\cPkt{MW}(\q) := \sum_{\bar{\e} \in \D{\S{\q}}} \bar{\e}(s_{\q})\r_{MW}(\q, \bar{\e}).
\end{align}
Then we have the following proposition.

\begin{proposition}
\label{prop: recursive formula for packet}
Suppose $\q \in \cQ{G}$ has discrete diagonal restriction and we fix $(\rho, A, B, \zeta) \in Jord(\q)$ such that $A > B$, then 
\begin{align*}
\cPkt{MW}(\q) = & \+_{C \in ]B, A]} (-1)^{A-C} <\zeta B, \cdots, -\zeta C> \rtimes \bar{\Jac}_{\zeta (B+2), \cdots, \zeta C} \cPkt{MW}(\q', (\rho, A, B+2, \zeta))  \\
& \+ (-1)^{[(A-B+1)/2]} \cPkt{MW}(\q', (\rho, A, B+1, \zeta), (\rho, B, B, \zeta)), 
\end{align*}
where $\q'$ is obtained from $\q$ by removing $(\rho, A, B, \zeta)$.
\end{proposition}

\begin{proof}
This proposition follows easily from Theorem~\ref{thm: MW formula} and the twisted character relation \eqref{eq: MW character relation GL(N)}, together with the compatibility of the twisted endoscopic transfer with parabolic inductions and Jacquet modules (see \cite{Xu:preprint3}, Section 6).
\end{proof}

From this formula, one can see the case of parameters with discrete diagonal restriction can be reduced to the case of elementary parameters. Later on, we will give a recursive formula of M{\oe}glin for $\r_{MW}(\q, \e)$, or more precisely for $\r_{M}(\q, \e)$ (see Section~\ref{sec: discrete diagonal restriction} for its definition), in the case of discrete diagonal restriction again, which is clearly motivated by the formula here. But in order to give M{\oe}glin's formula, we need to first study the Arthur packets for elementary parameters.

\section{Elementary Arthur packet}
\label{sec: elementary Arthur packet}

Let us recall $\q \in \cQ{G}$ is elementary if $\q \circ \Delta \in \cPdt{G}$ and $A = B$ for all $(\rho, A, B, \zeta) \in Jord(\q)$. And we have the following theorem about elementary Arthur packets due to M{\oe}glin \cite{Moeglin:2006}.

\begin{theorem}[M{\oe}glin]
\label{thm: elementary Arthur packet}
Suppose $\q \in \cQ{G}$ is elementary, then $\r_{W}(\q, \bar{\e})$ is always nonzero and irreducible. Moreover, $\r_{W}(\q, \bar{\e}) \neq \r_{W}(\q, \bar{\e}')$ if $\bar{\e} \neq \bar{\e}'$.
\end{theorem}

The main difficulty of this theorem remains at proving certain generalized Aubert involution (see Section~\ref{subsec: Aubert dual for G}) would take irreducible representations viewed as $\sH(G)$-modules in elementary Arthur packets to irreducible representations viewed as $\sH(G)$-modules up to a sign in the corresponding Grothendieck group. But this does not admit a direct approach. So instead, we will follow \cite{Moeglin:2006} to construct systematically a class of representations which generalizes the construction of discrete series representations of M{\oe}glin and Tadi{\'c} (see \cite{MoeglinTadic:2002} and also \cite{Xu:preprint3}, Section 10). This class of representations will form the candidates for elements in the elementary Arthur packets. In fact, what M{\oe}glin constructed are representations of $G^{\Sigma_{0}}$, but we can then take the irreducible representations of $G$ viewed as $\sH(G)$-modules defined by their restriction to $G$. The point is it is easier to show the generalized Aubert involution preserve this class of representations of $G^{\Sigma_{0}}$ and also their irreducibility. In the end, we are going to show the corresponding $\sH(G)$-modules are really elements in the elementary Arthur packets. 

First we need to define parabolic induction and Jacquet module on the category $\Rep(G^{\Sigma_{0}})$ of finite-length smooth representations of $G^{\Sigma_{0}}$. Let $P = MN$ be a standard parabolic subgroup of $G$. If $M$ is $\theta_{0}$-stable, we write $M^{\Sigma_{0}} := M \rtimes \Sigma_{0}$. Otherwise, we let $M^{\Sigma_{0}} = M$.
Suppose $\sigma^{\Sigma_{0}} \in \Rep(M^{\Sigma_{0}})$, $\r^{\Sigma_{0}} \in \Rep(G^{\Sigma_{0}})$. 

\begin{enumerate}

\item If $M^{\theta_{0}} = M$, we define the normalized parabolic induction $\Ind^{G^{\Sigma_{0}}}_{P^{\Sigma_{0}}} \sigma^{\Sigma_{0}}$ to be the extension of the representation $\Ind^{G}_{P}(\sigma^{\Sigma_{0}}|_{M})$ by an induced action of $\Sigma_{0}$, and we define the normalized Jacquet module $\Jac_{P^{\Sigma_{0}}} \r^{\Sigma_{0}}$ to be the extension of the representation $\Jac_{P}(\r^{\Sigma_{0}}|_{G})$ by an induced action of $\Sigma_{0}$.

\item If $M^{\theta_{0}} \neq M$, we define the normalized parabolic induction $\Ind^{G^{\Sigma_{0}}}_{P^{\Sigma_{0}}} \sigma^{\Sigma_{0}}$ to be $\Ind^{G^{\Sigma_{0}}}_{G} \Ind^{G}_{P}(\sigma^{\Sigma_{0}}|_{M})$, and we define the normalized Jacquet module $\Jac_{P^{\Sigma_{0}}} \r^{\Sigma_{0}}$ to be $\Jac_{P}(\r^{\Sigma_{0}}|_{G})$.

\end{enumerate}
It follows from the definition that
\[
(\Jac_{P^{\Sigma_{0}}} \r^{\Sigma_{0}})|_{M} = \Jac_{P}(\r^{\Sigma_{0}}|_{G}).
\]
And 
\[
(\Ind^{G^{\Sigma_{0}}}_{P^{\Sigma_{0}}} \sigma^{\Sigma_{0}})|_{G} = \Ind^{G}_{P}(\sigma^{\Sigma_{0}}|_{M}),
\] 
unless $G$ is special even orthogonal and $M^{\Sigma_{0}} = M$, in which case 
\[
(\Ind^{G^{\Sigma_{0}}}_{P^{\Sigma_{0}}} \sigma^{\Sigma_{0}})|_{G} = \Ind^{G}_{P}(\sigma^{\Sigma_{0}}|_{M}) \+ (\Ind^{G}_{P}(\sigma^{\Sigma_{0}}|_{M}))^{\theta_{0}}.
\] 
We can also define $\Jac_{x}$ on $\Rep(G^{\Sigma_{0}})$ as in the introduction.

\subsection{Construction of a class of representations}
\label{subsec: construction}

The construction of M{\oe}glin is by induction on the rank of the groups and it depends also on certain so-called {\bf basic properties}, which have to be established at the same time again by induction. So let us assume for $G = G(n')$ with $n' < n$ and elementary $\q \in \cQ{G}$, the irreducible representation $\r^{\Sigma_{0}}(\q, \e)$ of $G^{\Sigma_{0}}$ is well defined and distinct for $\e \in \D{\S{\q}^{\Sigma_{0}}}$. 


Let $b_{\rho, \q, \e} \in Jord_{\rho}(\q_{d})$ be the biggest integer such that $\e$ is ``$\rho$-cuspidal" for $Jord_{\rho, cusp}(\q) := \{ (\rho, \alpha, \delta_{\alpha}) \in Jord_{\rho}(\q) : \alpha \leqslant b_{\rho, \q, \e} \}$, i.e.,
\begin{enumerate}

\item if $(\rho, \alpha, \delta_{\alpha}) \in Jord_{\rho, cusp}(\q)$, then $(\rho, \alpha - 2, \delta_{\alpha - 2}) \in Jord_{\rho, cusp}(\q)$ as long as $\alpha - 2 > 0$;

\item if $(\rho, \alpha, \delta_{\alpha}), (\rho, \alpha - 2, \delta_{\alpha - 2}) \in Jord_{\rho, cusp}(\q)$, then $\e(\rho, \alpha, \delta_{\alpha}) \e(\rho, \alpha - 2, \delta_{\alpha - 2}) = -1$;

\item if $(\rho, 2, \delta_{2}) \in Jord_{\rho, cusp}(\q)$, then $\e(\rho, 2, \delta_{2}) = -1$.

\end{enumerate}
We allow $b_{\rho, \q, \e}$ to be zero. Let $a_{\rho, \q, \e} \in Jord_{\rho}(\q_{d})$ be the smallest integer such that $a_{\rho, \q, \e} > b_{\rho, \q, \e}$, and let $\delta_{\rho, \q, \e}$ be the associated sign. If such $a_{\rho, \q, \e}$ does not exist, we say $a_{\rho, \q, \e} = \infty$.

Along with our assumption on the existence of $\r^{\Sigma_{0}}(\q, \e)$, we also assume they satisfy the following basic properties.

{\bf Basic Properties} (\cite{Moeglin:2006}, Section 2.3):

\begin{enumerate}

\item (Jacquet module): If $\Jac_{\rho||^{x}}\r^{\Sigma_{0}}(\q, \e) \neq 0$, then there exists $b_{\rho, \q, \e} < \alpha \in Jord_{\rho}(\q_{d})$ such that $x = \delta_{\alpha} \alpha$. 

\item (Non-unitary irreducibility) : For $x \geqslant 1/2$, if $2x - 1 \notin Jord_{\rho}(\q_{d}) \cup \{0\}$ or $0 < x \leqslant (b_{\rho, \q, \e} - 1)/2$, then $\rho||^{x} \rtimes \r^{\Sigma_{0}}(\q, \e)$ is irreducible.

\item (Unitary reducibility) : Suppose $Jord_{\rho}(\q_{d})$ contains odd integers. Then $\rho \rtimes \r^{\Sigma_{0}}(\q, \e)$ is irreducible if $1 \in Jord_{\rho}(\q_{d})$, and is semisimple of length 2 without multiplicities otherwise. Moreover, let $\sigma^{\Sigma_{0}}$ be an irreducible subrepresentation of $\rho \rtimes \r^{\Sigma_{0}}(\q, \e)$ in both cases, then $\rho \times \cdots \times \rho \rtimes \sigma^{\Sigma_{0}}$ is irreducible.

\end{enumerate}

\begin{remark}
Property (1) is proved in (\cite{Moeglin:2006}, Section 2.5); Property (2) is proved in (\cite{Moeglin:2006}, Section 2.7). In the tempered case, Property (1) can be deduced easily from (\cite{Xu:preprint3}, Lemma 9.2). But, the general proof of Property (1) depends on Property (2). Property (2) is not obvious even in the tempered case, and its proof in the tempered case is more or less the same as in the general case. A fundamental case of Property (2) is when $\r^{\Sigma_{0}}(\q, \e)$ is supercuspidal, and that follows from (\cite{Xu:preprint3}, Corollary 9.1) (cf. Proposition~\ref{prop: cuspidal reducibility}). Property (3) is proved in (\cite{Moeglin:2006}, Section 2.8) without assuming any unitarity results of Arthur, and in the tempered case it follows easily from Arthur's theory.
\end{remark}

Based on our assumptions, now we can give the construction for $\r^{\Sigma_{0}}(\q, \e)$. 

\begin{definition}
\label{def: constructing elementary Arthur packet}

Suppose $\q \in \cQ{G(n)}$ is an elementary parameter and $\e \in \D{\S{\q}^{\Sigma_{0}}}$.

\begin{enumerate}

\item If $a_{\rho, \q, \e} = \infty$ for all $\rho$, then let $(\p_{cusp}, \e_{cusp}) := (\q_{d}, \e)$, and we define $\r^{\Sigma_{0}}(\q, \e)$ to be $\r_{W}^{\Sigma_{0}}(\p_{cusp}, \e_{cusp})$ in Theorem~\ref{thm: discrete series full orthogonal group}, which is supercuspidal by (\cite{Xu:preprint3}, Theorem 3.3) (cf. Theorem~\ref{thm: supercuspidal parametrization}). 

\item If $a_{\rho, \q, \e} > b_{\rho, \q, \e} + 2$ or $b_{\rho, \q, \e} = 0$, we define 
\[
\r^{\Sigma_{0}}(\q, \e) \hookrightarrow \rho||^{\delta_{\rho, \q, \e}(a_{\rho, \q, \e} -1)/2} \rtimes \r^{\Sigma_{0}}(\q', \e')
\]
to be the unique irreducible subrepresentation, where $(\q' ,\e')$ is obtained from $(\q, \e)$ by changing $(\rho, a_{\rho, \q, \e}, \delta_{\rho, \q, \e})$ to $(\rho, a_{\rho, \q, \e} - 2, \delta_{\rho, \q, \e})$. 

\item If $a_{\rho, \q, \e} = b_{\rho, \q, \e} + 2$, we need to divide into three cases. 

         \begin{enumerate}
         
         \item If $Jord_{\rho}(\q_{d})$ contains even integers and $b_{\rho, \q, \e} \neq 0$, then we define
         \[
         \r^{\Sigma_{0}}(\q, \e) \hookrightarrow <\delta_{\rho, \q, \e}(a_{\rho, \q, \e} -1)/2, \cdots, \delta_{\rho, \q, \e}1/2> \rtimes \r^{\Sigma_{0}}(\q_{-}, \e_{-})
         \]
         to be the unique irreducible subrepresentation, where $(\q_{-}, \e_{-})$ is obtained from $(\q, \e)$ by removing $(\rho, a_{\rho, \q, \e}, \delta_{\rho, \q, \e})$, and changing $(\rho, \alpha,
         \delta_{\alpha})$ to $(\rho, \alpha, -\delta_{\rho, \q, \e})$ with 
         \[
         \e_{-}(\rho, \alpha, -\delta_{\rho, \q, \e}) = -\e(\rho, \alpha, \delta_{\alpha})
         \] 
         for all $\alpha \leqslant b_{\rho, \q, \e}$. Moreover,
         \[
         \r^{\Sigma_{0}}(\q, \e) \hookrightarrow <\delta_{\rho, \q, \e}(a_{\rho, \q, \e} -1)/2, \cdots, -\delta_{\rho, \q, \e}(b_{\rho, \q, \e} -1)/2> \rtimes \r^{\Sigma_{0}}(\q', \e').         
         \]        
         where $(\q' ,\e')$ is obtained from $(\q, \e)$ by removing $a_{\rho, \q, \e}$ and $b_{\rho, \q, \e}$ from $Jord_{\rho}(\q_{d})$. 
                  
         \item If $Jord_{\rho}(\q_{d})$ contains odd integers and $b_{\rho, \q, \e} \neq 1$, then we define $\r^{\Sigma_{0}}(\q, \e)$ to be the unique common irreducible subrepresentation of 
         \[
         <\delta_{\rho, \q, \e}(a_{\rho, \q, \e} -1)/2, \cdots, 0> \rtimes \r^{\Sigma_{0}}(\q_{-}, \e_{-})
         \]
         and 
         \[
         <\delta_{\rho, \q, \e}(a_{\rho, \q, \e} -1)/2, \cdots, -\delta_{\rho, \q, \e}(b_{\rho, \q, \e} -1)/2> \rtimes \r^{\Sigma_{0}}(\q', \e').
         \]
         Here $(\q', \e')$ is obtained from $(\q, \e)$ by removing $a_{\rho, \q, \e}$ and $b_{\rho, \q, \e}$ from $Jord_{\rho}(\q_{d})$; $(\q_{-}, \e_{-})$ is obtained from $(\q, \e)$ by removing $(\rho, a_{\rho, \q, \e}, \delta_{\rho,
         \q, \e})$ and $(\rho, 1, \delta_{1})$, and changing $(\rho, \alpha, \delta_{\alpha})$ to $(\rho, \alpha, -\delta_{\rho, \q, \e})$ with 
         \[
         \e_{-}(\rho, \alpha, -\delta_{\rho, \q, \e}) = -\e(\rho, \alpha, \delta_{\alpha})
         \]
         for $1 < \alpha \leqslant b_{\rho, \q, \e}$. 
         
         \item If $a_{\rho, \q, \e} = 3, b_{\rho, \q, \e} = 1$, we have $(\q_{-}, \e_{-}) = (\q', \e')$ in the notation of $(b)$. By Property 3, $\sigma^{\Sigma_{0}} = \rho \rtimes \r^{\Sigma_{0}}(\q', \e')$ is semisimple of length
         2, and hence we can write $\sigma^{\Sigma_{0}} = \r^{\Sigma_{0}}_{+} \+ \r^{\Sigma_{0}}_{-}$ according to the following two cases.
         
                  \begin{enumerate}
                  
                  \item When $Jord_{\rho}(\q_{d})$ only contains $2$ elements, we fix arbitrary parametrization in $\sigma^{\Sigma_{0}}$, and we define $\r^{\Sigma_{0}}(\q, \e)$ to be the unique irreducible subrepresentation of $\rho||^{\delta_{3}} \rtimes \r^{\Sigma_{0}}_{\zeta}$, with $\zeta = \e(3)\delta_{3}$.

                  \item When $|Jord_{\rho}(\q_{d})| > 2$, i.e., $a_{\rho, \q', \e'} \neq \infty$, we can specify the parametrization in $\sigma^{\Sigma_{0}}$ as follows. Let $(\q'', \e'')$ be obtained from $(\q', \e')$ by
                  changing $(\rho, a_{\rho, \q', \e'}, \delta_{\rho, \q', \e'})$ to $(\rho, 1, \delta_{\rho, \q', \e'})$. Let 
                  \[
                  \Pi^{\Sigma_{0}} = \rho \times <\delta_{\rho, \q', \e'}(a_{\rho, \q', \e'} -1)/2, \cdots, \delta_{\rho, \q', \e'}> \rtimes \r^{\Sigma_{0}}(\q'', \e''),
                  \]
                  \[
                  \sigma^{\Sigma_{0}}_{q} = <\delta_{\rho, \q', \e'}(a_{\rho, \q', \e'} -1)/2, \cdots, 0> \rtimes \r^{\Sigma_{0}}(\q'', \e''),
                  \]
                  and
                  \[
                  \sigma^{\Sigma_{0}}_{s} = <\rho \times <\delta_{\rho, \q', \e'}(a_{\rho, \q', \e'} -1)/2, \cdots, \delta_{\rho, \q', \e'}>> \rtimes \r^{\Sigma_{0}}(\q'', \e'').
                  \]
                  There is an exact sequence
                  \[
                  \xymatrix{0 \ar[r] & \sigma^{\Sigma_{0}}_{s} \ar[r] & \Pi^{\Sigma_{0}} \ar[r] & \sigma^{\Sigma_{0}}_{q} \ar[r] & 0 \\
                                 & & \sigma^{\Sigma_{0}} \ar@{^{(}->}[u] & &}.
                  \]                  
                  We set $\r^{\Sigma_{0}}_{+} = \sigma^{\Sigma_{0}} \cap (s.s.\sigma^{\Sigma_{0}}_{q})$ and $\r^{\Sigma_{0}}_{-} = \sigma^{\Sigma_{0}} \cap (s.s.\sigma^{\Sigma_{0}}_{s})$. Then we define $\r^{\Sigma_{0}}(\q, \e)$ to be the unique irreducible
                  subrepresentation of $\rho||^{\delta_{3}} \rtimes \r^{\Sigma_{0}}_{\zeta}$, with $\zeta = \e(a_{\rho, \q', \e'}) \delta_{\rho, \q', \e'} \e(3) \delta_{3}$. Under such choice 
                  this parametrization is compatible with Arthur's parametrization of discrete series representations in the case $\q = \q_{d}$ (cf. Proposition~\ref{prop: parabolic reduction}, and also \cite{Xu:preprint3}, Proposition 9.3), and it also satisfies Theorem~\ref{thm: Aubert dual for G}. 
                  \end{enumerate}
                           
       \end{enumerate}

\end{enumerate}

\end{definition}

\begin{remark}
\label{rk: construction}
The uniqueness properties in the construction should follow from the property about Jacquet modules, i.e., Property (1). The parametrization of representations of $G^{\Sigma_{0}}$ in this construction is not uniquely determined due to the choices we make in Step (c - i). To fix this one can use the (twisted) endoscopy theory. In the tempered case, there are unique choices to be made here so that this parametrization is the same as Arthur's (cf. Theorem~\ref{thm: discrete series}, and also \cite{Xu:preprint3}, Theorem 2.2). In the nontempered case, we can fix the parametrization by that in the tempered case through the generalized Aubert involution, and we will denote such parametrization by $\r_{M}^{\Sigma_{0}}(\q, \e)$ later on.
\end{remark}

In the next few sections, we would like to show $\cPkt{\q}$ consists of $\sH(G)$-modules obtained from restriction of $\r^{\Sigma_{0}}(\q, \e)$ for $\e \in \D{\S{\q}^{\Sigma_{0}}}$.
To do so, we will introduce two kinds of generalized Aubert involution operators, one on the Grothendieck group of representations of $G^{\Sigma_{0}}$ (similarly also for representations of $G$ viewed as $\sH(G)$-modules), and the other on that of $GL(N) \rtimes <\theta_{N}>$. We will start with $G^{\Sigma_{0}}$ following (\cite{Moeglin:2006}, Section 4).

\subsection{Aubert involution for $G^{\Sigma_{0}}$}
\label{subsec: Aubert dual for G}

Let us fix a positive integer $X_{0}$ and write $x_{0} = (X_{0} - 1) / 2$. We also fix a self-dual irreducible unitary supercuspidal representation $\rho$ of $GL(d_{\rho})$. We denote by $\mathcal{P}^{\Sigma_{0}}_{d_{\rho}}$ the set of $\Sigma_{0}$-conjugacy classes of standard parabolic subgroups $P$ of $G$ whose Levi component $M$ is isomorphic to 
\begin{align}
\label{eq: Aubert dual for G Levi}
GL(a_{1}d_{\rho}) \times \cdots \times GL(a_{l}d_{\rho}) \times G(n - \sum_{i \in [1,l]}a_{i}d_{\rho}).
\end{align}
Here we also require $G(n - \sum_{i \in [1,l]}a_{i}d_{\rho}) \neq SO(2)$ when $d_{\rho} \neq 1$. Let $A_{M}$ be the maximal split central torus of $M$. For $P \in \mathcal{P}^{\Sigma_{0}}_{d_{\rho}}$ and $\sigma^{\Sigma_{0}} \in \Rep(M^{\Sigma_{0}})$, we denote by $\sigma^{\Sigma_{0}}_{< x_{0}}$ the direct sum of  irreducible constitutes of $\sigma$ whose cuspidal support on the general linear factors consist only of $\rho||^{x}$ with $|x| < x_{0}$. In particular, when $G(n - \sum_{i \in [1,l]}a_{i}d_{\rho}) = SO(2) \cong GL(1)$, we also impose this condition on $G(n - \sum_{i \in [1,l]}a_{i}d_{\rho})$.

We define the generalized Aubert involution for $G^{\Sigma_{0}}$ with respect to $(\rho, X_{0})$ as follows. For any $\r^{\Sigma_{0}} \in \Rep(G^{\Sigma_{0}})$, 
\[
inv_{< X_{0}}(\r^{\Sigma_{0}}) := \sum_{P \in \mathcal{P}^{\Sigma_{0}}_{d_{\rho}}} (-1)^{dim A_{M}} \Ind^{G^{\Sigma_{0}}}_{P^{\Sigma_{0}}}(\widetilde{\Jac}_{P^{\Sigma_{0}}} (\r^{\Sigma_{0}})_{< x_{0}}),
\] 
where
\[
\widetilde{\Jac}_{P^{\Sigma_{0}}}(\r^{\Sigma_{0}}) = \begin{cases}
                                                         \Jac_{P^{\Sigma_{0}}}(\r^{\Sigma_{0}}) \otimes \x_{0},  & \text{ if } G(n - \sum_{i \in [1,l]}a_{i}d_{\rho}) = SO(2), \\
                                                         \Jac_{P^{\Sigma_{0}}}(\r^{\Sigma_{0}}), & \text{ otherwise. }  
                                                         \end{cases}
\]
Analogously, we can define $inv_{\leqslant X_{0}}$ if we change all strict inequalities to inequalities here. 
Just as the usual Aubert involution, we have the following result. 

\begin{proposition}[\cite{Moeglin:2006}, Proposition 4]
\label{prop: Aubert dual for G}
$inv_{< X_{0}}$ is an involution on the Grothendieck group of finite-length smooth representations of $G^{\Sigma_{0}}$.
\end{proposition}

However, unlike the usual Aubert involution it is by no means clear that $inv_{< X_{0}}$ preserves irreducibility. Because of this we would like to show it preserves irreducibly at least for the class of representations that we have constructed in Section~\ref{subsec: construction}. The key ingredient of showing this is the following proposition. 

\begin{proposition}[\cite{Moeglin:2006}, Proposition 3]
\label{prop: submodule}
Let $\r^{\Sigma_{0}}(\q, \e)$ be a representation defined as in Section~\ref{subsec: construction}, and let $\mathcal{E}$ be an ordered multi-set of half-integers such that $\forall x \in \mathcal{E}, |x| < (a_{\rho, \q, \e} - 1)/2$. If $\r^{\Sigma_{0}}$ is an irreducible subquotient of $\times_{x \in \mathcal{E}} \rho||^{x} \rtimes \r^{\Sigma_{0}}(\q, \e)$, then there exists an ordered multi-set $\mathcal{E}'$ satisfying 
\[
\{\mathcal{E}'\} \cup \{- \mathcal{E}'\} =\{\mathcal{E}\} \cup \{- \mathcal{E}\},
\]
such that
\[
\r^{\Sigma_{0}} \hookrightarrow \times_{x \in \mathcal{E}'} \rho||^{x} \rtimes \r^{\Sigma_{0}}(\q, \e).
\]
\end{proposition}

Combining Proposition~\ref{prop: Aubert dual for G} and Proposition~\ref{prop: submodule}, one can show the following theorem. 

\begin{theorem}[\cite{Moeglin:2006}, Theorem 4.1]
\label{thm: irreducibility of Aubert dual}
$inv_{< X_{0}}\r^{\Sigma_{0}}(\q, \e)$ is irreducible with a sign in the Grothendieck group of representations of $G^{\Sigma_{0}}$. Moreover, the corresponding irreducible representation $|inv_{< X_{0}}\r^{\Sigma_{0}}(\q, \e)|$ also belongs to the class of representations constructed in Section~\ref{subsec: construction}.
\end{theorem}

One can also determine the sign in this theorem. Let $Jord(\q, \rho, < X_{0}) = \{\alpha \in Jord_{\rho}(\q_{d}) : \alpha < X_{0}\}$, and we define 
\[
\beta(\q, \rho, < X_{0}) := \begin{cases}
                                        (-1)^{|Jord(\q, \rho, < X_{0})|(|Jord(\q, \rho, < X_{0})| - 1)/2} \cdot \prod_{\alpha \in Jord(\q, \rho, < X_{0})}(-1)^{(\alpha - 1)/2}, \\
                                        \text{ if $Jord_{\rho}(\q_{d})$ contains odd integers; } \\
                                        \prod_{\alpha \in Jord(\q, \rho, < X_{0})} (-1)^{\alpha/2}, \text{ if $Jord_{\rho}(\q_{d})$ contains even integers. }                                        
                                        \end{cases}
\]

\begin{proposition} [\cite{Moeglin:2006}, Proposition 4.2]
\label{prop: sign for G}
\[
\beta(\q, \rho, < X_{0}) inv_{< X_{0}} \r^{\Sigma_{0}}(\q, \e) = \begin{cases}
                                                                                 \prod_{\alpha \in Jord(\q, \rho, < X_{0})} \e(\rho, \alpha, \delta_{\alpha}) |inv_{< X_{0}} \r^{\Sigma_{0}}(\q, \e)|, \\
                                                                                 \text{ if $Jord_{\rho}(\q_{d})$ contains even integers; } \\                                                                                
                                                                                 |inv_{< X_{0}} \r^{\Sigma_{0}}(\q, \e)|, \text{ if $Jord_{\rho}(\q_{d})$ contains odd integers. }                                                                                  
                                                                                 \end{cases}                                                                                 
\]
\end{proposition}                                                                                                                                                                 
                                                                                                                                        
Next we want to illustrate the second part of Theorem~\ref{thm: irreducibility of Aubert dual}. This makes use of a compatible relation between this Aubert involution and Jacquet module. To describe this relation, let $P = MN$ be in $\mathcal{P}^{\Sigma_{0}}_{d_{\rho}}$ and let $w_{P}$ be a Weyl group element in $W^{\Sigma_{0}}(M) := \Norm(A_{M}, G^{\Sigma_{0}})/M$ sending all positive roots outside $M$ to negative roots. We can also define $inv^{M^{\Sigma_{0}}}_{< X_{0}}$ by taking the usual Aubert involution on the general linear factors of \eqref{eq: Aubert dual for G Levi}. For any representation $\r^{\Sigma_{0}}$ of $G^{\Sigma_{0}}$, let $\Jac_{P^{\Sigma_{0}}, < x}(\r^{\Sigma_{0}}) = (\Jac_{P^{\Sigma_{0}}}(\r^{\Sigma_{0}}))_{< x}$. Then we have 
\begin{align}
\label{eq: compatible with involution}
\Jac_{P^{\Sigma_{0}}, < x} |inv_{< X_{0}} (\r^{\Sigma_{0}})| = \text{Ad}(w_{P}) |inv^{M^{\Sigma_{0}}}_{< X_{0}} \Jac_{P^{\Sigma_{0}}, < x} (\r^{\Sigma_{0}})|
\end{align}
for all $x \leqslant x_{0}$ and $\r^{\Sigma_{0}} \in \Rep(G^{\Sigma_{0}})$ (cf. \cite{Moeglin:2006}, Section 4.3). From this equality, one can easily conclude the following corollary.

\begin{corollary}[\cite{Moeglin:2006}, Corollary 4.3]
Let $\alpha \in Jord_{\rho}(\q)$ with $a_{\rho, \q, \e} < \alpha$.
\begin{enumerate}

\item If $a_{\rho, \q, \e} > b_{\rho, \q, \e} + 2$, then 
\[
|inv_{< \alpha}(\r^{\Sigma_{0}}(\q, \e))| \hookrightarrow \rho||^{-\delta_{\rho, \q, \e}(a_{\rho, \q, \e} - 1)/2} \rtimes |inv_{< \alpha} (\r^{\Sigma_{0}}(\q', \e'))|,
\] 
where $(\q', \e')$ is obtained by changing $(\rho, a_{\rho, \q, \e}, \delta_{\rho, \q, \e})$ to $(\rho, a_{\rho, \q, \e} - 2, \delta_{\rho, \q, \e})$.

\item If $a_{\rho, \q, \e} = b_{\rho, \q, \e} + 2$, then 
\[
|inv_{< \alpha}(\r^{\Sigma_{0}}(\q, \e))| \hookrightarrow <-\delta_{\rho, \q, \e}(a_{\rho, \q, \e} - 1)/2, \cdots, \delta_{\rho, \q, \e}(b_{\rho, \q, \e} - 1)/2> \rtimes |inv_{< \alpha} (\r^{\Sigma_{0}}(\q', \e'))|,
\]
where $(\q', \e')$ is obtained by removing $a_{\rho, \q, \e}$ and $b_{\rho, \q, \e}$ from $Jord_{\rho}(\q_{d})$.

\end{enumerate}
\end{corollary}

It is easy to see from this corollary that $|inv_{< X_{0}}\r^{\Sigma_{0}}(\q, \e)|$ is in the class of Section~\ref{subsec: construction}. In fact from here one can even describe the pair $(\q^{\sharp}, \e^{\sharp})$, which parametrizes $|inv_{< X_{0}}\r^{\Sigma_{0}}(\q, \e)|$. 

\begin{theorem}[\cite{Moeglin:2006}, Theorem 5]
\label{thm: Aubert dual for G}
For $\r^{\Sigma_{0}}(\q, \e)$, let $\q^{\sharp}$ be obtained from $\q$ by changing $\delta_{\alpha}$ to $-\delta_{\alpha}$ for all $\alpha \in Jord_{\rho}(\q_{d})$ such that $\alpha < X_{0}$, and let $\e^{\sharp} = \e$ under this correspondence. Then one can make suitable choices in the construction of representation corresponding to this new pair $(\q^{\sharp}, \e^{\sharp})$ (see Section~\ref{subsec: construction}, (c-i)) such that $\r^{\Sigma_{0}}(\q^{\sharp}, \e^{\sharp}) = |inv_{< X_{0}}\r^{\Sigma_{0}}(\q, \e)|$.
\end{theorem}

Let $\bar{\Rep}(G)$ be the category of finite-length smooth representations of $G$ viewed as $\sH(G)$-modules. We denote the elements in $\bar{\Rep}(G)$ by $[\r]$ for $\r \in \Rep(G)$, and we call $[\r]$ is irreducible if $\r$ is irreducible. Let
\[
\bar{\Jac}_{P} = \begin{cases}
                         \Jac_{P} + \Jac_{P} \circ \theta_{0},  & \text{ if $G = SO(2n)$ and $M^{\theta_{0}} \neq M$,} \\
                         \Jac_{P},         & \text{ otherwise. }
                         \end{cases}
\]
We can define parabolic induction and Jacquet module on $\bar{\Rep}(G)$ as follows
\[
\Ind^{G}_{P} [\sigma] := [\Ind^{G}_{P} \sigma] \text{ and } \bar{\Jac}_{P} [\r] := [\bar{\Jac}_{P} \r].
\]
Then the generalized Aubert involution $inv_{< X_{0}}$ can also be defined for $\bar{\Rep}(G)$ in an analogous way, i.e., 
\[
\bar{inv}_{< X_{0}}([\r]) := \sum_{P \in \mathcal{P}^{\Sigma_{0}}_{d_{\rho}}} (-1)^{dim A_{M}} \Ind^{G}_{P}(\bar{\Jac}_{P} ([\r])_{< x_{0}}).
\] 
For $\r^{\Sigma_{0}} \in \Rep(G^{\Sigma_{0}})$, we have
\[
[(\Ind^{G^{\Sigma_{0}}}_{P^{\Sigma_{0}}} \, \Jac_{P^{\Sigma_{0}}} \r^{\Sigma_{0}})|_{G}] = \Ind^{G}_{P} \, \bar{\Jac}_{P} [\r^{\Sigma_{0}}|_{G}],
\]
so
\[
[(inv_{< X_{0}} \r^{\Sigma_{0}})|_{G}] = \bar{inv}_{< X_{0}}([\r^{\Sigma_{0}}|_{G}]).
\]

\subsection{Twisted Aubert involution for $GL(N)$}
\label{subsec: Aubert dual for GL(N)}

As in the previous section, we again fix $X_{0}$, $x_{0}$ and $\rho$. We denote by $\mathcal{P}^{\theta_{N}}_{d_{\rho}}$ the set of $\theta_{N}$-invariant standard parabolic subgroups $P$ of $GL(N)$ whose Levi component $M$ is isomorphic to 
\begin{align}
\label{eq: twisted Aubert dual for Levi GL(N)}
GL(a_{1}d_{\rho}) \times \cdots \times GL(a_{l}d_{\rho}) \times GL(N - 2\sum_{i \in [1,l]}a_{i}d_{\rho}) \times GL(a_{l}d_{\rho}) \times \cdots \times GL(a_{1}d_{\rho}). 
\end{align}
Let $A_{M}$ be the maximal split central torus of $M$, and $(A_{M})_{\theta_{N}}$ be the group of its $\theta_{N}$-coinvariants. For $P \in \mathcal{P}^{\theta_{N}}_{d_{\rho}}$ and $\tau \in \Rep(M)$, we denote by $\tau_{< x_{0}}$ the direct sum of  irreducible constitutes of $\tau$ whose cuspidal support on $\times_{i \in [1, l]}GL(a_{i}d_{\rho})$ consists only of $\rho||^{x}$ with $|x| < x_{0}$. Then we define the generalized $\theta_{N}$-twisted Aubert involution for $GL(N)$ with respect to $(\rho, X_{0})$ as follows. For any self-dual representation $\r$ of $GL(N)$, let $\r^{+}$ be an extension of $\r$ to $GL(N) \rtimes <\theta_{N}>$,
\[
inv^{\theta_{N}}_{< X_{0}}(\r^{+}) := \sum_{P \in \mathcal{P}^{\theta_{N}}_{d_{\rho}}} (-1)^{dim (A_{M})_{\theta_{N}}} \Ind^{GL(N)}_{P}(\Jac_{P} (\r^{+})_{< x_{0}}).
\]
We should point out $inv^{\theta_{N}}_{< X_{0}}$ is defined differently from that in (\cite{MW:2006}, Section 3.1). Here $inv^{\theta_{N}}_{< X_{0}}(\r^{+})$ is only an element in the Grothendieck group of representations of $GL(N) \rtimes <\theta_{N}>$ (see \cite{MW:2006}, Section 3.2), even when we take $\r = \r(\q)$. However, if we only consider the $\theta_{N}$-twisted characters of $GL(N)$, we can still get a theorem parallel with Theorem~\ref{thm: Aubert dual for G}. 

\begin{theorem} [\cite{MW:2006}, Proposition 3.1]
\label{thm: Aubert dual for GL(N)}
Let $\q^{\sharp}$ be defined as in Theorem~\ref{thm: Aubert dual for G},  
\[
f_{N}(inv^{\theta_{N}}_{< X_{0}}(\r^{+}(\q))) = f_{N}(\r^{+}(\q^{\sharp})), \,\,\,\,\,\,\,\,\,\,\, f \in C^{\infty}_{c}(GL(N) \rtimes \theta_{N})
\] 
for certain normalization of $\r^{+}(\q^{\sharp})$ with respect to that of $\r^{+}(\q)$.
\end{theorem}

To determine the normalization of $\r^{+}(\q^{\sharp})$ in this theorem, we need the following proposition.

\begin{proposition}[\cite{MW:2006}, Lemma 3.2.2]
Suppose $\r^{+}(\q)$ in Theorem~\ref{thm: Aubert dual for GL(N)} is normalized according to M{\oe}glin-Waldspurger (cf. Section~\ref{sec: M-W normalization}), then the corresponding normalization of $\theta_{N}$ on $\r^{+}(\q^{\sharp})$ differs from $\theta_{MW}(\q^{\sharp})$ by $\beta(\q, \rho, < X_{0})$.
\end{proposition}

The careful readers may notice this proposition is slightly different from the original result of M{\oe}glin-Waldspurger, and that is due to a sign mistake in the statement of (\cite{MW:2006}, Lemma 3.2.2). As a consequence of this proposition, we can rewrite Theorem~\ref{thm: Aubert dual for GL(N)} as follows.

\begin{corollary}
\begin{align}
\label{eq: Aubert dual for GL(N)}
f_{N}(inv^{\theta_{N}}_{< X_{0}}(\r^{+}_{MW}(\q))) = \beta(\q, \rho, < X_{0}) f_{N}(\r^{+}_{MW}(\q^{\sharp})), \,\,\,\,\,\,\,\,\,\,\, f \in C^{\infty}_{c}(GL(N) \rtimes \theta_{N})
\end{align}
where $\r^{+}_{MW}(\q)$ and $\r^{+}_{MW}(\q^{\sharp})$ are normalized extensions of $\r(\q)$ and $\r(\q^{\sharp})$ according to M{\oe}glin-Waldspurger.
\end{corollary}

\subsection{Construction of elementary Arthur packet by Aubert involution}
\label{subsec: construction by Aubert involution}

In the tempered case, we already know $\r(\q, \bar{\e})$ is a $\Sigma_{0}$-orbit of discrete series representations (cf. Proposition~\ref{prop: parabolic reduction} and also \cite{Xu:preprint3}, Proposition 9.3), and moreover its parametrization by $(\q, \bar{\e})$ is the same as Arthur's if we make certain choices in our definition of $\r(\q, \bar{\e})$ (cf. Section~\ref{subsec: construction}, (c-i)). To obtain the nontempered packet, we need to use \eqref{eq: Aubert dual for GL(N)}, and show the following diagram commutes when restricting to distributions associated with elementary parameters.
\begin{align}
\label{diag: twisted compatible with Aubert dual} 
\xymatrix{\D{SI}(G) \ar[d]_{\bar{inv}_{< X_{0}}}  \ar[r] & \D{I}(N^{\theta}) \ar[d]^{inv^{\theta_{N}}_{< X_{0}}} \\
                \D{SI}(G) \ar[r]  & \D{I}(N^{\theta}). }
\end{align}
Here $\D{SI}(G)$ is the space of stable invariant distributions on $G$, $\D{I}(N^{\theta})$ is the space of twisted invariant distributions on $GL(N)$, and the horizontal arrows denote the twisted spectral endoscopic transfers. The commutativity of this diagram (under our restriction) essentially follows from the compatibility of twisted endoscopic transfer with both Jacquet module and parabolic induction, and we will give its proof
in Appendix~\ref{sec: compatibility}. If we apply this diagram to $\cPkt{MW}(\q)$ (see \eqref{eq: stable distribution MW}) and expand using \eqref{eq: MW character relation GL(N)} and \eqref{eq: Aubert dual for GL(N)}, we get
\begin{align*}
f^{G} (\sum_{\bar{\e} \in \D{\S{\q}}} \bar{\e}(s_{\q}) \bar{inv}_{< X_{0}} \r_{MW}(\q, \bar{\e})) & =  \beta(\q, \rho, < X_{0}) f_{N^{\theta}, MW}(\r(\q^{\sharp}))  \\
& = \beta(\q, \rho, < X_{0}) f^{G} (\sum_{\bar{\e} \in \D{\S{\q^{\sharp}}}} \bar{\e}(s_{\q^{\sharp}}) \r_{MW}(\q^{\sharp}, \bar{\e})),
\end{align*}
where $f \in C^{\infty}_{c}(GL(N))$, and $f^{G} \in C^{\infty}_{c}(G)$ is its twisted endoscopic transfer. Hence
\begin{align}
\label{eq: character relation via Aubert dual}
\sum_{\bar{\e} \in \D{\S{\q}}} \bar{\e}(s_{\q}) f_{G}(\bar{inv}_{< X_{0}} \r_{MW}(\q, \bar{\e})) = \beta(\q, \rho, < X_{0}) \sum_{\bar{\e} \in \D{\S{\q^{\sharp}}}} \bar{\e}(s_{\q^{\sharp}}) f_{G}(\r_{MW}(\q^{\sharp}, \bar{\e})),
\end{align}
for any $f \in \sH(G)$. 

\begin{lemma}
\label{lemma: sign}
\[
\e(s_{\q}) / \e(s_{\q^{\sharp}}) = \begin{cases}
                                                   \prod_{\alpha \in Jord(\q, \rho, < X_{0})} \e(\rho, \alpha, \delta_{\alpha}),
                                                   \text{ if $Jord_{\rho}(\q_{d})$ contains even integers, } \\                                                                                
                                                   1, \text{ if $Jord_{\rho}(\q_{d})$ contains odd integers. }                                                                                  
                                                   \end{cases}  
\]
\end{lemma}

\begin{proof}
It suffices to note that 
\[
s_{\q} s_{\q^{\sharp}}(\rho, \alpha, \delta_{\alpha}) = \begin{cases}
                                                                                   -1 & \text{ if $\alpha < X_{0}$ and $\alpha$ is even, } \\
                                                                                   1  & \text{ otherwise. }
                                                                                   \end{cases}
\]
\end{proof}

The equality \eqref{eq: character relation via Aubert dual} suggests we may construct the nontempered Arthur packet by applying the generalized Aubert involution consecutively to tempered packet. So we make the following definition.

\begin{definition}
Suppose $\q \in \cQ{G}$ is elementary, for $\e \in \D{\S{\q}^{\Sigma_{0}}}$ we define
\[
\r^{\Sigma_{0}}_{M}(\q, \e) := \circ_{(\rho, a, \delta_{a}) \in Jord(\q) : \delta_{a} = -1} (|inv_{< a}| \circ |inv_{\leqslant a}|) \r^{\Sigma_{0}}_{W}(\q_{d}, \e)
\]
and
\[
\r_{M}(\q, \bar{\e}) := \circ_{(\rho, a, \delta_{a}) \in Jord(\q) : \delta_{a} = -1} (|\bar{inv}_{< a}| \circ |\bar{inv}_{\leqslant a}|) \r_{W}(\q_{d}, \bar{\e}),
\]
where we have $\D{\S{\q}^{\Sigma_{0}}} \cong \D{\S{\q_{d}}^{\Sigma_{0}}}$ (resp. $\D{\S{\q}} \cong \D{\S{\q_{d}}}$) by identifying $Jord(\q)$ with $Jord(\q_{d})$.
\end{definition}

From Theorem~\ref{thm: Aubert dual for G}, it is clear that $\r^{\Sigma_{0}}_{M}(\q, \e) = \r^{\Sigma_{0}}(\q, \e)$ constructed in Section~\ref{subsec: construction}, but with fixed parametrization determined by that of tempered representations (cf. Remark~\ref{rk: construction}). It follows from Theorem ~\ref{thm: discrete series full orthogonal group} that 
\begin{align}
\label{eq: full orthogonal character}
\r^{\Sigma_{0}}_{M}(\q, \e \e_{0}) \cong \r^{\Sigma_{0}}_{M}(\q, \e) \otimes \x_{0}.
\end{align}
Moreover, $[\r^{\Sigma_{0}}_{M}(\q, \e)|_{G}] = 2 \r_{M}(\q, \bar{\e})$ if $G$ is special even orthogonal and $\S{\q}^{\Sigma_{0}} = \S{\q}$, or $\r_{M}(\q ,\bar{\e})$ otherwise. In particular, $\r_{M}(\q ,\bar{\e})$ is irreducible.

\begin{theorem}
Suppose $\q \in \cQ{G}$ is elementary, then
\label{thm: Arthur packet via Aubert dual}
\[
\cPkt{MW}(\q) = \sum_{\bar{\e} \in \D{\S{\q}}} \bar{\e}(s_{\q})\r_{M}(\q, \bar{\e}).
\]
\end{theorem}

\begin{proof}
Note in the tempered case $\r_{M}(\q, \bar{\e}) =  \r_{W}(\q, \bar{\e}) = \r_{MW}(\q, \bar{\e})$, so this is already known. Then from the tempered packet, one can apply the generalized Aubert involution and use the equality \eqref{eq: character relation via Aubert dual} step by step. At last, note 
\begin{align}
\label{eq: Arthur packet via Aubert dual}
\bar{\e}(s_{\q})\beta(\q, \rho, < X_{0}) \bar{inv}_{< X_{0}} \r_{M}(\q, \bar{\e}) = \bar{\e}(s_{\q^{\sharp}}) \r_{M}(\q^{\sharp}, \bar{\e}),
\end{align}
which follows from Proposition~\ref{prop: sign for G} and Lemma~\ref{lemma: sign}.

\end{proof}

At this point, we have shown the elementary Arthur packets of $G$ do contain irreducible representations of $G$ viewed as $\sH(G)$-modules obtained by restriction from the class of representations of $G^{\Sigma_{0}}$ constructed in Section~\ref{subsec: construction}. However, to prove Theorem~\ref{thm: elementary Arthur packet} we still need to find the relation between $\r_{W}(\q, \bar{\e})$ and $\r_{M}(\q, \bar{\e})$. One may think of this as a problem of parametrization, but in fact it is much more subtle than that for we do not know a priori that $\r_{W}(\q, \bar{\e})$ is irreducible or not. Nonetheless, we will show they are irreducible, and at same time compute the difference of parametrization between $\r_{W}(\q, \bar{\e})$ and $\r_{M}(\q, \bar{\e})$.

To describe this difference, we have to introduce a special element $\e^{M/MW}_{\q} \in \D{\S{\q}^{\Sigma_{0}}}$. It is defined in the following way.

\begin{definition}
\label{def: M/MW elementary}
Suppose $\q \in \cQ{G}$ is elementary, and $\alpha \in Jord_{\rho}(\q_{d})$.
\begin{enumerate}

\item If $\alpha$ is even, $\e^{M/MW}_{\q}(\rho, \alpha, \delta_{\alpha}) = 1$.

\item If $\alpha$ is odd, let $m = \sharp \{\alpha' \in Jord_{\rho}(\q_{d}): \alpha' > \alpha, \delta_{\alpha'} = -1\}$ and $n = \sharp \{\alpha' \in Jord_{\rho}(\q_{d}): \alpha' < \alpha\}$. Then
\[
\e^{M/MW}_{\q}(\rho, \alpha, \delta_{\alpha}) = \begin{cases}
                                                                            (-1)^{m} & \text{ if $\delta_{\alpha} = +1$, } \\
                                                                            (-1)^{m+n} & \text{ if $\delta_{\alpha} = -1$. }
                                                                            \end{cases}
\]

\end{enumerate}
\end{definition}

\begin{theorem}
\label{thm: M/MW elementary}
Suppose $\q \in \cQ{G}$ is elementary, then
\[
\r_{M}(\q, \bar{\e}) = \r_{MW}(\q, \bar{\e} \bar{\e}^{M/MW}_{\q}).
\]
\end{theorem}

\begin{proof}
The idea is similar to the proof of Theorem~\ref{thm: Arthur packet via Aubert dual} that we have to apply the generalized Aubert involution step by step. First note in the tempered case, we have by definition $\r_{M}(\q ,\bar{\e}) = \r_{MW}(\q, \bar{\e})$, and it is easy to check that $\e_{\q}^{M/MW} = 1$ in this case. Next, let us assume $\q$ is some elementary parameter satisfying the theorem, and we would like to prove the theorem for $\q^{\sharp}$. In fact this is the critical step in our proof. To be more precise, we have now
\[
\r_{M}(\q, \bar{\e}) = \r_{MW}(\q, \bar{\e} \bar{\e}^{M/MW}_{\q})
\]
under our assumption, and we want to show
\[
\r_{M}(\q^{\sharp}, \bar{\e}) := |\bar{inv}_{< X_{0}} \r_{M}(\q, \bar{\e})| = \r_{MW}(\q^{\sharp}, \bar{\e} \bar{\e}^{M/MW}_{\q^{\sharp}}).
\]
The main ingredient of the proof is a commutative diagram analogous to the diagram \eqref{diag: twisted compatible with Aubert dual}. Note we can identify $\S{\q}$ with $\S{\q^{\sharp}}$, and for any $s \in \S{\q} \cong \S{\q^{\sharp}}$, let $(H, \q_{H}) \rightarrow (\q, s)$ and $(H, \q_{H}^{\sharp}) \rightarrow (\q^{\sharp}, s)$, where $H = G_{I} \times G_{II}$ and $\q_{H} = \q_{I} \times \q_{II}$. Then the following diagram commutes when restricting to distributions associated with elementary parameters, and this again follows from the compatibility of endoscopic transfer with Jacquet module and parabolic induction (see \cite{Hiraga:2004} and Appendix~\ref{sec: compatibility}).
\begin{align}
\label{diag: compatible with Aubert dual} 
\xymatrix{\D{SI}(H) \ar[d]_{\bar{inv}^{H}_{< X_{0}}}  \ar[r] & \D{I}(G) \ar[d]^{\bar{inv}_{< X_{0}}} \\
                \D{SI}(H) \ar[r]  & \D{I}(G). }
\end{align}
Here $\D{I}(G)$ is the space of invariant distributions on $G$, $\D{SI}(H)$ is the space of stable invariant distributions on $H$, and the horizontal arrows denote the spectral endoscopic transfers. We define
\[
\bar{inv}^{H}_{< X_{0}} := \bar{inv}^{\,G_{I}}_{< X_{0}} \otimes \bar{inv}^{G_{II}}_{< X_{0}}
\]
with $\bar{inv}^{\,G_{I}}_{< X_{0}}$ respecting $\rho \otimes \eta'$, where $\eta' = \eta_{I}$ in (Example~\ref{eg: endoscopy} (1), (2)), and $\eta' = 1$ in (Example~\ref{eg: endoscopy} (3)). Applying this diagram to $\cPkt{MW}(\q_{H}) := \cPkt{MW}(\q_{I}) \otimes \cPkt{MW}(\q_{II})$, we get 
\[
\beta(\q_{H}, \rho, < X_{0}) f^{H}_{MW}(\q^{\sharp}_{H}) = \sum_{\bar{\e} \in \D{\S{\q}}} \bar{\e}(s s_{\q}) f_{G}(inv_{< X_{0}} \r_{MW}(\q, \bar{\e})), \,\,\,\,\,\,\,\,\,\,\,\,\,\, f \in \sH(G),
\]
where $\beta(\q_{H}, \rho, < X_{0}) = \beta(\q_{I}, \rho \otimes \eta', < X_{0}) \beta(\q_{II}, \rho, < X_{0}) $. By our assumption, the right hand side can be written as
\[
\sum_{\bar{\e} \in \D{\S{\q}}} \bar{\e}(s s_{\q}) f_{G}(inv_{< X_{0}}  \r_{M}(\q, \bar{\e} \bar{\e}^{M/MW}_{\q})) = \sum_{\bar{\e} \in \D{\S{\q}}} \bar{\e} \bar{\e}^{M/MW}_{\q}(s s_{\q}) f_{G}(inv_{< X_{0}} \r_{M}(\q, \bar{\e})).
\]
Combining \eqref{eq: Arthur packet via Aubert dual}, we have 
\begin{align*}
f^{H}_{MW}(\q^{\sharp}_{H}) & = \beta(\q_{H}, \rho, < X_{0}) \sum_{\bar{\e} \in \D{\S{\q}}} \bar{\e} \bar{\e}^{MW}_{\q}(s s_{\q}) f_{G}(inv_{< X_{0}} \r_{M}(\q, \bar{\e})) \\
& = \beta(\q_{H}, \rho, < X_{0}) \sum_{\bar{\e} \in \D{\S{\q}}} \bar{\e} \bar{\e}^{M/MW}_{\q} (s s_{\q}) \beta(\q, \rho, < X_{0}) \bar{\e}(s_{\q} s_{\q^{\sharp}}) f_{G}(\r_{M}(\q^{\sharp}, \bar{\e})) \\
& = \beta(\q_{H}, \rho, < X_{0}) \beta(\q, \rho, < X_{0}) \bar{\e}^{M/MW}_{\q} (s s_{\q}) \sum_{\bar{\e} \in \D{\S{\q}}} \bar{\e}(s s_{\q^{\sharp}}) f_{G}(\r_{M}(\q^{\sharp}, \bar{\e})).
\end{align*}
Finally, it is a simple fact that
\(
\e^{M/MW}_{\q}(s_{\q}) = 1.
\)
So 
\begin{align}
\label{eq: compatible with Aubert dual}
f^{H}_{MW}(\q^{\sharp}_{H}) = \beta(\q_{H}, \rho, < X_{0}) \beta(\q, \rho, < X_{0}) \bar{\e}^{M/MW}_{\q} (s) \sum_{\bar{\e} \in \D{\S{\q}}} \bar{\e}(s s_{\q^{\sharp}}) f_{G}(\r_{M}(\q^{\sharp}, \bar{\e})).
\end{align}
On the other hand, we have from the character relation that 
\[
f^{H}_{MW}(\q^{\sharp}_{H}) = \sum_{\bar{\e} \in \D{\S{\q^{\sharp}}}} \bar{\e}(s s_{\q^{\sharp}}) f_{G}(\r_{MW}(\q^{\sharp}, \bar{\e})).
\]
Since we know from linear algebra that $\r_{MW}(\q^{\sharp}, \bar{\e})$ are completely determined by these identities for all $s \in \S{\q^{\sharp}}$, it remains for us to show
\[
\beta(\q_{H}, \rho, < X_{0}) \beta(\q, \rho, < X_{0}) = \e^{M/MW}_{\q}  \e^{M/MW}_{\q^{\sharp}}(s).
\]
If $Jord_{\rho}(\q_{d})$ contains even integers, then it is easy to show from the definitions that both sides are equal to $1$. So now let us assume $Jord_{\rho}(\q_{d})$ contains odd integers. Note $Jord(\q) = Jord(\q_{I} \otimes \eta') \sqcup Jord(\q_{II})$. Let $u = |Jord(\q_{I}, \rho \otimes \eta', < X_{0})|$ and $v = |Jord(\q_{II}, \rho, < X_{0})|$, then

\begin{align*}
\beta(\q_{H}, \rho, < X_{0}) \beta(\q, \rho, < X_{0}) = (-1)^{u(u - 1)/2 + v(v - 1)/2 - (u + v)(u + v - 1)/2}
= (-1)^{u v}
\end{align*}
On the other hand, we can index $Jord_{\rho}(\q_{d})$ according to the natural order of integers and assume $Jord(\q_{I}, \rho \otimes \eta', < X_{0}) = \{\alpha_{t_{j}}\}_{j =1}^{u}$. Then
\[
\e^{M/MW}_{\q}  \e^{M/MW}_{\q^{\sharp}}(s) = \prod^{u}_{j = 1}(-1)^{(u + v - t_{j}) + (t_{j} - 1)} 
= (-1)^{u(u + v - 1)}
= (-1)^{u v}.
\]
This finishes the proof.

\end{proof}

\begin{corollary}
Suppose $\q \in \cQ{G}$ is elementary, let
\(
\e^{M/W}_{\q}: = \e^{M/MW}_{\q} \e^{MW/W}_{\q}.
\)
Then
\[
\r_{W}(\q, \bar{\e} \bar{\e}^{M/W}_{\q}) = \r_{M}(\q, \bar{\e}).
\]

\end{corollary}

\begin{proof}
It is clear from Proposition~\ref{prop: MW/W discrete}.
\end{proof}

In particular, this proves Theorem~\ref{thm: elementary Arthur packet}.

\begin{corollary}
\label{cor: full orthogonal character elementary}
Suppose $G$ is special even orthogonal and $\q \in \cQ{G}$ is elementary. For $\bar{\e} \in \D{\S{\q}}$, let $\r_{W}(\q, \bar{\e}) = [\r]$. Then $\r^{\theta_{0}} \cong \r$ if and only if $\S{\q}^{\Sigma_{0}} \neq \S{\q}$.
\end{corollary}

\begin{proof}
This follows from \eqref{eq: full orthogonal character}.
\end{proof}

If $\q \in \cQ{G}$ is elementary, we can define $\Pkt{\q}^{\Sigma_{0}}$ to be the set of irreducible representations of $G^{\Sigma_{0}}$, whose restriction to $G$ belongs to $\cPkt{\q}$. Then it follows from Corollary~\ref{cor: full orthogonal character elementary} and Theorem~\ref{thm: character relation full orthogonal group} that there is a canonical bijection between 
\[
\xymatrix{\D{\S{\q}^{\Sigma_{0}}}  \ar[r] & \Pkt{\q}^{\Sigma_{0}} \\
                 \e \ar@{{|}->}[r]  & \r^{\Sigma_{0}}_{W}(\q, \e),}
\]
such that 
\begin{itemize}

\item
$[\r^{\Sigma_{0}}_{W}(\q, \e)|_{G}] = 2 \r_{W}(\q , \bar{\e})$ if $G$ is special even orthogonal and $\S{\q}^{\Sigma_{0}} = \S{\q}$, or $\r_{W}(\q ,\bar{\e})$ otherwise.

\item For any $s \in \S{\q}^{\Sigma_{0}}$ but not in $\S{\q}$ and $(H, \q_{H}) \rightarrow (\q, s)$, the following identity holds

\begin{align*}
f^{H}_{W}(\q_{H}) = \sum_{\bar{\e} \in \D{\S{\q}}} \e(ss_{\q}) f_{G}(\r_{W}^{\Sigma_{0}}(\q ,\e)) \,\,\,\,\,\,\,\,\,\,\, f \in C^{\infty}_{c}(G \rtimes \theta_{0}).
\end{align*}

\end{itemize}

Let us define $\r^{\Sigma_{0}}_{MW}(\q, \e) := \r_{W}^{\Sigma_{0}}(\q, \e \e^{MW/W}_{\q})$ for $\e \in \D{\S{\q}^{\Sigma_{0}}}$, then we can show in the same way as Proposition~\ref{prop: MW/W discrete} that for any $s \in \S{\q}^{\Sigma_{0}}$ but not in $\S{\q}$ and $(H, \q_{H}) \rightarrow (\q, s)$,

\begin{align*}
f^{H}_{MW}(\q_{H}) = \sum_{\bar{\e} \in \D{\S{\q}}} \e(ss_{\q}) f_{G}(\r_{MW}^{\Sigma_{0}}(\q ,\e)) \,\,\,\,\,\,\,\,\,\,\, f \in C^{\infty}_{c}(G \rtimes \theta_{0}).
\end{align*}

At last, we can extend Theorem~\ref{thm: M/MW elementary} to $G^{\Sigma_{0}}$.

\begin{theorem}
\label{thm: M/MW elementary full orthogonal group}
Suppose $\q \in \cQ{G}$ is elementary, then
\[
\r^{\Sigma_{0}}_{M}(\q, \e) = \r^{\Sigma_{0}}_{MW}(\q, \e \e_{\q}^{M/MW}).
\]
\end{theorem}

\begin{proof}
We can assume $G$ is special even orthogonal and $\S{\q}^{\Sigma_{0}} \neq \S{\q}$. Since $\r_{M}(\q, \bar{\e}) = \r_{MW}(\q, \bar{\e} \bar{\e}_{\q}^{M/MW})$, then 
\[
\r^{\Sigma_{0}}_{MW}(\q, \e \e_{\q}^{M/MW}) =  \r^{\Sigma_{0}}_{M}(\q, \e) \text{ or } \r^{\Sigma_{0}}_{M}(\q, \e) \otimes \x_{0}.
\]
Note when $\q$ is tempered, $\e_{\q}^{M/MW} = \e_{\q}^{MW/W} = 1$ and $\r^{\Sigma_{0}}_{M}(\q, \e) = \r^{\Sigma_{0}}_{W}(\q, \e) = \r^{\Sigma_{0}}_{MW}(\q, \e)$. So as in the proof of Theorem~\ref{thm: M/MW elementary}, we can assume
\[
\r^{\Sigma_{0}}_{M}(\q, \e) = \r^{\Sigma_{0}}_{MW}(\q, \e \e_{\q}^{M/MW})
\]
for some parameter $\q$ by induction, and the critical step is to show 
\[
\r^{\Sigma_{0}}_{M}(\q^{\sharp}, \e) := |inv_{< X_{0}} \r^{\Sigma_{0}}_{M}(\q, \e)| = \r^{\Sigma_{0}}_{MW}(\q^{\sharp}, \e \e^{M/MW}_{\q^{\sharp}}).
\]
We identity $\S{\q}^{\Sigma_{0}} \cong \S{\q^{\sharp}}^{\Sigma_{0}}$, and choose $s^{*} \in \S{\q}^{\Sigma_{0}}$ but not in $\S{\q}$. Let $(H, \q_{H}) \rightarrow (\q, s^{*})$ and $(H, \q_{H}^{\sharp}) \rightarrow (\q^{\sharp}, s)$, where $H = G_{I} \times G_{II}$ and $\q_{H} = \q_{I} \times \q_{II}$. Then we can have the following commutative diagram analogous to \eqref{diag: compatible with Aubert dual} (see Appendix~\ref{sec: compatibility}):
\begin{align}
\label{diag: twisted even orthogonal compatible with Aubert dual} 
\xymatrix{\D{SI}(H) \ar[d]_{inv^{H}_{< X_{0}}}  \ar[r] & \D{I}(G^{\theta_{0}}) \ar[d]^{inv_{< X_{0}}} \\
                \D{SI}(H) \ar[r]  & \D{I}(G^{\theta_{0}}). }
\end{align}
Here $\D{I}(G^{\theta_{0}})$ is the space of $\theta_{0}$-twisted invariant distributions on $G$, and the horizontal arrows denote the twisted spectral endoscopic transfers. We define
\[
inv^{H}_{< X_{0}} := inv^{\,G_{I}}_{< X_{0}} \otimes inv^{G_{II}}_{< X_{0}}
\]
with $inv^{\,G_{I}}_{< X_{0}}$ (resp. $inv^{G_{II}}_{< X_{0}}$) respecting $\rho \otimes \eta_{I}$ (resp. $\rho \otimes \eta_{II}$). Applying this diagram to $\cPkt{MW}(\q_{H}) := \cPkt{MW}(\q_{I}) \otimes \cPkt{MW}(\q_{II})$, one can show 
\begin{align*}
f^{H}_{MW}(\q^{\sharp}_{H}) & = \beta(\q_{H}, \rho, < X_{0}) \beta(\q, \rho, < X_{0}) \e^{M/MW}_{\q} (s^{*}) \sum_{\bar{\e} \in \D{\S{\q}}} \e(s^{*} s_{\q^{\sharp}}) f_{G}(\r^{\Sigma_{0}}_{M}(\q^{\sharp}, \e)) 
\end{align*}
for $f \in C^{\infty}_{c}(G \rtimes \theta_{0})$ (cf. \eqref{eq: compatible with Aubert dual}). As in the proof of Theorem~\ref{thm: M/MW elementary}, we also have
\[
\beta(\q_{H}, \rho, < X_{0}) \beta(\q, \rho, < X_{0}) = \e^{M/MW}_{\q}  \e^{M/MW}_{\q^{\sharp}}(s^{*}).
\]
Since 
\[
f^{H}_{MW}(\q^{\sharp}_{H}) = \sum_{\bar{\e} \in \D{\S{\q^{\sharp}}}} \e(s^{*} s_{\q^{\sharp}}) f_{G}(\r^{\Sigma_{0}}_{MW}(\q^{\sharp}, \e)),
\]
then
\begin{align*}
\sum_{\bar{\e} \in \D{\S{\q^{\sharp}}}} \e(s^{*} s_{\q^{\sharp}}) f_{G}(\r^{\Sigma_{0}}_{MW}(\q^{\sharp}, \e)) & = \sum_{\bar{\e} \in \D{\S{\q}}} \e \e^{M/MW}_{\q^{\sharp}}(s^{*} s_{\q^{\sharp}}) f_{G}(\r^{\Sigma_{0}}_{M}(\q^{\sharp}, \e)) \\
& = \sum_{\bar{\e} \in \D{\S{\q^{\sharp}}}} \e (s^{*} s_{\q^{\sharp}}) f_{G}(\r^{\Sigma_{0}}_{M}(\q^{\sharp}, \e \e^{M/MW}_{\q^{\sharp}})).
\end{align*}
By the linear independence of twisted characters, we have for any $\bar{\e} \in \D{\S{\q^{\sharp}}}$
\[
\e(s^{*} s_{\q^{\sharp}}) f_{G}(\r^{\Sigma_{0}}_{MW}(\q^{\sharp}, \e)) =  \e (s^{*} s_{\q^{\sharp}}) f_{G}(\r^{\Sigma_{0}}_{M}(\q^{\sharp}, \e \e^{M/MW}_{\q^{\sharp}})),
\]
and hence $f_{G}(\r^{\Sigma_{0}}_{MW}(\q^{\sharp}, \e)) = f_{G}(\r^{\Sigma_{0}}_{M}(\q^{\sharp}, \e \e^{M/MW}_{\q^{\sharp}}))$, i.e., $\r^{\Sigma_{0}}_{M}(\q^{\sharp}, \e) = \r^{\Sigma_{0}}_{MW}(\q^{\sharp}, \e \e_{\q^{\sharp}}^{M/MW})$.

\end{proof}

\begin{remark}
Later on we will see M{\oe}glin defines $\r^{\Sigma_{0}}_{M}(\q, \e)$ in the general case, and if one also extends the definition of $\e^{M/MW}_{\q}$ 
to the general case, then Theorem~\ref{thm: M/MW elementary full orthogonal group} is still valid (see Theorem~\ref{thm: M/MW general full orthogonal group}).
\end{remark}

\section{Case of discrete diagonal restriction}
\label{sec: discrete diagonal restriction}

In this section, we would like to look into the Arthur packets associated with parameters having discrete diagonal restrictions. To be more precise, we want to give a parametrization of irreducible constituents of $\r_{W}(\q, \bar{\e})$ (or equivalently $\r_{MW}(\q, \bar{\e})$) in this case. This parametrization is given by M{\oe}glin and we will follow her paper \cite{Moeglin:2009} closely.

As in the elementary case, we start by constructing certain elements in the Grothendieck group of representations of $G^{\Sigma_{0}}$. These elements are parametrized by $\q \in \cQ{G}$ with discrete diagonal restriction and $\e \in \D{\S{\q}^{\Sigma_{0}}}$.

\begin{definition}
Suppose $\q \in \cQ{G}$ has discrete diagonal restriction, and there exists $(\rho, A, B, \zeta) \in Jord(\q)$ such that $A > B$. Let $\e \in \D{\S{\q}^{\Sigma_{0}}}$ and $\eta_{0} := \e(\rho, A, B, \zeta)$. Then we define
\begin{align*}
\r_{M}^{\Sigma_{0}}(\q, \e) := & \+_{C \in ]B, A]} (-1)^{A-C} <\zeta B, \cdots, -\zeta C> \rtimes \Jac_{\zeta (B+2), \cdots, \zeta C} \r_{M}^{\Sigma_{0}}(\q', \e', (\rho, A, B+2, \zeta; \eta_{0})) \\
& \+_{\eta = \pm 1} (-1)^{[(A - B + 1) / 2]} \eta^{A - B + 1} \eta_{0}^{A - B} \r_{M}^{\Sigma_{0}}(\q', \e', (\rho, A, B+1, \zeta; \eta), (\rho, B, B, \zeta; \eta \eta_{0})),
\end{align*}
where $\q'$ is obtained from $\q$ by removing $(\rho, A, B, \zeta)$, and $\e'(\cdot)$ is the restriction of $\e(\cdot)$. 
\end{definition}

\begin{remark}

\begin{enumerate}
\item
When $A = B + 1$ and $\eta_{0} = -1$, the term involving $(\rho, A, B+2, \zeta, \eta_{0})$ does not appear for $\e'(\cdot)$ does not define a character of $\D{\S{\q'}^{\Sigma_{0}}}$ in this case. 

\item

It is clear by induction that
\begin{align}
\label{eq: full orthogonal character 1}
\r_{M}^{\Sigma_{0}}(\q, \e \e_{0})  \cong  \r_{M}^{\Sigma_{0}}(\q, \e) \otimes \x_{0}.
\end{align}

\item

We could also define $\r_{M}(\q, \bar{\e})$ in a similar way. Let 
\[
Jord(\q^{1}) = Jord(\q') \cup \{(\rho, A, B + 2, \zeta)\},
\] 
and 
\[
Jord(\q^{2}) = Jord(\q') \cup \{(\rho, A, B + 1, \zeta), (\rho, B, B, \zeta)\}.
\] 
We can identify $\S{\q} \cong \S{\q^{1}}$ by sending $(\rho, A, B, \zeta)$ to $(\rho, A, B+2, \zeta)$, and map $s \in \S{\q}$ into $\S{\q^{2}}$ by letting 
\[
s(\rho, A, B + 1, \zeta) = s(\rho, B, B, \zeta) := s(\rho, A, B, \zeta).
\]
Then $\S{\q} \hookrightarrow \S{\q^{2}}$ is of index $1$ or $2$. We denote the image of $\bar{\e}$ in $\D{\S{\q^{1}}}$ by $\bar{\e}_{1}$. Let us define
\begin{align*}
\r_{M}(\q, \bar{\e}) := & \+_{C \in ]B, A]} (-1)^{A-C} <\zeta B, \cdots, -\zeta C> \rtimes \bar{\Jac}_{\zeta (B+2), \cdots, \zeta C} \r_{M}(\q^{1}, \bar{\e}_{1}) \\
& \+_{\bar{\e} \leftarrow \bar{\e}_{2} \in \D{\S{\q^{2}}}} (-1)^{[(A - B + 1) / 2]} \e_{2}(\rho, A, B+1, \zeta)^{A - B + 1} \e(\rho, A, B, \zeta)^{A - B}  \r_{M}(\q^{2}, \bar{\e}_{2}).
\end{align*} 
By induction again one observes 
the restriction of $\r^{\Sigma_{0}}_{M}(\q, \e)$ to $G$ viewed as $\sH(G)$-modules is $2 \r_{M}(\q, \bar{\e})$ if $G$ is special even orthogonal and $\S{\q}^{\Sigma_{0}} = \S{\q}$, or $\r_{M}(\q, \bar{\e})$ otherwise. Later we will show $\r_{M}^{\Sigma_{0}}(\q, \e)$ is a representation of $G^{\Sigma_{0}}$, and $\r_{M}(\q, \bar{\e})$ consists of irreducible representations of $G$ viewed as $\sH(G)$-modules in the restriction of $\r_{M}^{\Sigma_{0}}(\q, \e)$ to $G$ without multiplicities. 


\end{enumerate}

\end{remark}

Next we want to show $\cPkt{\q}$ consists of $\r_{M}(\q, \bar{\e})$, and furthermore we would like to compute the difference between the parametrizations of $\r_{M}(\q, \bar{\e})$ and $\r_{MW}(\q, \bar{\e})$. To do so, we need to extend the definition of $\e^{M/MW}_{\q} \in \D{\S{\q}^{\Sigma_{0}}}$ in the previous section.

\begin{definition}
\label{def: M/MW discrete}
Suppose $\q \in \cQ{G}$ has discrete diagonal restriction, and $(\rho, a, b) \in Jord(\q)$.
\begin{enumerate}

\item If $a + b$ is odd, $\e^{M/MW}_{\q}(\rho, a, b) = 1$.

\item If $a + b$ is even, let 
\[
m = \sharp \{(\rho, a', b') \in Jord(\q): a', b' \text{ odd}, \zeta_{a', b'} = -1, |a' - b'| > |a - b| \},
\] 
and 
\[
n = \sharp \{(\rho, a' , b')  \in Jord(\q): a', b' \text{ odd}, |a' - b'| < |a - b| \}.
\] 
Then
\[
\e^{M/MW}_{\q}(\rho, a, b) = \begin{cases}
                                              1 &\text{ if $a, b$ even, } \\
                                              (-1)^{m} &\text{ if $a, b$ odd, $\zeta_{a, b} = +1$, } \\
                                              (-1)^{m + n} &\text{ if $a, b$ odd, $\zeta_{a, b} = -1$. }
                                              \end{cases}
\]

\end{enumerate}
\end{definition}

There is a simple fact about this character $\e^{M/MW}_{\q}$.

\begin{lemma}
Suppose $\q \in \cQ{G}$ has discrete diagonal restriction, then $\e^{M/MW}_{\q}(s_{\q}) = 1$.
\end{lemma}

\begin{proof}
From the definition, we see $\e^{M/MW}_{\q}(\rho, a, b) = 1$ if $b$ is even. Then
\[
\e^{M/MW}_{\q}(s_{\q}) = \prod_{\substack{(\rho, a, b) \in Jord(\q) \\ b \text{ even }}} \e^{M/MW}_{\q}(\rho, a, b) = 1.
\]
\end{proof}

\begin{theorem}
\label{thm: M/MW discrete}
Suppose $\q \in \cQ{G}$ has discrete diagonal restriction, then
\[
\r_{M}(\q, \bar{\e}) = \r_{MW}(\q, \bar{\e} \bar{\e}^{M/MW}_{\q}).
\]
\end{theorem}

Before we prove the theorem, for any $s \in \S{\q}$ let
\begin{align*}
\cPkt{MW, s}(\q) &:= \sum_{\bar{\e} \in \D{\S{\q}}} \bar{\e}(ss_{\q}) \r_{MW}(\q, \bar{\e}), \\
\cPkt{M, s}(\q) &:= \sum_{\bar{\e} \in \D{\S{\q}}} \bar{\e}(ss_{\q}) \r_{M}(\q, \bar{\e}).
\end{align*}
In particular, $\cPkt{MW}(\q) = \cPkt{MW, 1}(\q)$ and we denote $\cPkt{M}(\q) = \cPkt{M, 1}(\q)$. For $\cPkt{M, s}(\q)$, we have the following recursive formula.

\begin{lemma}
\label{lemma: recursive formula for endoscopy}
Suppose $\q \in \cQ{G}$ has discrete diagonal restriction and $s \in \S{\q}$. Let $(\rho, A, B, \zeta) \in Jord(\q)$ such that $A > B$, then
\begin{align*}
\cPkt{M, s}(\q) = & \+_{C \in ]B, A]} (-1)^{A-C} <\zeta B, \cdots, -\zeta C> \rtimes \bar{\Jac}_{\zeta (B+2), \cdots, \zeta C} \cPkt{M, s}(\q', (\rho, A, B+2, \zeta))  \\
& \+ (-1)^{[(A-B+1)/2]} \cPkt{M, s}(\q', (\rho, A, B+1, \zeta), (\rho, B, B, \zeta)),
\end{align*}
where we let $s(\rho, A, B, \zeta) = s(\rho, A, B + 2, \zeta) = s(\rho, A, B + 1, \zeta) = s(\rho, B, B, \zeta)$.
\end{lemma}

\begin{proof}
By definition we have for any $\bar{\e} \in \D{\S{\q}}$,
\begin{align*}
\bar{\e}(ss_{\q}) \, \r_{M}(\q, \bar{\e}) = & \+_{C \in ]B, A]} (-1)^{A-C} <\zeta B, \cdots, -\zeta C> \rtimes \bar{\Jac}_{\zeta (B+2), \cdots, \zeta C}  \, \bar{\e}(ss_{\q}) \, \r_{M}(\q^{1}, \bar{\e}_{1}) \\
& \+_{\bar{\e} \leftarrow \bar{\e}_{2} \in \D{\S{\q^{2}}}} (-1)^{[(A - B + 1) / 2]} \e_{2}(\rho, A, B+1, \zeta)^{A - B + 1} \e(\rho, A, B, \zeta)^{A - B} \, \bar{\e}(ss_{\q}) \, \r_{M}(\q^{2}, \bar{\e}_{2}).
\end{align*}
So it suffices to show $\bar{\e}_{1}(ss_{\q^{1}}) = \bar{\e}(ss_{\q})$ and 
\[
\bar{\e}_{2}(ss_{\q^{2}}) = \e_{2}(\rho, A, B+1, \zeta)^{A - B + 1} \e(\rho, A, B, \zeta)^{A - B} \bar{\e}(ss_{\q}).
\]
The first one is easy for $s_{\q^{1}} = s_{\q}$ under our identification. For the second one, note $\bar{\e}_{2}(s) = \bar{\e}(s)$ and 
\[
\bar{\e}(s_{\q}) = \prod_{(\rho, a, b) \in Jord(\q)} \e(\rho, a, b)^{b-1} = \prod_{(\rho, A, B, \zeta) \in Jord(\q)} \e(\rho, A, B, \zeta)^{A - \zeta B} 
\]
Then
\[
\bar{\e}_{2}(s_{\q^{2}}) / \bar{\e}(s_{\q}) = \e_{2}(\rho, A, B+1, \zeta)^{A - \zeta (B + 1)} \e_{2}(\rho, B, B, \zeta)^{B - \zeta B} / \e(\rho, A, B, \zeta)^{A - \zeta B}.
\]
Using the fact that $\e_{2}(\rho, A, B+1, \zeta) \e_{2}(\rho, B, B, \zeta) = \e(\rho, A, B, \zeta)$, we have
\begin{align*}
\bar{\e}_{2}(s_{\q^{2}}) / \bar{\e}(s_{\q}) & = \e_{2}(\rho, A, B+1, \zeta)^{A - \zeta (B + 1)} \e(\rho, A, B, \zeta)^{B - \zeta B} \e_{2}(\rho, A, B+1, \zeta)^{-B + \zeta B} / \e(\rho, A, B, \zeta)^{A - \zeta B}  \\
& = \e_{2}(\rho, A, B+1, \zeta)^{A - B - 1} \e(\rho, A, B, \zeta)^{B - A} = \e_{2}(\rho, A, B+1, \zeta)^{A - B + 1} \e(\rho, A, B, \zeta)^{A - B}. 
\end{align*}
This finishes the proof.

\end{proof}

\begin{lemma}
Suppose $\q \in \cQ{G}$ has discrete diagonal restriction, then $\cPkt{MW}(\q) = \cPkt{M}(\q)$. 
\end{lemma}

\begin{proof}
Lemma~\ref{lemma: recursive formula for endoscopy} and Proposition~\ref{prop: recursive formula for packet} allows us to reduce this lemma to the case of elementary Arthur packets, where the statement is already known.
\end{proof}

Now we can give the poof of Theorem~\ref{thm: M/MW discrete}.

\begin{proof}
Since $\e^{M/MW}_{\q}(s_{\q}) = 1$, it is enough to show $\cPkt{M, s}(\q) = \e^{M/MW}_{\q}(s) \cPkt{MW, s}(\q)$ for all $s \in \S{\q}$. From the previous lemma, we know this is true for $s = 1$. So we can assume $s \neq 1$ in the rest of the proof. By induction, we may assume the theorem is true for $\q^{1}$ and $\q^{2}$, i.e.,
\begin{align*}
\cPkt{M, s}(\q^{1}) & = \e^{M/MW}_{\q^{1}}(s) \cPkt{MW, s}(\q^{1}), \\
\cPkt{M, s}(\q^{2}) & = \e^{M/MW}_{\q^{2}}(s) \cPkt{MW, s}(\q^{2}).
\end{align*}
Suppose $(H, \q_{H}) \rightarrow (\q ,s)$ and $\q_{s} := \q_{H} = \q_{I} \times \q_{II}$. We can assume $(\rho, A, B, \zeta) \in Jord(\q_{II})$ for the other case is similar. Let $\q^{1}_{s} = \q^{1}_{I} \times \q^{1}_{II}$ and $\q^{2}_{s} = \q^{2}_{I} \times \q^{2}_{II}$. In particular, $\q_{I} = \q^{1}_{I} = \q^{2}_{I}$. Note $\cPkt{MW, s}(\q^{1})$ (resp. $\cPkt{MW, s}(\q^{2})$) is the spectral endoscopic transfer of $\cPkt{MW}(\q^{1}_{I}) \otimes \cPkt{MW}(\q^{1}_{II})$ (resp. $\cPkt{MW}(\q^{2}_{I}) \otimes \cPkt{MW}(\q^{2}_{II})$). By the compatibility of endoscopic transfer with Jacquet module and parabolic induction, we can conclude $\cPkt{M, s}(\q)$ is the spectral endoscopic transfer of 
\begin{align*}
& \+_{C \in ]B, A]} (-1)^{A-C} \e^{M/MW}_{\q^{1}}(s) <\zeta B, \cdots, -\zeta C> \rtimes \bar{\Jac}_{\zeta (B+2), \cdots, \zeta C} (\cPkt{MW}(\q^{1}_{I}) \otimes \cPkt{MW}(\q^{1}_{II})) \\
&\+ (-1)^{[(A-B+1)/2]}  \e^{M/MW}_{\q^{2}}(s) \cPkt{MW}(\q^{2}_{I}) \otimes \cPkt{MW}(\q^{2}_{II}).
\end{align*}
Note $\bar{\Jac}_{\zeta D} \cPkt{MW}(\q^{1}_{I}) = 0$ for any $B+2 \leqslant D \leqslant A$, which follows from the corresponding vanishing fact for Jacquet modules of $\r(\q^{1}_{I})$. 
Then we can rewrite it as 
\begin{align*}
& \+_{C \in ]B, A]} (-1)^{A-C} \e^{M/MW}_{\q^{1}}(s) \cPkt{MW}(\q^{1}_{I}) \otimes \begin{pmatrix} <\zeta B, \cdots, -\zeta C> \rtimes \bar{\Jac}_{\zeta (B+2), \cdots, \zeta C} \cPkt{MW}(\q^{1}_{II}) \end{pmatrix} \\
&\+ (-1)^{[(A-B+1)/2]}  \e^{M/MW}_{\q^{2}}(s) \cPkt{MW}(\q^{2}_{I}) \otimes \cPkt{MW}(\q^{2}_{II}).
\end{align*}
If we can show 
\begin{align}
\label{eq: M/MW discrete 1}
\e^{M/MW}_{\q}(s) = \e^{M/MW}_{\q^{1}}(s) = \e^{M/MW}_{\q^{2}}(s),
\end{align}
then that means $\cPkt{M, s}(\q)$ is the spectral endoscopic transfer of $\e^{M/MW}_{\q}(s) \cPkt{MW}(\q_{I}) \otimes \cPkt{MW}(\q_{II})$. Hence 
\[
\cPkt{M, s}(\q) = \e^{M/MW}_{\q}(s) \cPkt{MW, s}(\q).
\]
Finally, it is an easy exercise to verify \eqref{eq: M/MW discrete 1}. In fact, one can assume $s(\rho, A, B, \zeta) = 1$, then the set of Jordan blocks $(\rho, a', b')$ such that $s(\rho, a', b') = -1$ is the same for $\q, \q^{1}$ and $\q^{2}$, and it is enough to show $\e^{M/MW}_{\q}(\rho, a', b') = \e^{M/MW}_{\q^{1}}(\rho, a', b') = \e^{M/MW}_{\q^{2}}(\rho, a', b')$ for any $(\rho, a', b')$ in this set. Recall
\begin{align*}
(\rho, A, B+2, \zeta) & = (\rho, a+2\zeta, b-2\zeta), \\
(\rho, A, B+1, \zeta) & = (\rho, a+ \zeta, b-\zeta), \\
(\rho, B, B, \zeta) & = (\rho, sup(0, a-b) + 1, sup(0, b-a) + 1). 
\end{align*}
One checks easily that the contribution of $(\rho, A, B, \zeta)$ to the numbers $m, n$ in Definition~\ref{def: M/MW discrete} for $\q$ is the same as $(\rho, A, B+2, \zeta)$ for $\q^{1}$, and $(\rho, A, B+1, \zeta), (\rho, B, B, \zeta)$ for $\q^{2}$ modulo 2. Then the rest is clear.

\end{proof}

One consequence of Theorem~\ref{thm: M/MW discrete} is that $\r_{M}(\q, \bar{\e})$ is an $\sH(G)$-module, which is by no means clear from our definition. In fact, the main goal of \cite{Moeglin:2009} is to show $\r^{\Sigma_{0}}_{M}(\q, \e)$ is a representation of $G^{\Sigma_{0}}$ and characterize its irreducible constituents, which also implies $\r_{M}(\q, \bar{\e})$ is an $\sH(G)$-module independent of Arthur's theory. 

\begin{theorem}[\cite{Moeglin:2009}, Theorem 4.2]
\label{thm: M parametrization}
Suppose $\q \in \cQ{G}$ has discrete diagonal restriction, and there exists $(\rho, A, B, \zeta) \in Jord(\q)$ such that $A > B$. Let $\e \in \D{\S{\q}^{\Sigma_{0}}}$ and $\eta_{0} := \e(\rho, A, B, \zeta)$. Then we have
\begin{align*}
\r_{M}^{\Sigma_{0}}(\q, \e) & = \+_{l \in [0, [(A - B + 1)/2]]} \+_{\eta = \pm 1 \, : \, \eta_{0} = \eta^{A - B + 1} \prod_{C \in [B+l, A-l]} (-1)^{[C]}} <<\zeta B, \cdots, -\zeta A> \\ 
& \times \cdots \times <\zeta (B + l - 1), \cdots, -\zeta (A - l + 1)>  \rtimes \r_{M}^{\Sigma_{0}}(\q', \e', \cup_{C \in [B+l, A-l]}(\rho, C, C, \zeta; \eta(-1)^{[C]}))>,
\end{align*}
where $\q'$ is obtained from $\q$ by removing $(\rho, A, B, \zeta)$, and $\e'(\cdot)$ is the restriction of $\e(\cdot)$. In particular, when $l = (A - B + 1)/2$ and $\eta_{0} = 1$, we will just take one value for $\eta$, since both values give the same term.
\end{theorem}

\begin{remark}
The complicated condition on $\eta$ comes from the fact that $\eta(-1)^{[C]}$ with $\e'(\cdot)$ needs to define a character $\e_{-}$ of $\S{\q_{-}}^{\Sigma_{0}}$, where $Jord(\q_{-})$ is obtained from $Jord(\q')$ by adding $\cup_{C \in [B+l, A-l]}(\rho, C, C, \zeta)$. 
\end{remark}

This theorem shows $\r_{M}^{\Sigma_{0}}(\q, \e)$ is a representation $G^{\Sigma_{0}}$, and allows us to decompose it according to two parameters $\ul, \ueta$, where $\ul$ is an integer-valued function on $Jord(\q)$ and $\ueta$ is a $\Two$-valued function on $Jord(\q)$. In the notations of this theorem, we let $\ul(\rho, A, B, \zeta) = l$ and $\ueta(\rho, A, B, \zeta) = \eta(-1)^{[B + l]}$. Then 
\[
\ul(\rho, A, B, \zeta) \in [0, [(A-B+1)/2]],
\]
and
\begin{align}
\label{eq: endoscopic character formula}
\e(\rho, A, B, \zeta) = \ueta(\rho, A, B, \zeta)^{A - B + 1} (-1)^{[(A-B+1)/2] + \ul(\rho, A, B, \zeta)}.
\end{align}
Let us denote by $\e_{\ul, \ueta}$ the character of $\S{\q}^{\Sigma_{0}}$ defined by $(\ul, \ueta)$ through this formula. Then we define for any pair $(\ul, \ueta)$ such that $\e_{\ul, \ueta} \in \D{\S{\q}^{\Sigma_{0}}}$,
\begin{align*}
\r_{M}^{\Sigma_{0}}(\q, \ul, \ueta) &:= <<\zeta B, \cdots, -\zeta A>  \times \cdots \times <\zeta (B + \ul(\rho, A, B, \zeta) - 1), \cdots, -\zeta (A - \ul(\rho, A, B, \zeta) + 1)>  \\
& \rtimes \r_{M}^{\Sigma_{0}}(\q_{-}, \ul_{-}, \ueta_{-})>,
\end{align*}
where $\q_{-}$ is defined as in the remark, and $\ul_{-}, \ueta_{-}$ are extended from $\ul, \ueta$ by letting $\ul_{-}(\rho, C, C, \zeta) = 0$ and $\ueta_{-}(\rho, C, C, \zeta) = \eta(-1)^{[C]}$. 
In the theorem, M{\oe}glin shows $\r_{M}^{\Sigma_{0}}(\q, \ul, \ueta)$ is irreducible. In fact, one can also show
\begin{align*}
\r_{M}^{\Sigma_{0}}(\q, \ul, \ueta) & \hookrightarrow \times_{(\rho, A, B, \zeta) \in Jord(\q)} 
       \begin{pmatrix}
              \zeta B & \cdots & -\zeta A \\
              \vdots &  & \vdots \\
              \zeta (B + \ul(\rho, A, B, \zeta) - 1) & \cdots & - \zeta (A - \ul(\rho, A, B, \zeta) + 1)
       \end{pmatrix} \\
& \rtimes \r_{M}^{\Sigma_{0}}\Big(\cup_{(\rho, A, B, \zeta) \in Jord(\q)} \cup_{C \in [B + \ul(\rho, A, B, \zeta), A - \ul(\rho, A, B, \zeta)]} (\rho, C, C, \zeta; \ueta(\rho, A, B, \zeta)(-1)^{C - B - \ul(\rho, A, B, \zeta)})\Big)
\end{align*}
as the unique irreducible subrepresentation. We define $\r_{M}(\q, \ul, \ueta)$ to be the irreducible representation of $G$ viewed as $\sH(G)$-module in the restriction of $\r_{M}^{\Sigma_{0}}(\q, \ul, \ueta)$ to $G$. Then
\begin{align*}
\r_{M}(\q, \ul, \ueta) & \hookrightarrow \times_{(\rho, A, B, \zeta) \in Jord(\q)} 
       \begin{pmatrix}
              \zeta B & \cdots & -\zeta A \\
              \vdots &  & \vdots \\
              \zeta (B + \ul(\rho, A, B, \zeta) - 1) & \cdots & - \zeta (A - \ul(\rho, A, B, \zeta) + 1)
       \end{pmatrix} \\
& \rtimes \r_{M}\Big(\cup_{(\rho, A, B, \zeta) \in Jord(\q)} \cup_{C \in [B+\ul(\rho, A, B, \zeta), A - \ul(\rho, A, B, \zeta)]} (\rho, C, C, \zeta; \ueta(\rho, A, B, \zeta)(-1)^{C - B - \ul(\rho, A, B, \zeta)})\Big)
\end{align*}
as the unique irreducible element in $\bar{\Rep}(G)$ forming an $\sH(G)$-submodule.

We define an equivalence relation on pairs $(\ul, \ueta)$, such that 
\(
(\ul, \ueta) \sim_{\Sigma_{0}} (\ul', \ueta')
\)
if and only if $\ul = \ul'$ and $(\ueta/\ueta') (\rho, A, B, \zeta)= 1$ unless $\ul(\rho, A, B, \zeta) = (A - B +1)/2$. It is clear that $\r_{M}^{\Sigma_{0}}(\q, \ul, \ueta) \cong \r_{M}^{\Sigma_{0}}(\q, \ul', \ueta')$ if $(\ul, \ueta) \sim_{\Sigma_{0}} (\ul', \ueta')$. In fact, the converse is also true.

\begin{proposition}
\label{prop: M parametrization}
Suppose $\q \in \cQ{G}$ has discrete diagonal restriction and $\e \in \D{\S{\q}^{\Sigma_{0}}}$, then
\[
\r^{\Sigma_{0}}_{M}(\q, \e) = \bigoplus_{\{(\ul, \ueta): \, \e = \e_{\ul, \ueta} \}/\sim_{\Sigma_{0}}} \r^{\Sigma_{0}}_{M}(\q, \ul, \ueta).
\]
Moreover, $\r_{M}^{\Sigma_{0}}(\q, \ul, \ueta) \cong \r_{M}^{\Sigma_{0}}(\q, \ul', \ueta')$ if and only if $(\ul, \ueta) \sim_{\Sigma_{0}} (\ul', \ueta')$. 
\end{proposition}

\begin{proof}
The only thing which may not be so obvious from Theorem~\ref{thm: M parametrization} is the fact that $\r_{M}^{\Sigma_{0}}(\q, \ul, \ueta) \ncong \r_{M}^{\Sigma_{0}}(\q, \ul', \ueta')$ if $(\ul, \ueta) \nsim_{\Sigma_{0}} (\ul', \ueta')$. But this can be shown by comparing the Jacquet modules of these representations.
\end{proof}

\begin{remark}
If $Jord(\q)$ contains $(\rho, a, b)$ with $a = b$, then our definition of $\r^{\Sigma_{0}}_{M}(\q, \ul, \ueta)$ will depend on the choice of sign $\zeta_{a, b}$. However, it is not hard to show the representation $\r^{\Sigma_{0}}_{M}(\q, \ul, \ueta)$ is independent of $\zeta_{a, b}$.
\end{remark}

If $G$ is special even orthogonal, and $\q \in \cQ{G}$ has discrete diagonal restriction, we define a $\Two$-valued function on $Jord(\q)$ by
\[
\ueta_{0}(\rho, A, B, \zeta) = \begin{cases}
                                             -1, & \text{ if $d_{\rho}$ is odd and $A \in \mathbb{Z}$}, \\
                                              1, & \text{ otherwise. }
                                             \end{cases}
\]
Then $\e_{0}(\rho, A, B, \zeta) = \ueta_{0}(\rho, A, B, \zeta)^{A - B +1}$, and hence $\e_{\ul, \ueta \, \ueta_{0}} = \e_{\ul, \ueta} \, \e_{0}$. In general, we let $\ueta_{0} = 1$ if $G$ is not special even orthogonal.

\begin{corollary}
Suppose $\q \in \cQ{G}$ has discrete diagonal restriction, then
\begin{align}
\label{eq: full orthogonal character 2}
\r_{M}^{\Sigma_{0}}(\q, \ul, \ueta \, \ueta_{0}) \cong \r_{M}^{\Sigma_{0}}(\q, \ul, \ueta) \otimes \x_{0}.
\end{align}
\end{corollary}

\begin{proof}
This follows from the formula of $\r_{M}^{\Sigma_{0}}(\q, \ul, \ueta)$ and \eqref{eq: full orthogonal character} in the elementary case.
\end{proof}

We define another equivalence relation on pairs $(\ul, \ueta)$, such that 
\(
(\ul, \ueta) \sim (\ul', \ueta')
\)
if and only if $(\ul, \ueta) \sim_{\Sigma_{0}} (\ul', \ueta')$ or $(\ul, \ueta) \sim_{\Sigma_{0}} (\ul', \ueta' \, \ueta_{0})$. It follows from this corollary that $\r_{M}(\q, \ul, \ueta) = \r_{M}(\q, \ul', \ueta')$ if and only if $(\ul, \ueta) \sim (\ul', \ueta')$. 

\begin{corollary}
\label{cor: M parametrization for G}
Suppose $\q \in \cQ{G}$ has discrete diagonal restriction and $\bar{\e} \in \D{\S{\q}}$, then
\[
\r_{M}(\q, \bar{\e}) = \bigoplus_{\{(\ul, \ueta): \, \bar{\e} = \bar{\e}_{\ul, \ueta} \}/\sim} \r_{M}(\q, \ul, \ueta).
\]
Moreover, 
\[
\bigoplus_{\bar{\e} \leftarrow \e \in \D{\S{\q}^{\Sigma_{0}}}} \r_{M}^{\Sigma_{0}}(\q, \e)
\]
consists of all irreducible representations of $G^{\Sigma_{0}}$, whose restriction to $G$ belong to $\r_{M}(\q, \bar{\e})$.
\end{corollary}

\begin{proof}
We can assume $G$ is special even orthogonal. It follows from Proposition~\ref{prop: M parametrization} that
\[
m \cdot \r_{M}(\q, \bar{\e}) = \r^{\Sigma_{0}}_{M}(\q, \e)|_{G} = \bigoplus_{\{(\ul, \ueta): \, \e = \e_{\ul, \ueta} \}/\sim_{\Sigma_{0}}} \r^{\Sigma_{0}}_{M}(\q, \ul, \ueta)|_{G}
\]
where $m = 2$ if $\S{\q}^{\Sigma_{0}} = \S{\q}$, and $m = 1$ otherwise. By \eqref{eq: full orthogonal character 2}, one can easily see the right hand side is 
\[
m  \bigoplus_{\{(\ul, \ueta): \, \bar{\e} = \bar{\e}_{\ul, \ueta} \}/\sim} \r_{M}(\q, \ul, \ueta).
\]
This proves the first part, and the second part should then be clear.
\end{proof}

Motivated by this corollary, we can define $\Pkt{\q}^{\Sigma_{0}}$ to be the set of irreducible representations of $G^{\Sigma_{0}}$, whose restriction to $G$ belong to $\cPkt{\q}$. In the case $G$ is special even orthogonal and $\q \in \cQ{G}$ has discrete diagonal restriction, suppose $\S{\q}^{\Sigma_{0}} \neq \S{\q}$, then for any $(\ul, \ueta)$,
\[
(\ul, \ueta \, \ueta_{0}) \nsim_{\Sigma_{0}} (\ul, \ueta),
\] 
and hence $\r^{\theta_{0}} \cong \r$ for any irreducible constituent $[\r]$ in $\r_{M}(\q, \bar{\e})$ by \eqref{eq: full orthogonal character 2}. Then it follows from Theorem~\ref{thm: character relation full orthogonal group} that there is a canonical disjoint decomposition 
\[
\Pkt{\q}^{\Sigma_{0}} = \bigsqcup_{\e \in \D{\S{\q}^{\Sigma_{0}}}} \r^{\Sigma_{0}}_{W}(\q, \e)
\]
such that 
\begin{itemize}

\item
$[\r^{\Sigma_{0}}_{W}(\q, \e)|_{G}] = 2 \r_{W}(\q , \bar{\e})$ if $G$ is special even orthogonal and $\S{\q}^{\Sigma_{0}} = \S{\q}$, or $\r_{W}(\q ,\bar{\e})$ otherwise.

\item For any $s \in \S{\q}^{\Sigma_{0}}$ but not in $\S{\q}$ and $(H, \q_{H}) \rightarrow (\q, s)$, the following identity holds

\begin{align*}
f^{H}_{W}(\q_{H}) = \sum_{\bar{\e} \in \D{\S{\q}}} \e(ss_{\q}) f_{G}(\r_{W}^{\Sigma_{0}}(\q ,\e)) \,\,\,\,\,\,\,\,\,\,\, f \in C^{\infty}_{c}(G \rtimes \theta_{0}).
\end{align*}

\end{itemize}

Let us define $\r^{\Sigma_{0}}_{MW}(\q, \e) := \r_{W}^{\Sigma_{0}}(\q, \e \e^{MW/W}_{\q})$ for $\e \in \D{\S{\q}^{\Sigma_{0}}}$, then we can show in the same way as Proposition~\ref{prop: MW/W discrete} that for any $s \in \S{\q}^{\Sigma_{0}}$ but not in $\S{\q}$ and $(H, \q_{H}) \rightarrow (\q, s)$,

\begin{align*}
f^{H}_{MW}(\q_{H}) = \sum_{\bar{\e} \in \D{\S{\q}}} \e(ss_{\q}) f_{G}(\r_{MW}^{\Sigma_{0}}(\q ,\e)) \,\,\,\,\,\,\,\,\,\,\, f \in C^{\infty}_{c}(G \rtimes \theta_{0}).
\end{align*}

At last, we can extend Theorem~\ref{thm: M/MW discrete} to $G^{\Sigma_{0}}$.

\begin{theorem}
\label{thm: M/MW discrete full orthogonal group}
Suppose $\q \in \cQ{G}$ has discrete diagonal restriction, then
\[
\r^{\Sigma_{0}}_{M}(\q, \e) = \r^{\Sigma_{0}}_{MW}(\q, \e \e_{\q}^{M/MW}).
\]
\end{theorem}

\begin{proof}
We can assume $G$ is special even orthogonal and $\S{\q}^{\Sigma_{0}} \neq \S{\q}$. The proof goes in the same way as that of Theorem~\ref{thm: M/MW discrete}. First we choose $s^{*} \in \S{\q}^{\Sigma_{0}}$ but not in $\S{\q}$, and we define
\begin{align*}
\Pkt{MW, s^{*}}^{\Sigma_{0}}(\q) &:= \sum_{\bar{\e} \in \D{\S{\q}}} \e(s^{*}s_{\q}) \r^{\Sigma_{0}}_{MW}(\q, \e), \\
\Pkt{M, s^{*}}^{\Sigma_{0}}(\q) &:= \sum_{\bar{\e} \in \D{\S{\q}}} \e(s^{*}s_{\q}) \r^{\Sigma_{0}}_{M}(\q, \e).
\end{align*}
Secondly we can extend Lemma~\ref{lemma: recursive formula for endoscopy} to this case, i.e., for $(\rho, A, B, \zeta) \in Jord(\q)$ such that $A > B$, 
\begin{align*}
\Pkt{M, s^{*}}^{\Sigma_{0}}(\q) = & \+_{C \in ]B, A]} (-1)^{A-C} <\zeta B, \cdots, -\zeta C> \rtimes \Jac_{\zeta (B+2), \cdots, \zeta C} \Pkt{M, s^{*}}^{\Sigma_{0}}(\q', (\rho, A, B+2, \zeta))  \\
& \+ (-1)^{[(A-B+1)/2]} \Pkt{M, s^{*}}^{\Sigma_{0}}(\q', (\rho, A, B+1, \zeta), (\rho, B, B, \zeta)),
\end{align*}
where we let $s^{*}(\rho, A, B, \zeta) = s^{*}(\rho, A, B + 2, \zeta) = s^{*}(\rho, A, B + 1, \zeta) = s^{*}(\rho, B, B, \zeta)$. The proof is the same. Then we can show by induction that 
\[
\Pkt{M, s^{*}}^{\Sigma_{0}}(\q) = \e^{M/MW}_{\q}(s^{*}) \Pkt{MW, s^{*}}^{\Sigma_{0}}(\q).
\]
This is because of Theorem~\ref{thm: M/MW elementary full orthogonal group} and the fact that \eqref{eq: M/MW discrete 1} still holds in this case. Finally, since $\r_{M}(\q, \bar{\e}) = \r_{MW}(\q, \bar{\e} \bar{\e}_{\q}^{M/MW})$, we have
\[
\e(s^{*}s_{\q}) \r^{\Sigma_{0}}_{M}(\q, \e) = \e^{M/MW}_{\q}(s^{*}) \cdot \e \e^{M/MW}_{\q} (s^{*}s_{\q}) \r^{\Sigma_{0}}_{MW}(\q, \e \e^{M/MW}_{\q}) = \e(s^{*}s_{\q}) \r^{\Sigma_{0}}_{MW}(\q, \e \e^{M/MW}_{\q})
\]
by the linear independence of twisted characters. Hence $\r^{\Sigma_{0}}_{M}(\q, \e) = \r^{\Sigma_{0}}_{MW}(\q, \e \e_{\q}^{M/MW})$.

\end{proof}



\section{General case}
\label{sec: general case}

In this section, we consider M{\oe}glin's parametrization of elements in $\cPkt{\q}$ for general $\q \in \cQ{G}$. The idea is similar to Section~\ref{sec: M-W normalization}. We first assume $\q = \q_{p}$, and fix an order $>_{\q}$ on $Jord(\q)$ satisfying condition $(\mathcal{P})$. We also choose a parameter $\q_{\gg}$ dominating $\q$ with discrete diagonal restriction and natural order, and we identify $\S{\q^{>}}^{\Sigma_{0}} \cong \S{\q_{\gg}}^{\Sigma_{0}}$. Then we define for $\e \in \D{\S{\q^{>}}^{\Sigma_{0}}}$,
\begin{align}
\label{eq: M general full orthogonal group}
\r_{M}^{\Sigma_{0}}(\q, \e) := \circ _{(\rho, A, B, \zeta) \in Jord(\q)} \Jac_{(\rho, A_{\gg}, B_{\gg}, \zeta) \mapsto (\rho, A, B, \zeta)} \r^{\Sigma_{0}}_{M}(\q_{\gg}, \e),
\end{align}
where the composition is taken in the decreasing order. Since $\r_{M}^{\Sigma_{0}}(\q_{\gg}, \e \e_{0}) \cong \r_{M}^{\Sigma_{0}}(\q_{\gg}, \e) \otimes \x_{0}$, then
\[
\r_{M}^{\Sigma_{0}}(\q, \e \e_{0}) \cong \r_{M}^{\Sigma_{0}}(\q, \e) \otimes \x_{0}.
\]
We also define
\begin{align}
\label{eq: M general}
\r_{M}(\q, \bar{\e}) := \circ _{(\rho, A, B, \zeta) \in Jord(\q)} \bar{\Jac}_{(\rho, A_{\gg}, B_{\gg}, \zeta) \mapsto (\rho, A, B, \zeta)} \r_{M}(\q_{\gg}, \bar{\e}).
\end{align}
It follows from the case of discrete diagonal restriction that 
the restriction of $\r^{\Sigma_{0}}_{M}(\q, \e)$ to $G$ viewed as $\sH(G)$-modules is $2 \r_{M}(\q, \bar{\e})$ if $G$ is special even orthogonal and $\S{\q}^{\Sigma_{0}} = \S{\q}$, or $\r_{M}(\q, \bar{\e})$ otherwise. 

Next we extend the definition of $\e^{M/MW}_{\q} \in \D{\S{\q^{>}}^{\Sigma_{0}}}$ to this case.

\begin{definition}
\label{def: M/MW general}
Suppose $\q = \q_{p} \in \cQ{G}$, and $(\rho, a, b) \in Jord(\q)$. We fix an order $>_{\q}$ on $Jord(\q)$ satisfying condition $(\mathcal{P})$.
\begin{enumerate}

\item If $a + b$ is odd, $\e^{M/MW}_{\q}(\rho, a, b) = 1$.

\item If $a + b$ is even, let 
\[
m = \sharp \{(\rho, a', b') \in Jord(\q): a', b' \text{ odd}, \zeta_{a', b'} = -1, (\rho, a', b') >_{\q} (\rho, a, b) \},
\] 
and 
\[
n = \sharp \{(\rho, a' , b')  \in Jord(\q): a', b' \text{ odd}, (\rho, a', b') <_{\q} (\rho, a, b) \}.
\] 
Then
\[
\e^{M/MW}_{\q}(\rho, a, b) = \begin{cases}
                                              1 &\text{ if $a, b$ even, } \\
                                              (-1)^{m} &\text{ if $a, b$ odd, $\zeta_{a, b} = +1$, } \\
                                              (-1)^{m + n} &\text{ if $a, b$ odd, $\zeta_{a, b} = -1$. }
                                              \end{cases}
\]

\end{enumerate}
\end{definition}


\begin{proposition}
\label{prop: M/MW general}
Suppose $\q = \q_{p} \in \cQ{G}$ and $\bar{\e} \in \D{\S{\q^{>}}}$, then 
\[
\r_{M}(\q, \bar{\e}) = \r_{MW}(\q, \bar{\e} \bar{\e}^{M/MW}_{\q}).
\]
\end{proposition}

\begin{proof}
By the definition of \eqref{eq: MW general} and \eqref{eq: M general}, it suffices to show $\r_{M}(\q_{\gg}, \bar{\e}) = \r_{MW}(\q_{\gg}, \bar{\e}\bar{\e}^{M/MW}_{\q})$. One checks easily $\e^{M/MW}_{\q} = \e^{M/MW}_{\q_{\gg}}$ by the definition. So now this proposition will follow from Theorem~\ref{thm: M/MW discrete} directly.
\end{proof}

As a consequence, we have the following result.

\begin{proposition}
\label{prop: reducibility general}
Suppose $\q = \q_{p} \in \cQ{G}$ and $\bar{\e} \in \D{\S{\q^{>}}}$. Let $\rho$ be a unitary irreducible supercuspidal representation of $GL(d_{\rho})$.

\begin{enumerate}

\item For $\zeta \in \{\pm 1\}$ and segment $[x, y]$ with $0 \leqslant x \leqslant y$, 
\(
\bar{\Jac}_{\zeta x, \cdots, \zeta y} \r_{M}(\q, \bar{\e}) = 0
\)
unless there exists a sequence of Jordan blocks $\{(\rho, A_{i}, B_{i}, \zeta)\}_{i=1}^{n} \subseteq Jord(\q)$ such that $B_{1} = x, A_{n} \geqslant y$, and $B_{i} \leqslant B_{i+1} \leqslant A_{i} + 1$.

\item For $x \in \mathbb{R}$, let $m = \sharp \{ (\rho, A, B, \zeta) \in Jord(\q): \zeta B = x \}$, then $\bar{\Jac}_{\underbrace{x, \cdots, x}_{n}} \, \r_{M}(\q, \bar{\e}) = 0$ if $n > m$.

\end{enumerate}

\end{proposition}

\begin{proof}

Note $\r_{M}(\q, \bar{\e}) = \r_{MW}(\q, \bar{\e} \bar{\e}^{M/MW})$ and 
\[
\r_{MW}(\q, \bar{\e} \bar{\e}^{M/MW}) = \begin{cases}
                                             \r_{W}(\q, \bar{\e} \bar{\e}^{M/MW} \bar{\e}^{MW/W}), & \text{ if $\bar{\e} \bar{\e}^{M/MW} \bar{\e}^{MW/W} \in \D{\S{\q}}$,} \\
                                             0, & \text{ otherwise.}
                                             \end{cases}
\]
So it suffices to show the proposition for $\r_{W}(\q, \bar{\e})$ and $\bar{\e} \in \D{\S{\q}}$. As we see from the proof of Proposition~\ref{prop: MW/W},  
\[
\r_{W}(\q, \bar{\e}) = \frac{\bar{\e}(s_{\q})}{|\S{\q}|} \sum_{s \in \S{\q}} \bar{\e}(s) \cPkt{W, s}(\q), 
\]
where $\cPkt{W, s}(\q)$ is transferred from $\cPkt{\q_{H}}$ for $(H, \q_{H}) \rightarrow (\q, s)$. By \eqref{eq: nontempered character relation GL(N)}, it suffices to show the vanishing of the corresponding Jacquet modules for $\r_{\q_{H}} := \r_{\q_{I}} \otimes \r_{\q_{II}}$. In fact, it suffices to consider 
\[
\r_{\q} = \times_{(\rho, a, b) \in Jord(\q)} Sp(St(\rho, a), b).
\] 
Then one can check easily that $\Jac^{\theta}_{\zeta x, \cdots, \zeta y} \r_{\q} = 0$ unless there exists a sequence of Jordan blocks 
\[
\{(\rho, A_{i}, B_{i}, \zeta)\}_{i=1}^{n} \subseteq Jord(\q)
\] 
such that $B_{1} = x, A_{n} \geqslant y$, and $B_{i} \leqslant B_{i+1} \leqslant A_{i} + 1$. It is also easy to see $\Jac^{\theta}_{\underbrace{x, \cdots, x}_{n}} \r_{\q} = 0$ if $n > m$. 

\end{proof}

\begin{remark}
This proposition implies the same kind of statements are also true for $\r_{M}^{\Sigma_{0}}(\q, \e)$.
\end{remark}

For functions $\ul(\rho, A, B, \zeta) \in [0, [(A-B+1)/2]]$ and $\ueta(\rho, A, B, \zeta) \in \Two$ on $Jord(\q)$ such that
\[
\e_{\ul, \ueta}(\rho, A, B, \zeta) := \ueta(\rho, A, B, \zeta)^{A - B + 1} (-1)^{[(A-B+1)/2] + \ul(\rho, A, B, \zeta)}
\]
defines a character $\e_{\ul, \ueta}$ of $\S{\q^{>}}^{\Sigma_{0}}$, we define 
\[
\r_{M}^{\Sigma_{0}}(\q, \ul, \ueta) := \circ _{(\rho, A, B, \zeta) \in Jord(\q)} \Jac_{(\rho, A_{\gg}, B_{\gg}, \zeta) \mapsto (\rho, A, B, \zeta)} \r_{M}^{\Sigma_{0}}(\q_{\gg}, \ul, \ueta),
\]
where the composition is taken in the decreasing order,  
\[
\ul(\rho, A, B, \zeta) = \ul(\rho, A_{\gg}, B_{\gg}, \zeta) \text{ and } \ueta(\rho, A, B, \zeta) = \ueta(\rho, A_{\gg}, B_{\gg}, \zeta).
\] 
Then we have the following result about this representation.

\begin{proposition}[\cite{Moeglin:2010}, Proposition 2.8.1]
\label{prop: multiplicity one}
For $\q = \q_{p} \in \cQ{G}$, $\r_{M}^{\Sigma_{0}}(\q, \ul, \ueta)$ only depends on $>_{\q}$, but not on $\q_{\gg}$. Moreover, $\r_{M}^{\Sigma_{0}}(\q, \ul, \ueta)$ is either zero or irreducible. If $\r_{M}^{\Sigma_{0}}(\q, \ul, \ueta) \neq 0$, then
\begin{align*}
\label{eq: multiplicity one}
\r_{M}^{\Sigma_{0}}(\q_{\gg}, \ul, \ueta) \hookrightarrow \Big(\times_{(\rho, A, B, \zeta) \in Jord(\q)} <X^{\gg}_{(\rho, A, B, \zeta)}> \Big) \rtimes \r_{M}^{\Sigma_{0}}(\q, \ul, \ueta),
\end{align*}
where the product is taken in the increasing order.
\end{proposition}

\begin{proof}
First, we would like to show $\r_{M}^{\Sigma_{0}}(\q, \ul, \ueta)$ only depends on $>_{\q}$. Suppose there are two dominating parameter $\q^{1}_{\gg}$ and $\q^{2}_{\gg}$ with discrete diagonal restriction and natural order, we can always choose a third one $\q^{*}_{\gg}$ which dominates both $\q^{1}_{\gg}$ and $\q^{2}_{\gg}$. It is clear that 
\[
\r_{M}^{\Sigma_{0}}(\q^{i}_{\gg}, \ul, \ueta) = \circ _{(\rho, A, B, \zeta) \in Jord(\q)} \Jac_{(\rho, A^{*}_{\gg}, B^{*}_{\gg}, \zeta) \mapsto (\rho, A^{i}_{\gg}, B^{i}_{\gg}, \zeta)} \r_{M}^{\Sigma_{0}}(\q^{*}_{\gg}, \ul, \ueta)
\]
for $i = 1, 2$, where the composition is taken in the decreasing order. For all $(\rho', A', B', \zeta') >_{\q} (\rho, A, B, \zeta)$, it is easy to check 
\[
\Jac_{(\rho, A^{i}_{\gg}, B^{i}_{\gg}, \zeta) \mapsto (\rho, A, B, \zeta)} \text{ and } \Jac_{(\rho', A^{'*}_{\gg}, B^{'*}_{\gg}, \zeta') \mapsto (\rho', A^{'i}_{\gg}, B^{'i}_{\gg}, \zeta')} 
\]
commutes (cf. \cite{Xu:preprint3}, Lemma 5.6). Also note 
\[
\Jac_{(\rho, A^{i}_{\gg}, B^{i}_{\gg}, \zeta) \mapsto (\rho, A, B, \zeta)} \circ \Jac_{(\rho, A^{*}_{\gg}, B^{*}_{\gg}, \zeta) \mapsto (\rho, A^{i}_{\gg}, B^{i}_{\gg}, \zeta)} =  \Jac_{(\rho, A^{*}_{\gg}, B^{*}_{\gg}, \zeta) \mapsto (\rho, A, B, \zeta)}.
\]
Then
\begin{align*}
& \circ _{(\rho, A, B, \zeta) \in Jord(\q)} \Jac_{(\rho, A^{i}_{\gg}, B^{i}_{\gg}, \zeta) \mapsto (\rho, A, B, \zeta)} \r_{M}^{\Sigma_{0}}(\q^{i}_{\gg}, \ul, \ueta) \\
= & \circ _{(\rho, A, B, \zeta) \in Jord(\q)} \Jac_{(\rho, A^{*}_{\gg}, B^{*}_{\gg}, \zeta) \mapsto (\rho, A, B, \zeta)} \r_{M}^{\Sigma_{0}}(\q^{*}_{\gg}, \ul, \ueta).
\end{align*}
This finishes the first part of the proof.

Next we index $Jord(\q)$ according to $>_{\q}$, such that 
\[
(\rho_{i}, A_{i}, B_{i}, \zeta_{i}) >_{\q} (\rho_{i-1}, A_{i-1}, B_{i-1}, \zeta_{i-1}).
\]
Let $\q_{\gg}$ be obtained from $\q$ by shifting $(\rho_{i}, A_{i}, B_{i}, \zeta_{i})$ to $(\rho_{i}, A_{i} + T_{i}, B_{i} + T_{i}, \zeta_{i})$. We also define $\q^{k}$ from $\q_{\gg}$ by shifting $(\rho_{i}, A_{i} + T_{i}, B_{i} + T_{i}, \zeta_{i})$ back to $(\rho_{i}, A_{i}, B_{i}, \zeta_{i})$ for $i \leqslant k$. Suppose $\r_{M}^{\Sigma_{0}}(\q, \ul, \ueta) \neq 0$, then $\r_{M}^{\Sigma_{0}}(\q^{k}, \ul, \ueta) \neq 0$ by definition. We would like to show by induction that $\r_{M}^{\Sigma_{0}}(\q^{k}, \ul, \ueta)$ is irreducible and
\begin{align}
\label{eq: multiplicity one}
\r_{M}^{\Sigma_{0}}(\q^{k-1}, \ul, \ueta) \hookrightarrow 
       \begin{pmatrix}
              \zeta_{k} (B_{k} + T_{k}) & \cdots & \zeta_{k} (B_{k} + 1) \\
              \vdots &  & \vdots \\
              \zeta_{k} (A_{k} + T_{k}) & \cdots & \zeta_{k} (A_{k} + 1)
       \end{pmatrix} 
\rtimes \r_{M}^{\Sigma_{0}}(\q^{k}, \ul, \ueta) 
\end{align}
as the unique irreducible subrepresentation. Note $\q^{0} = \q_{\gg}$ and $\q^{n} = \q$, where $n = |Jord(\q)|$. So let us assume $\r_{M}^{\Sigma_{0}}(\q^{k-1}, \ul, \ueta)$ is irreducible. For $0 \leqslant l \leqslant T_{k} - 1$, we denote
\begin{align*}
\tau_{l} :=
       \begin{pmatrix}
              \zeta_{k} (B_{k} + T_{k}) & \cdots & \zeta_{k} (B_{k} + l + 1) \\
              \vdots &  & \vdots \\
              \zeta_{k} (A_{k} + T_{k}) & \cdots & \zeta_{k} (A_{k} + l + 1)
       \end{pmatrix}.  
\end{align*}
Let $\q^{k-1, l}$ be obtained from $\q^{k-1}$ by shifting $(\rho_{k}, A_{k} + T_{k}, B_{k} + T_{k}, \zeta_{k})$
to $(\rho_{k}, A_{k} + l, B_{k} + l, \zeta_{k})$. We claim $\r_{M}^{\Sigma_{0}}(\q^{k-1, l}, \ul, \ueta)$ is irreducible and
\begin{align*}
\r_{M}^{\Sigma_{0}}(\q^{k-1}, \ul, \ueta) \hookrightarrow \tau_{l} \rtimes \r_{M}^{\Sigma_{0}}(\q^{k-1, l}, \ul, \ueta)
\end{align*}
as the unique irreducible subrepresentation. In particular, $\q^{k-1, 0} = \q^{k}$, so this is what we want.

To prove the claim, we assume it is true for $l + 1$, and we would like to establish it for $l$.  
\begin{align*}
\r_{M}^{\Sigma_{0}}(\q^{k-1}, \ul, \ueta) \hookrightarrow \tau_{l+1} \rtimes \r_{M}^{\Sigma_{0}}(\q^{k-1, l+1}, \ul, \ueta). 
\end{align*}
Since
\[
\Jac_{\zeta_{k}(B_{k} + l + 1), \cdots, \zeta_{k}(A_{k} + l + 1)} \r_{M}^{\Sigma_{0}}(\q^{k-1, l+1}, \ul, \ueta) \neq 0,
\]
there exists an irreducible representation $\sigma^{\Sigma_{0}}_{l}$ and $C \in [B_{k} + l + 1, A_{k} + l + 1]$ such that
\[
 \r_{M}^{\Sigma_{0}}(\q^{k-1, l+1}, \ul, \ueta) \hookrightarrow <\zeta_{k} C, \cdots, \zeta_{k}(A_{k} + l + 1)> \rtimes \sigma^{\Sigma_{0}}_{l}.
\]
If $C > B_{k} + l + 1$, then by Proposition~\ref{prop: reducibility general} there exists $(\rho_{i}, A_{i}, B_{i}, \zeta_{i}) \in Jord(\q)$ for $i < k$ such that 
\[
\rho_{i} = \rho_{k}, \zeta_{i} = \zeta_{k}, B_{i} > B_{k} + l +1 \text{ and } A_{i} \geqslant A_{k} + l + 1.
\] 
But this is impossible by the condition $(\mathcal{P})$ on $>_{\q}$. Therefore, we must have $C = B_{k} + l + 1$. It follows $\sigma^{\Sigma_{0}}_{l}$ is a constituent of $\r_{M}^{\Sigma_{0}}(\q^{k-1, l}, \ul, \ueta)$. Apply Proposition~\ref{prop: reducibility general} to $\q^{k-1, l}$, we have
\begin{align}
\label{eq: multiplicity one 1}
\Jac_{\zeta_{k} C', \cdots, \zeta_{k} C''} \sigma^{\Sigma_{0}}_{l} = 0
\end{align}
for $C' \in [B_{k} + l + 1, A_{k} + T_{k}]$, $C'' \in [A_{k} + l + 1, A_{k} + T_{k}]$. To sum up,
\begin{align*}
\r_{M}^{\Sigma_{0}}(\q^{k-1}, \ul, \ueta) \hookrightarrow 
      \tau_{l+1}
 \times
       \begin{pmatrix}
              \zeta_{k} (B_{k} + l + 1) \\
              \vdots  \\
              \zeta_{k} (A_{k} + l + 1)
       \end{pmatrix}    
       \rtimes \sigma^{\Sigma_{0}}_{l}.
\end{align*}
If we apply $\Jac_{(\rho_{k}, A_{k} + T_{k}, B_{k} + T_{k}, \zeta_{k}) \mapsto (\rho_{k}, A_{k} + l, B_{k} + l, \zeta_{k})}$ to 
\begin{align}
\label{eq: multiplicity one 2} 
      \tau_{l+1}
 \times
       \begin{pmatrix}
              \zeta_{k} (B_{k} + l + 1) \\
              \vdots  \\
              \zeta_{k} (A_{k} + l + 1)
       \end{pmatrix}    
       \rtimes \sigma^{\Sigma_{0}}_{l},
\end{align}
we should get $\sigma^{\Sigma_{0}}_{l}$ by \eqref{eq: multiplicity one 1}. So 
\[
\r_{M}^{\Sigma_{0}}(\q^{k-1, l}, \ul, \ueta) := \Jac_{(\rho_{k}, A_{k} + T_{k}, B_{k} + T_{k}, \zeta_{k}) \mapsto (\rho_{k}, A_{k} + l, B_{k} + l, \zeta_{k})} \r_{M}^{\Sigma_{0}}(\q^{k-1}, \ul, \ueta) = \sigma^{\Sigma_{0}}_{l},
\]
and \eqref{eq: multiplicity one 2} has a unique irreducible subrepresentation. Hence
\begin{align*}
\r_{M}^{\Sigma_{0}}(\q^{k-1}, \ul, \ueta) \hookrightarrow \tau_{l} \rtimes \r_{M}^{\Sigma_{0}}(\q^{k-1, l}, \ul, \ueta)
\end{align*}
as the unique irreducible subrepresentation. This finishes the proof of our claim.

\end{proof}

\begin{remark}
It is an interesting problem to determine when $\r_{M}^{\Sigma_{0}}(\q, \ul, \ueta)$ is not zero, and a solution to such problem would have many applications (e.g. \cite{Moeglin:2011}, \cite{Moeglin2:2011}). In a sequel to this paper, we will give a procedure for finding explicit nonvanishing conditions on $(\ul, \ueta)$ for $\r_{M}^{\Sigma_{0}}(\q, \ul, \ueta)$.
\end{remark}

\begin{corollary}
For $\q = \q_{p} \in \cQ{G}$, if $\r_{M}^{\Sigma_{0}}(\q, \ul, \ueta) \cong \r_{M}^{\Sigma_{0}}(\q, \ul', \ueta') \neq 0$, then $(\ul, \ueta) \sim_{\Sigma_{0}}  (\ul', \ueta')$.
\end{corollary}

\begin{proof}
Suppose $\r_{M}^{\Sigma_{0}}(\q, \ul, \ueta) \cong \r_{M}^{\Sigma_{0}}(\q, \ul', \ueta') \neq 0$, then by applying \eqref{eq: multiplicity one} step by step, one can conclude $\r_{M}^{\Sigma_{0}}(\q_{\gg}, \ul, \ueta) \cong \r_{M}^{\Sigma_{0}}(\q_{\gg}, \ul', \ueta')$. This implies $(\ul, \ueta) \sim_{\Sigma_{0}}  (\ul', \ueta')$.
\end{proof}

Let $\r_{M}(\q, \ul, \ueta)$ be the irreducible representation of $G$ viewed as $\sH(G)$-module in the restriction of $\r^{\Sigma_{0}}_{M}(\q, \ul, \ueta)$ to $G$ if $\r^{\Sigma_{0}}_{M}(\q, \ul, \ueta) \neq 0$, and zero otherwise. Then 
\[
\r_{M}(\q, \ul, \ueta) = \circ _{(\rho, A, B, \zeta) \in Jord(\q)} \bar{\Jac}_{(\rho, A_{\gg}, B_{\gg}, \zeta) \mapsto (\rho, A, B, \zeta)} \r_{M}(\q_{\gg}, \ul, \ueta),
\]
where the composition is taken in the decreasing order. The following proposition follows easily from the definitions and similar statements in the case of discrete diagonal restriction (cf. Proposition~\ref{prop: M parametrization} and Corollary~\ref{cor: M parametrization for G}). 

\begin{proposition}
For $\q = \q_{p} \in \cQ{G}$ and $\e \in \D{\S{\q^{>}}^{\Sigma_{0}}}$,
\[
\r^{\Sigma_{0}}_{M}(\q, \e) = \bigoplus_{\{(\ul, \ueta): \, \e = \e_{\ul, \ueta} \}/\sim_{\Sigma_{0}}} \r^{\Sigma_{0}}_{M}(\q, \ul, \ueta),
\]
and 
\[
\r_{M}(\q, \bar{\e}) = \bigoplus_{\{(\ul, \ueta): \, \bar{\e} = \bar{\e}_{\ul, \ueta} \}/\sim} \r_{M}(\q, \ul, \ueta).
\]
Moreover, 
\[
\bigoplus_{\bar{\e} \leftarrow \e \in \D{\S{\q^{>}}^{\Sigma_{0}}}} \r_{M}^{\Sigma_{0}}(\q, \e)
\]
consists of all irreducible representations of $G^{\Sigma_{0}}$, whose restriction to $G$ belong to $\r_{M}(\q, \bar{\e})$.
\end{proposition}

As a consequence, for $\q = \q_{p} \in \cQ{G}$ we can define $\Pkt{\q}^{\Sigma_{0}}$ to be the set of irreducible representations of $G^{\Sigma_{0}}$, whose restriction to $G$ belong to $\cPkt{\q}$. In the case $G$ is special even orthogonal, if $\S{\q}^{\Sigma_{0}} \neq \S{\q}$, then $\r^{\theta_{0}} \cong \r$ for any irreducible constituent $[\r]$ in $\r_{M}(\q, \bar{\e})$. So it follows from Theorem~\ref{thm: character relation full orthogonal group} that there is a canonical disjoint decomposition 
\[
\Pkt{\q}^{\Sigma_{0}} = \bigsqcup_{\e \in \D{\S{\q}^{\Sigma_{0}}}} \r^{\Sigma_{0}}_{W}(\q, \e)
\]
such that 
\begin{itemize}

\item
$[\r^{\Sigma_{0}}_{W}(\q, \e)|_{G}] = 2 \r_{W}(\q , \bar{\e})$ if $G$ is special even orthogonal and $\S{\q}^{\Sigma_{0}} = \S{\q}$, or $\r_{W}(\q ,\bar{\e})$ otherwise.

\item For any $s \in \S{\q}^{\Sigma_{0}}$ but not in $\S{\q}$ and $(H, \q_{H}) \rightarrow (\q, s)$, the following identity holds

\begin{align*}
f^{H}_{W}(\q_{H}) = \sum_{\bar{\e} \in \D{\S{\q}}} \e(ss_{\q}) f_{G}(\r_{W}^{\Sigma_{0}}(\q ,\e)) \,\,\,\,\,\,\,\,\,\,\, f \in C^{\infty}_{c}(G \rtimes \theta_{0}).
\end{align*}

\end{itemize}

Let us also define for $\e \in \D{\S{\q^{>}}^{\Sigma_{0}}}$,
\begin{align}
\label{eq: MW general full orthogonal group}
\r^{\Sigma_{0}}_{MW}(\q, \e):= \circ _{(\rho, A, B, \zeta) \in Jord(\q)} \Jac_{(\rho, A_{\gg}, B_{\gg}, \zeta) \mapsto (\rho, A, B, \zeta)} \r^{\Sigma_{0}}_{MW}(\q_{\gg}, \e).
\end{align}
Then we have the following theorem.

\begin{theorem}
\label{thm: M/MW general full orthogonal group}
Suppose $\q = \q_{p} \in \cQ{G}$ and $\e \in \D{\S{\q^{>}}^{\Sigma_{0}}}$, 
\[
\r^{\Sigma_{0}}_{MW}(\q, \e) = \begin{cases}
                           \r^{\Sigma_{0}}_{W}(\q, \e \e^{MW/W}_{\q}), & \text{ if $\e \e^{MW/W}_{\q} \in \D{\S{\q}^{\Sigma_{0}}}$,} \\
                            0, & \text{ otherwise.}
                          \end{cases}
\]
and
\[
\r^{\Sigma_{0}}_{M}(\q, \e) = \r^{\Sigma_{0}}_{MW}(\q, \e \e^{M/MW}_{\q}).
\]

\end{theorem}

\begin{proof}
We can assume $G$ is special even orthogonal and $\S{\q}^{\Sigma_{0}} \neq \S{\q}$. Since 
\[
\r_{MW}(\q, \bar{\e}) = \begin{cases}
                           \r_{W}(\q, \bar{\e} \bar{\e}^{MW/W}_{\q}), & \text{ if $\bar{\e} \bar{\e}^{MW/W}_{\q} \in \D{\S{\q}}$,} \\
                            0, & \text{ otherwise}
                          \end{cases}
\]
we have $\r^{\Sigma_{0}}_{MW}(\q, \e \e^{MW/W}_{\q}) \neq 0$ only if $\e  \in \D{\S{\q}^{\Sigma_{0}}}$.

Let us choose $s^{*} \in \S{\q^{>}}^{\Sigma_{0}}$ but not in $\S{\q^{>}}$, and we denote its image in $\S{\q}^{\Sigma_{0}}$ again by $s^{*}$. Then let us define
\begin{align*}
\Pkt{MW, s^{*}}^{\Sigma_{0}}(\q) &:= \sum_{\bar{\e} \in \D{\S{\q^{>}}}} \e(s^{*}s^{>}_{\q}) \r^{\Sigma_{0}}_{MW}(\q, \e), \\
\Pkt{W, s^{*}}^{\Sigma_{0}}(\q) &:= \sum_{\bar{\e} \in \D{\S{\q}}} \e(s^{*}s_{\q}) \r^{\Sigma_{0}}_{W}(\q, \e).
\end{align*}
As in Proposition~\ref{prop: MW/W} one can show
\[
\Pkt{MW, s^{*}}^{\Sigma_{0}}(\q) = \e^{MW/W}_{\q}(s_{\q}^{>}s^{*}) \Pkt{W, s^{*}}^{\Sigma_{0}}(\q)
\]
(cf. \eqref{eq: MW/W}). By the linear independence of twisted characters, we have for $\e \in \D{\S{\q}^{\Sigma_{0}}}$
\[
\e \e^{MW/W}_{\q}(s^{*}s^{>}_{\q}) \r^{\Sigma_{0}}_{MW}(\q, \e \e_{\q}^{MW/W}) = \e^{MW/W}_{\q}(s_{\q}^{>}s^{*}) \cdot \e(s^{*}s_{\q}) \r^{\Sigma_{0}}_{W}(\q, \e).
\]
And hence
\[
\r^{\Sigma_{0}}_{MW}(\q, \e \e_{\q}^{MW/W}) = \r^{\Sigma_{0}}_{W}(\q, \e).
\]
This proves the first part. The second part follows from the case of the discrete diagonal restriction and the fact that $\e^{M/MW}_{\q_{\gg}} = \e^{M/MW}_{\q}$.

\end{proof}

The following corollary is a direct consequence of Theorem~\ref{thm: M/MW general full orthogonal group}.

\begin{corollary}
Suppose $\q = \q_{p} \in \cQ{G}$ and $\e \in \D{\S{\q^{>}}^{\Sigma_{0}}}$, let
\(
\e^{M/W}_{\q}: = \e^{M/MW}_{\q} \e^{MW/W}_{\q}.
\)
Then
\[
\r^{\Sigma_{0}}_{M}(\q, \e) = \begin{cases}
                           \r^{\Sigma_{0}}_{W}(\q, \e \e^{M/W}_{\q}), & \text{ if $\e \e^{M/W}_{\q} \in \D{\S{\q}^{\Sigma_{0}}}$,} \\
                            0, & \text{ otherwise.}
                          \end{cases}
\]
\end{corollary}

Finally for $\q \in \cQ{G}$, 
\[
\cPkt{\q} = \r_{\q_{np}}  \rtimes \cPkt{\q_{p}} 
\]
We define
\[
\Pkt{\q}^{\Sigma_{0}} := \Big(\times_{(\rho, a, b) \in Jord(\q_{np})}Sp(St(\rho, a), b) \Big) \rtimes \Pkt{\q_{p}}^{\Sigma_{0}},
\]
and
\[
\r^{\Sigma_{0}}_{W}(\q, \e) := \Big(\times_{(\rho, a, b) \in Jord(\q_{np})}Sp(St(\rho, a), b) \Big) \rtimes \r^{\Sigma_{0}}_{W}(\q_{p}, \e),
\]
\[
\r^{\Sigma_{0}}_{M}(\q, \e) := \Big(\times_{(\rho, a, b) \in Jord(\q_{np})}Sp(St(\rho, a), b) \Big) \rtimes \r^{\Sigma_{0}}_{M}(\q_{p}, \e)
\]
for $\e \in \D{\S{\q}^{\Sigma_{0}}}$. Then we again have
\[
\r^{\Sigma_{0}}_{M}(\q, \e) = \begin{cases}
                           \r^{\Sigma_{0}}_{W}(\q, \e \e^{M/W}_{\q}), & \text{ if $\e \e^{M/W}_{\q} \in \D{\S{\q}^{\Sigma_{0}}}$,} \\
                            0, & \text{ otherwise.}
                          \end{cases}
\]
For $\ul(\rho, A, B, \zeta) \in [0, [(A-B+1)/2]]$ and $\ueta(\rho, A, B, \zeta) \in \Two$ on $Jord(\q_{p})$ such that $\e_{\ul, \ueta} \in \D{\S{\q^{>}}^{\Sigma_{0}}}$, we also define
\[
\r^{\Sigma_{0}}_{M}(\q, \ul, \ueta) = \Big(\times_{(\rho, a, b) \in Jord(\q_{np})}Sp(St(\rho, a), b) \Big) \rtimes \r^{\Sigma_{0}}_{M}(\q_{p}, \ul, \ueta)
\]
and
\[
\r_{M}(\q, \ul, \ueta) = \Big(\times_{(\rho, a, b) \in Jord(\q_{np})}Sp(St(\rho, a), b) \Big) \rtimes \r_{M}(\q_{p}, \ul, \ueta).
\]

\begin{proposition}[\cite{Moeglin1:2006}, Theorem 6]
For $\q \in \cQ{G}$, $\r^{\Sigma_{0}}_{M}(\q, \ul, \ueta)$ is irreducible or zero.
\end{proposition}

As a consequence of this proposition, $\r_{M}(\q, \ul, \ueta)$ is the irreducible representation of $G$ viewed as $\sH(G)$-module in the restriction of $\r^{\Sigma_{0}}_{M}(\q, \ul, \ueta)$ to $G$ if $\r^{\Sigma_{0}}_{M}(\q, \ul, \ueta) \neq 0$, and zero otherwise. To summarize, we obtain M{\oe}glin's multiplicity free result for Arthur packets.

\begin{theorem}[M{\oe}glin]
For $\q \in \cQ{G}$, 
\[
\Pkt{}^{\Sigma_{0}}(\q) := \bigoplus_{\e \in \D{\S{\q}^{\Sigma_{0}}}} \r^{\Sigma_{0}}_{W}(\q, \e)
\]
\[
(\text{ resp. $\cPkt{}(\q) := \bigoplus_{\bar{\e} \in \D{\S{\q}}} \r_{W}(\q, \bar{\e})$ })
\]
is a multiplicity free representation of $G^{\Sigma_{0}}$ (resp. $\sH(G)$-module).
\end{theorem}

\appendix

\section{Compatibility of endoscopic transfer with Aubert involution}
\label{sec: compatibility}

In this section, we want to establish the compatibility of (twisted) endoscopic transfer with generalized (twisted) Aubert involution (cf. \eqref{diag: twisted compatible with Aubert dual}, \eqref{diag: compatible with Aubert dual} and \eqref{diag: twisted even orthogonal compatible with Aubert dual}). We will start by considering the usual (twisted) Aubert involution. Let $F$ be a $p$-adic field and $G$ be a quasisplit connected reductive group over $F$. Let $\theta$ be an $F$-automorphism of $G$ preserving an $F$-splitting. We denote the space of (resp. twisted) invariant distributions on $G$ by $\D{I}(G)$ (resp. $\D{I}(G^{\theta})$), and denote the space of stable invariant distributions on $G$ by $\D{SI}(G)$. Let $\mathcal{P}^{\theta}$ be the set of $\theta$-stable standard parabolic subgroups of $G$. Let $G^{+} = G \rtimes <\theta>$. For any $\r^{+} \in \Rep(G^{+})$, we define the $\theta$-twisted Aubert involution as follows:
\[
inv^{\theta}(\r^{+}) = \sum_{P \in \mathcal{P}^{\theta}} (-1)^{dim(A_{P})_{\theta}} \Ind^{G}_{P} (\Jac_{P} \r^{+})
\] 
where $A_{P}$ is the maximal split central torus of the Levi component $M$ of $P$. Let $H$ be a twisted endoscopic group of $G$, and we denote by $inv^{H}$ the Aubert involution on Grothendieck group of $\Rep(H)$. Then we want to show the following diagram commutes:
\begin{align}
\label{diag: twisted compatible with Aubert dual general} 
\xymatrix{\D{SI}(H) \ar[d]_{inv^{H}}  \ar[r] & \D{I}(G^{\theta}) \ar[d]^{inv^{\theta}} \\
                \D{SI}(H) \ar[r]  & \D{I}(G^{\theta}) }
\end{align}
where the horizontal maps correspond to the twisted spectral endoscopic transfer. To establish this diagram, we need to know the compatibility of twisted endoscopic transfer with Jacque modules, and we will recall its formulation here following (\cite{Xu:preprint3}, Appendix C). 

For simplicity, we will assume there is an embedding 
\[
\xi: \L{H} \rightarrow \L{G},
\]
and $\xi(\L{H}) \subseteq \Cent(s, \L{G})$ and $\D{H} \cong \Cent(s, \D{G})^{0}$ for some semisimple $s \in \D{G} \rtimes \D{\theta}$. We fix ($\D{\theta}$-stable) $\Gal{F}$-splittings $(\mathcal{B}_{H}, \mathcal{T}_{H}, \{\mathcal{X}_{\alpha_{H}}\})$ and $(\mathcal{B}_{G}, \mathcal{T}_{G}, \{\mathcal{X}_{\alpha}\})$ for $\D{H}$ and $\D{G}$ respectively. By taking certain $\D{G}$-conjugate of $\xi$, we can assume $s \in \mathcal{T}_{G} \rtimes \D{\theta}$ and $\xi(\mathcal{T}_{H}) = (\mathcal{T}_{G}^{\D{\theta}})^{0}$ and $\xi(\mathcal{B}_{H}) \subseteq \mathcal{B}_{G}$. Let $W_{H} = W(\D{H}, \mathcal{T}_{H})$ and $W_{G^{\theta}} = W(\D{G}, \mathcal{T}_{G})^{\D{\theta}}$, then $W_{H}$ can be viewed as a subgroup of $W_{G^{\theta}}$. We also view $\L{H}$ as a subgroup of $\L{G}$ through $\xi$. For $P = MN \in \mathcal{P}^{\theta}$ with standard embedding $\L{P} \hookrightarrow \L{G}$, there exists a torus $S \subseteq (\mathcal{T}_{G}^{\D{\theta}})^{0}$ such that $\L{M} = \Cent(S, \L{G})$. Let $W_{M^{\theta}} =  W(\D{M}, \mathcal{T}_{G})^{\D{\theta}}$. We define
\[
W_{G^{\theta}}(H, M) := \{w \in W_{G^{\theta}} | \, \Cent(w(S), \L{H}) \rightarrow W_{F} \text{ surjective }\}.
\]
For any $w \in W_{G^{\theta}}(H, M)$, let us take $g \in \D{G}$ such that $\Int(g)$ induces $w$. Since $\Cent(w(S), \L{H}) \rightarrow W_{F}$ is surjective, $g \L{P} g^{-1} \cap \L{H}$ defines a parabolic subgroup of $\L{H}$ with Levi component $g \L{M} g^{-1} \cap \L{H}$. So we can choose a standard parabolic subgroup $P'_{w} = M'_{w}N'_{w}$ of $H$ with standard embedding $\L{P'_{w}} \hookrightarrow \L{H}$ such that $\L{P'_{w}}$ (resp. $\L{M'_{w}}$) is $\D{H}$-conjugate to $g \L{P} g^{-1} \cap \L{H}$ (resp. $g \L{M} g^{-1} \cap \L{H}$). In particular, $M'_{w}$ can be viewed as a twisted endoscopic group of $M$, and the embedding $\xi_{M'_{w}}: \L{M'_{w}} \rightarrow \L{M}$ is given by the following diagram:
\[
\xymatrix{ \L{P'_{w}} \ar@{^{(}->}[d] & \L{M'_{w}} \ar[l] \ar[r]^{\xi_{M'_{w}}} & \L{M} \ar[r] & \L{P} \ar@{^{(}->}[d]  \\
\L{H} \ar[r]^{\Int(h)} & \L{H} \ar[r]^{\xi} & \L{G} & \L{G} \ar[l]_{\Int(g)} 
}
\]
where $h \in \D{H}$ induces an element in $W_{H}$. Note the choice of $h$ is unique up to $\D{M}'_{w}$-conjugation, and so is $\xi_{M'_{w}}$. If we change $g$ to $h'gm$, where $h' \in \D{H}$ induces an element in $W_{H}$ and $m \in \D{M}$ induces an element in $W_{M^{\theta}}$, then we still get $P'_{w}$, but $\xi_{M'_{w}}$ changes to $\Int(m^{-1}) \circ \xi_{M'_{w}}$ up to $\D{M}'_{w}$-conjugation. To summarize, for any element $w$ in 
\[
W_{H} \backslash W_{G^{\theta}}(H, M)/W_{M^{\theta}}
\]
we can associate a standard parabolic subgroup $P'_{w} = M'_{w} N'_{w}$ of $H$ and a $\D{M}$-conjugacy class of embeddings $\xi_{M'_{w}}: \L{M'_{w}} \rightarrow \L{M}$. Then the following diagram commutes
\begin{align}
\label{diag: compatible with twisted endoscopic transfer general} 
\xymatrix{\D{SI}(H) \ar[d]_{\+_{w} \Jac_{P'_{w}}}  \ar[r] & \D{I}(G^{\theta}) \ar[d]^{\Jac_{P}} \\
                \bigoplus_{w} \D{SI}(M'_{w}) \ar[r]  & \D{I}(M^{\theta}), }
\end{align}
where the sum is over $W_{H} \backslash W_{G^{\theta}}(H, M)/W_{M^{\theta}}$, and the horizontal maps correspond to the twisted spectral endoscopic transfers with respect to $\xi$ on the top and $\xi_{M'_{w}}$ on the bottom. Let us denote the twisted spectral endoscopic transfer from $H$ to $G$ by $\text{Tran}^{G^{\theta}}_{H}$, and the twisted spectral endoscopic transfer from $M'_{w}$ to $M$ by $\text{Tran}^{M^{\theta}}_{M'_{w}}$. Then we can translate the diagram \eqref{diag: compatible with twisted endoscopic transfer general} into the following identity. For $\Theta^{H} \in \D{SI}(H)$,
\begin{align}
\label{eq: compatible with twisted endoscopic transfer general}
\sum_{w} \Tran^{M^{\theta}}_{M'_{w}} \Jac_{P'_{w}} \Theta^{H} = \Jac_{P} \Tran^{G^{\theta}}_{H} \Theta^{H}.
\end{align}
It follows
\[
\sum_{w} \Ind^{G}_{P} \big(\Tran^{M^{\theta}}_{M'_{w}} \Jac_{P'_{w}} \Theta^{H}\big) =  \Ind^{G}_{P} \big(\Jac_{P} \Tran^{G^{\theta}}_{H} \Theta^{H}\big).
\]
By the compatibility of twisted endoscopic transfer with parabolic induction, 
\[
\Ind^{G}_{P} \Tran^{M^{\theta}}_{M'_{w}} \big(\Jac_{P'_{w}} \Theta^{H}\big) = \Tran^{G^{\theta}}_{H} \Ind^{H}_{P'_{w}} \big(\Jac_{P'_{w}} \Theta^{H}\big).
\]
So
\[
\Tran^{G^{\theta}}_{H} \big( \sum_{w} \Ind^{H}_{P'_{w}} \Jac_{P'_{w}} \Theta^{H} \big) = \Ind^{G}_{P} \Jac_{P} \big(\Tran^{G^{\theta}}_{H} \Theta^{H}\big).
\]
We can multiply both sides by $(-1)^{dim (A_{P})_{\theta}}$, and then sum over $P \in \mathcal{P}^{\theta}$,
\[
\Tran^{G^{\theta}}_{H} \big( \sum_{P \in \mathcal{P}^{\theta}} (-1)^{dim(A_{P})_{\theta}} \sum_{w} \Ind^{H}_{P'_{w}} \Jac_{P'_{w}} \Theta^{H} \big) = inv^{\theta} \big(\Tran^{G^{\theta}}_{H} \Theta^{H}\big).
\]
To establish the diagram \eqref{diag: twisted compatible with Aubert dual general}, it is enough to show
\begin{align*}
\sum_{P \in \mathcal{P}^{\theta}} (-1)^{dim(A_{P})_{\theta}} \sum_{w} \Ind^{H}_{P'_{w}} \Jac_{P'_{w}} \Theta^{H} = inv^{H} \Theta^{H}. 
\end{align*}
By the definition 
\[
inv^{H} \Theta^{H} = \sum_{P' \in \mathcal{P}^{H}} (-1)^{dim A_{P'}} \Ind^{H}_{P'} \Jac_{P'} \Theta^{H} 
\]
where $\mathcal{P}^{H}$ denotes the set of standard parabolic subgroups of $H$. So it suffices to prove the following proposition.

\begin{proposition}
\label{prop: combinatorial identity}

For any $P' = M'N' \in \mathcal{P}^{H}$,
\begin{align}
\label{eq: combinatorial identity}
\sum_{P \in \mathcal{P}^{\theta}} (-1)^{dim(A_{P})_{\theta}} a_{M', H, M} = (-1)^{dim A_{P'}},
\end{align}
where 
\[
a_{M', H, M} := \sharp \{w \in W_{H} \backslash W_{G^{\theta}}(H, M)/W_{M^{\theta}} | P'_{w} = P'\}.
\]

\end{proposition}

Hiraga proved this proposition in the non-twisted case (see \cite{Hiraga:2004}), and we will extend his arguments to prove the twisted case here. First we need to introduce some more notations. 

Let $A^{\D{G}, \D{\theta}}$ be the identity component of $\Gal{F}$-invariant elements in $(\mathcal{T}^{\D{\theta}}_{G})^{0}$, and $A^{\D{H}}$ be the identity component of $\Gal{F}$-invariant elements in $\mathcal{T}_{H}$. By the choice of $\D{G}$-conjugate of $\xi$, we can further assume $\xi(A^{\D{H}}) \subseteq A^{\D{G}, \D{\theta}}$ and there is a $\theta$-stable standard Levi subgroup $M^{H}$ of $G$ such that $\L{M^{H}} = \Cent(A^{\D{H}}, \L{G})$. 

For any $\theta$-stable standard Levi subgroup $M$ of $G$, we denote by $R_{res}(\D{M})$ the root system (not necessarily reduced) obtained by restriction from the root system $R(\D{M}, \mathcal{T}_{G})$ to $(\mathcal{T}^{\D{\theta}}_{G})^{0}$, and we denote the set of simple roots in $R_{res}(\D{M})$ by $\Delta_{res}(\D{M})$. Let $R^{\pm}_{res}(\D{M})$ be the set of positive (negative) roots. We write $r_{res}(M)$ for the number of $\Gal{F}$-orbits in $\Delta_{res}(\D{M})$. Note $\mathcal{P}^{\theta}$ is in bijection with the $\Gal{F}$-stable subsets of $\Delta_{res}(\D{G})$.

For any standard Levi subgroup $M'$ of $H$, we denote by $R(\D{M'})$ the root system $R(\D{M'}, \mathcal{T}_{H})$ and we denote the set of simple roots in $R(\D{M'})$ by $\Delta(\D{M'})$. Let $R^{\pm}(\D{M'})$ be the set of positive (negative) roots. We write $r(M')$ for the number of $\Gal{F}$-orbits in $\Delta(\D{M'})$. Note $\mathcal{P}^{H}$ is in bijection with the $\Gal{F}$-stable subsets of $\Delta(\D{H})$. It is easy to see $R^{\pm}(\D{H}) \subseteq R^{\pm}_{res}(\D{G})$.

If we multiply both sides of \eqref{eq: combinatorial identity} by $(-1)^{dim A^{\D{G}, \D{\theta}}}$, then we will get
\begin{align}
\label{eq: combinatorial identity 1}
\sum_{P \in \mathcal{P}^{\theta}} (-1)^{r_{res}(M)} a_{M', H, M} = (-1)^{r_{res}(M^{H}) + r(M')}.
\end{align}
We will break the proof of this identity into four steps.

{\bf Step 1:}  We fix a $\theta$-stable standard Levi subgroup $M$ of $G$. Let 
\[
D_{M^{\theta}} = \{w \in W_{G^{\theta}} | w^{-1}(\Delta_{res}(\D{M})) \subseteq R^{+}_{res}(\D{G})\}
\]
and
\[
D_{H} = \{ w \in W_{G^{\theta}} | w^{-1}(\Delta(\D{H})) \subseteq R^{+}_{res}(\D{G})\}
\]
We would like to show $D_{H, M^{\theta}} := D^{-1}_{M^{\theta}} \cap D_{H}$ is a set of representatives of $W_{H} \backslash W_{G^{\theta}}/W_{M^{\theta}}$.

\begin{lemma}
\label{lemma: coset representative}
$D_{H}$ (resp. $D_{M^{\theta}}$) is a set of representatives of $W_{H} \backslash W_{G^{\theta}}$ (resp. $W_{M^{\theta}} \backslash W_{G^{\theta}}$).
\end{lemma}

\begin{proof}
For any $w \in W_{G^{\theta}}$, let $\D{B}_{H} := \D{H} \cap w(\mathcal{B}_{G})$. Then $\D{B}_{H}$ is a Borel subgroup of $\D{H}$. So there exists a unique $w_{H} \in W_{H}$ such that $w_{H}(\D{B}_{H}) = \mathcal{B}_{H}$. It follows $\mathcal{B}_{H} = w_{H}(\D{H} \cap w(\mathcal{B}_{G})) = \D{H} \cap w_{H}w(\mathcal{B}_{G})$, and hence $w_{H}w \in D_{H}$. By the uniqueness of $w_{H}$, we see $D_{H}$ is a set of representatives of $W_{H} \backslash W_{G^{\theta}}$.

The proof for $W_{M^{\theta}} \backslash W_{G^{\theta}}$ is similar. One just needs to notice $W_{G^{\theta}} \cong W(\D{G}^{1}, (\mathcal{T}^{\D{\theta}}_{G})^{0})$ and $W_{M^{\theta}} \cong W(\D{M}^{1}, (\mathcal{T}^{\D{\theta}}_{G})^{0})$, where $\D{G}^{1}$ (resp. $\D{M}^{1}$) is the identity component of $\D{\theta}$-invariant elements in $\D{G}$ (resp. $\D{M}$). 
\end{proof}

For $w \in W_{G^{\theta}}$, we define
\[
l_{M^{\theta}}(w) = \sharp \{\alpha \in R^{+}_{res}(\D{M}) | w\alpha \in R^{-}_{res}(\D{G})\}
\]
and
\[
l_{H}(w) = \sharp \{\alpha \in R^{+}(\D{H}) | w\alpha \in R^{-}_{res}(\D{G})\}.
\]

\begin{lemma}
\label{lemma: nonempty intersection}
For any $w \in W_{G^{\theta}}$,
\[
D_{H, M^{\theta}} \cap W_{H}wW_{M^{\theta}} \neq \emptyset.
\]
\end{lemma}

\begin{proof}
Since $D_{H}$ is a set of representatives of $W_{H} \backslash W_{G^{\theta}}$, we can choose $w_{0} \in W_{H}wW_{M^{\theta}}$ such that $w_{0} \in D_{H}$. Note $w_{0}^{-1} \in D_{M^{\theta}}$ if and only if $l_{M^{\theta}}(w_{0}) = 0$. So we can make an induction on $l_{M^{\theta}}(w_{0})$. Suppose $l_{M^{\theta}}(w_{0}) > 0$, then there exists $\alpha \in \Delta_{res}(\D{M})$ such that $w_{0}\alpha \in R^{-}_{res}(\D{G})$. We claim
\[
l_{M^{\theta}}(w_{0}s_{\alpha}) < l_{M^{\theta}}(w_{0})
\]
where $s_{\alpha}$ is corresponding the simple reflection. To see this, note 
\[
s_{\alpha}(R^{+}_{res}(\D{M}) - \mathbb{Z}_{+} \alpha) = R^{+}_{res}(\D{M}) - \mathbb{Z}_{+} \alpha,
\]
and $w_{0}\alpha \in R^{-}_{res}(\D{G})$. So 
\begin{align*}
l_{M^{\theta}}(w_{0}s_{\alpha}) & = \sharp \{\alpha' \in R^{+}_{res}(\D{M}) - \mathbb{Z}_{+} \alpha | w_{0}s_{\alpha} \alpha' \in R^{-}_{res}(\D{G})\} \\
& = \sharp \{\alpha'' \in R^{+}_{res}(\D{M}) - \mathbb{Z}_{+} \alpha | w_{0} \alpha'' \in R^{-}_{res}(\D{G})\}.
\end{align*}
Then
\[
l_{M^{\theta}}(w_{0}) = l_{M^{\theta}}(w_{0}s_{\alpha}) + |\mathbb{Z}_{+} \alpha \cap R^{+}_{res}(\D{M})| > l_{M^{\theta}}(w_{0}s_{\alpha}).
\]
We still need to show $w_{0}s_{\alpha} \in D_{H}$. For that let us consider $(w_{0}s_{\alpha})^{-1}(\Delta(\D{H})) = s_{\alpha}w_{0}^{-1}(\Delta(\D{H}))$. Since 
\[
s_{\alpha}(R^{+}_{res}(\D{G}) - \mathbb{Z}_{+} \alpha) = R^{+}_{res}(\D{G}) - \mathbb{Z}_{+} \alpha,
\]
we only need to show $w_{0}^{-1}(\Delta(\D{H})) \cap \mathbb{Z}_{+} \alpha = \emptyset$. This is guaranteed by the fact that $w_{0}\alpha \in R^{-}_{res}(\D{G})$.

\end{proof}

Now we have the following proposition.

\begin{proposition}
\label{prop: double coset representative}
$D_{H, M^{\theta}}$ is a set of representatives of $W_{H} \backslash W_{G^{\theta}}/W_{M^{\theta}}$.
\end{proposition}

\begin{proof}
In view of Lemma~\ref{lemma: nonempty intersection}, we just need to show $W_{H}wW_{M^{\theta}}$ contains a unique element in $D_{H, M^{\theta}}$ for any $w \in W_{G^{\theta}}$. Suppose $w_{0}, w'_{0} \in D_{H, M^{\theta}} \cap W_{H}wW_{M^{\theta}}$, then we can assume
\[
w'_{0} = w_{H}w_{0}w_{M^{\theta}}
\]
for $w_{H} \in W_{H}$ and $w_{M^{\theta}} \in W_{M^{\theta}}$. First we want to show $w_{H}$ can be chosen to be trivial. Note $w_{H} = 1$ if and only if $l_{H}(w_{H}^{-1}) = 0$. Suppose $l_{H}(w_{H}^{-1}) > 0$, then there exists $\alpha \in \Delta(\D{H})$ such that $w_{H}^{-1}(\alpha) \in R^{-}(\D{H})$. Since $w_{0}, w'_{0} \in D_{H}$, we have $\beta = w_{0}^{-1}w_{H}^{-1}\alpha \in R^{-}_{res}(\D{G})$ and $w_{M^{\theta}}^{-1}\beta = (w'_{0})^{-1} \alpha \in R^{+}_{res}(\D{G})$. So $\beta \in R^{-}_{res}(\D{M})$. Hence
\begin{align*}
w'_{0}  &= w_{H}w_{0}w_{M^{\theta}} = (s_{\alpha} \cdot s_{\alpha})w_{H}w_{0}w_{M^{\theta}} = s_{\alpha} w_{H} s_{w_{H}^{-1}\alpha} w_{0} w_{M^{\theta}} \\
& =  (s_{\alpha} w_{H}) w_{0} (s_{w_{0}^{-1}w_{H}^{-1}\alpha}w_{M^{\theta}}) = (s_{\alpha} w_{H}) w_{0} (s_{\beta}w_{M^{\theta}} ).
\end{align*}
As in the proof of Lemma~\ref{lemma: nonempty intersection}, one can show 
\[
l_{H}(w_{H}^{-1}s_{\alpha}) < l_{H}(w_{H}^{-1}).
\]
So by induction on $l_{H}(w_{H}^{-1})$, we can assume 
\[
w'_{0} = w_{0}w_{M^{\theta}}.
\]
Since $w_{0}, w'_{0} \in D^{-1}_{M^{\theta}}$, we must have $w_{M^{\theta}} = 1$ and hence $w'_{0} = w_{0}$.

\end{proof}

Next we would like to describe 
\[
D_{H, M^{\theta}} \cap W_{G^{\theta}}(H, M),
\] 
which is a set of representatives of $W_{H} \backslash W_{G^{\theta}}(H, M)/W_{M^{\theta}}$. Since $\L{M} = \Cent((A_{\D{M}}^{\D{\theta}})^{0}, \L{G})$, $w \in W_{G^{\theta}}(H, M)$ is characterized by the condition that
\[
\Cent(w(A_{\D{M}}^{\D{\theta}})^{0}, \L{H}) \rightarrow W_{F}
\]
is surjective. For $w \in D_{H}$, the above condition is equivalent to requiring $w(A_{\D{M}}^{\D{\theta}})^{0} \subseteq A^{\D{H}}$. So let us define
\[
\lif{D}_{M^{\theta}} = \{w \in D_{M^{\theta}} | w^{-1}(A_{\D{M}}^{\D{\theta}})^{0} \subseteq A^{\D{H}} \}.
\]
Then $\lif{D}_{H, M^{\theta}}: = \lif{D}_{M^{\theta}}^{-1} \cap D_{H}$ is equal to $D_{H, M^{\theta}} \cap W_{G^{\theta}}(H, M)$. 

For $w \in \lif{D}_{H, M^{\theta}}$, it is easy to see $\D{M'_{w}} = w(\D{M}) \cap \D{H}$. So we would like to define $\D{M'_{w}} := w(\D{M}) \cap \D{H}$ for all $w \in D_{H, M^{\theta}}$, and note $M'_{w}$ is only a standard Levi subgroup of $H$ over $\bar{F}$ in this case. For any standard Levi subgroup $M'$ of $H$ over $\bar{F}$, let us define
\[
D_{M'} = \{ w \in W_{G^{\theta}} | w^{-1}(\Delta(\D{M'})) \subseteq R^{+}_{res}(\D{G})\}.
\]
We also define
\[
D_{M', H, M^{\theta}} := \{w \in D_{H, M^{\theta}} | M'_{w} = M'\}
\] 
and
\[
\lif{D}_{M', H, M^{\theta}} := \{w \in \lif{D}_{H, M^{\theta}} | M'_{w} = M'\}.
\] 
It is clear that $\lif{D}_{M', H, M^{\theta}} \neq \emptyset$ only when $M'$ is defined over $F$.

{\bf Step 2:} We again fix a $\theta$-stable standard Levi subgroup $M$ of $G$, and we will take $M'$ to be standard Levi subgroups of $H$ over $\bar{F}$ (if not specified). Let
\[
\lif{\xi}_{M^{\theta}} = \sum_{w \in \lif{D}_{M^{\theta}}} w,
\]
and
\[
\xi_{M'} = \sum_{w \in D_{M'}} w.
\]
For any 
\(
\xi = \sum_{w \in W_{G^{\theta}}} a_{w} w,
\)
let us write
\[
[\xi]_{H} = \sum_{\substack{w \in W_{G^{\theta}} \\ w(A^{\D{H}}) = A^{\D{H}}}} a_{w} w.
\]
Then we want to show 
\begin{align}
\label{eq: algebraic identity}
[\xi_{H}\lif{\xi}_{M^{\theta}}]_{H} = \sum_{P' \in \mathcal{P}^{H}} a_{M', H, M^{\theta}} [\xi_{M'}]_{H}. 
\end{align}

For any $x \in W_{G^{\theta}}$ satisfying $x(A^{\D{H}}) = A^{\D{H}}$, the coefficient of it in $[\xi_{H}\lif{\xi}_{M^{\theta}}]_{H}$ is given by number of pairs $(d_{H}, d_{M^{\theta}}) \in D_{H} \times \lif{D}_{M^{\theta}}$ such that $x = d_{H}d_{M^{\theta}}$, in other words, we need to count $x\lif{D}^{-1}_{M^{\theta}} \cap D_{H}$.

By Proposition~\ref{prop: double coset representative}, it is enough to count 
\begin{align}
\label{eq: intersection}
(x\lif{D}^{-1}_{M^{\theta}} \cap D_{H}) \cap W_{H}wW_{M^{\theta}}
\end{align}
for all $w \in D_{H, M^{\theta}}$. Let
\begin{align}
\label{eq: algebraic decomposition}
w^{-1}x = w_{M^{\theta}}(x, w) \cdot d_{M^{\theta}}(x, w)
\end{align}
for $w_{M^{\theta}}(x, w) \in W_{M^{\theta}}$ and $d_{M^{\theta}}(x, w) \in D_{M^{\theta}}$. Note this decomposition makes sense for all $x \in W_{G^{\theta}}$.

\begin{lemma}
\label{lemma: intersection}
Suppose $x \in W_{G^{\theta}}$ satisfies $x(A^{\D{H}}) = A^{\D{H}}$ and $w \in D_{H, M^{\theta}}$, then $d_{M^{\theta}}(x, w) \in \lif{D}_{M^{\theta}}$ if and only if $w \in \lif{D}_{H, M^{\theta}}$.
\end{lemma}

\begin{proof}
Since $xd_{M^{\theta}}(x, w)^{-1}(A_{\D{M}}^{\D{\theta}})^{0} = ww_{M^{\theta}}(x, w)(A_{\D{M}}^{\D{\theta}})^{0} = w(A_{\D{M}}^{\D{\theta}})^{0}$, the lemma is clear.
\end{proof}

Before we give the result for \eqref{eq: intersection}, we would like to consider a slightly general situation.

\begin{proposition}
\label{prop: intersection}
For $x \in W_{G^{\theta}}$ and $w \in D_{H, M^{\theta}}$,
\[
(xD^{-1}_{M^{\theta}} \cap D_{H}) \cap W_{H}wW_{M^{\theta}} = \begin{cases}
                                                                                                       \{xd_{M^{\theta}}(x, w)^{-1}\}, & \text{ if } xd_{M^{\theta}}(x, w)^{-1} \in D_{H} \\
                                                                                                        \emptyset, & \text{ otherwise. } 
                                                                                                        \end{cases}
\]                                                                                                                                                                                                         
\end{proposition}

To prove this proposition, we need the following lemma.

\begin{lemma}
\label{lemma: unique expression}
Suppose $w \in D_{H, M^{\theta}}$, every element in $W_{H}wW_{M^{\theta}}$ has a unique expression as
\[
w_{H}ww_{M^{\theta}}
\]
for $w_{M^{\theta}} \in W_{M^{\theta}}$ and $w_{H} \in D_{M'}^{-1} \cap W_{H}$, where $M' = M'_{w}$. Moreover,
\[
l_{H}(w_{M^{\theta}}^{-1}w^{-1}w_{H}^{-1}) \geqslant l_{H}(w_{H}^{-1}).
\]
\end{lemma}

\begin{proof}
As in Lemma~\ref{lemma: coset representative}, one can show $D_{M'}^{-1} \cap W_{H}$ is a set of representatives of $W_{H}/W_{M'}$. Then
\[
w_{H}ww_{M^{\theta}} = (d^{-1}_{M'}w_{M'})ww_{M^{\theta}} =  d^{-1}_{M'}w(w^{-1}w_{M'}w)w_{M^{\theta}},
\]
for $d_{M'} \in D_{M'}$ and $w_{M'} \in W_{M'}$. Since $W_{H} \cap w W_{M^{\theta}} w^{-1} = W_{M'}$, we have $w^{-1}w_{M'}w \in W_{M^{\theta}}$. This proves the existence of the expression. To see the uniqueness, we can assume 
\[
w_{H}ww_{M^{\theta}} = w'_{H}ww'_{M^{\theta}}
\]
both in the desired expressions. Then 
\(
w_{H}ww_{M^{\theta}}(w'_{M^{\theta}})^{-1} = w'_{H}w.
\)
So we can rather assume 
\[
w_{H}ww_{M^{\theta}} = w'_{H}w
\]
It follows $ww_{M^{\theta}} = w^{-1}_{H}w'_{H}w \in W_{H}w$. So $ww_{M^{\theta}}w^{-1} \in W_{H}$. Hence
\[
w_{M'} := ww_{M^{\theta}}w^{-1} \in W_{M'}.
\]
Now we get $w_{H}w_{M'} = w'_{H}$. Since $w_{H}, w'_{H} \in D_{M'}^{-1} \cap W_{H}$, we must have $w_{M'} = 1$. Then $w_{H} = w'_{H}$ and $w_{M^{\theta}} = 1$.

Next we want to show 
\[
l_{H}(w_{M^{\theta}}^{-1}w^{-1}w_{H}^{-1}) \geqslant l_{H}(w_{H}^{-1})
\]
for $w_{M^{\theta}} \in W_{M^{\theta}}$ and $w_{H} \in D_{M'}^{-1} \cap W_{H}$. Note
\[
R^{+}(\D{H}) = \big(R^{+}(\D{H}) - w_{H}(R^{+}(\D{M'}))\big) \bigsqcup w_{H}(R^{+}(\D{M'})).
\]
Then
\[
w_{H}^{-1}(R^{+}(\D{H})) = \big( w_{H}^{-1}(R^{+}(\D{H})) - R^{+}(\D{M'}) \big) \bigsqcup R^{+}(\D{M'}).
\]
We claim $\alpha \in  w_{H}^{-1}(R^{+}(\D{H})) - R^{+}(\D{M'})$ is positive if and only if $w_{M^{\theta}}^{-1}w^{-1} \alpha$ is positive. It is clear that for
$\alpha \in  R(\D{H})$, $\alpha$ is positive if and only if $w^{-1} \alpha$ is positive. So we only need to show $w^{-1} \alpha \notin R_{res}(\D{M})$ for $\alpha \in  w_{H}^{-1}(R^{+}(\D{H})) - R^{+}(\D{M'})$, or equivalently, $\alpha \notin w (R_{res}(\D{M}))$. To see this, we consider 
\[
R^{+}(\D{H}) \cap w_{H}w(R_{res}(\D{M})) = R^{+}(\D{H}) \cap w_{H}\big(R(\D{H}) \cap w(R_{res}(\D{M}))\big) =  R^{+}(\D{H}) \cap w_{H}(R(\D{M'})).
\]
Since $w_{H} \in D_{M'}^{-1} \cap W_{H}$, then $w_{H}(R^{\pm}(\D{M'})) \subseteq R^{\pm}(\D{H})$, and we have 
\[
R^{+}(\D{H}) \cap w_{H}(R(\D{M'})) = w_{H}(R^{+}(\D{M'})).
\] 
Therefore,
\[
R^{+}(\D{H}) \cap w_{H}w(R_{res}(\D{M})) = w_{H}(R^{+}(\D{M'})).
\]
Multiply both sides by $w_{H}^{-1}$,
\[
w_{H}^{-1}(R^{+}(\D{H})) \cap w(R_{res}(\D{M})) = R^{+}(\D{M'}).
\]
From this identity, one can easily see $\alpha \notin w (R_{res}(\D{M}))$ for $\alpha \in  w_{H}^{-1}(R^{+}(\D{H})) - R^{+}(\D{M'})$. This shows our claim. Consequently, we have
\[
l_{H}(w_{M^{\theta}}^{-1}w^{-1}w_{H}^{-1}) = l_{H}(w_{H}^{-1}) + \sharp \{ \alpha \in R^{+}(\D{M'}) | w_{M^{\theta}}^{-1}w^{-1} \alpha \in R^{-}_{res}(\D{G})\} \geq l_{H}(w_{H}^{-1}).
\]

\end{proof}

\begin{corollary}
\label{cor: unique expression}
For $w \in D_{H, M^{\theta}}$, $D_{H} \cap W_{H}wW_{M^{\theta}} \subseteq wW_{M^{\theta}}$.
\end{corollary}

\begin{proof}
For $w_{H}ww_{M^{\theta}} \in D_{H} \cap W_{H}wW_{M^{\theta}}$, we can assume $w_{H} \in D_{M'}^{-1} \cap W_{H}$ by Lemma~\ref{lemma: unique expression}. Then
\[
0 = l_{H}(w_{M^{\theta}}^{-1}w^{-1}w_{H}^{-1}) \geqslant l_{H}(w_{H}^{-1}).
\]
So $l_{H}(w_{H}^{-1}) = 0$, and hence $w_{H} = 1$.
\end{proof}

Now we will prove Proposition~\ref{prop: intersection}. For $x \in W_{G^{\theta}}$ and
\[
y \in (xD^{-1}_{M^{\theta}} \cap D_{H}) \cap W_{H}wW_{M^{\theta}}, 
\]
we can assume $y = ww_{M^{\theta}}$ for $w_{M^{\theta}} \in W_{M^{\theta}}$ by Corollary~\ref{cor: unique expression}. There exists $d_{M^{\theta}} \in D_{M^{\theta}}$ such that 
\[
xd^{-1}_{M^{\theta}} = y = ww_{M^{\theta}}.
\]
So $w^{-1}x = w_{M^{\theta}}d_{M^{\theta}}$. Compared with \eqref{eq: algebraic decomposition}, we get $d_{M^{\theta}} = d_{M^{\theta}}(x, w)$ and $w_{M^{\theta}} = w_{M^{\theta}}(x, w)$. Then $y = xd_{M^{\theta}}(x, w)^{-1} \in D_{H}$. On the other hand, suppose $xd_{M^{\theta}}(x, w)^{-1} \in D_{H}$, it is clear that $xd_{M^{\theta}}(x, w)^{-1} \in xD^{-1}_{M^{\theta}} \cap D_{H}$. Moreover, $xd_{M^{\theta}}(x, w)^{-1} = ww_{M^{\theta}}(x, w) \in W_{H}wW_{M^{\theta}}$. So 
\[
xd_{M^{\theta}}(x, w)^{-1} \in (xD^{-1}_{M^{\theta}} \cap D_{H}) \cap W_{H}wW_{M^{\theta}}.
\] 
This finishes the proof.

Since there is a decomposition 
\[
D_{H, M^{\theta}} = \bigsqcup_{P'} D_{M', H, M^{\theta}},
\]
where the sum is over all standard parabolic subgroup $P'$ of $H$ over $\bar{F}$, we would like to refine Proposition~\ref{prop: intersection} by restricting to $D_{M', H, M^{\theta}}$.

\begin{proposition}
\label{prop: intersection 1}
For $x \in W_{G^{\theta}}$ and $w \in D_{M', H, M^{\theta}}$, $(xD^{-1}_{M^{\theta}} \cap D_{H}) \cap W_{H}wW_{M^{\theta}} \neq \emptyset$ if and only if $x \in D_{M'}$.
\end{proposition}

\begin{proof}
By Proposition~\ref{prop: intersection}, it is enough to show $xd_{M^{\theta}}(x, w)^{-1} \in D_{H}$ if and only if $x \in D_{M'}$. Since 
\[
R^{+}(\D{H}) \cap w(R_{res}(\D{M})) = R^{+}(\D{M'})
\]
and $xd_{M^{\theta}}(x, w)^{-1} = ww_{M^{\theta}}(x, w)$, we have
\begin{align*}
d_{M^{\theta}}(x, w)x^{-1}(R^{+}(\D{M'})) & = d_{M^{\theta}}(x, w)x^{-1}(R^{+}(\D{H})) \cap w_{M^{\theta}}(x, w)^{-1}w^{-1}w(R_{res}(\D{M})) \\
& = d_{M^{\theta}}(x, w)x^{-1}(R^{+}(\D{H})) \cap R_{res}(\D{M}).
\end{align*}
If $xd_{M^{\theta}}(x, w)^{-1} \in D_{H}$, then $d_{M^{\theta}}(x, w)x^{-1}(R^{+}(\D{H})) \subseteq R^{+}_{res}(\D{G})$. So 
\[
d_{M^{\theta}}(x, w)x^{-1}(R^{+}(\D{M'})) \subseteq R^{+}_{res}(\D{M}).
\]
Then
\[
x^{-1}(R^{+}(\D{M'})) \subseteq d_{M^{\theta}}(x, w)^{-1}(R^{+}_{res}(\D{M})) \subseteq R^{+}_{res}(\D{G}).
\]
This means $x \in D_{M'}$.

Conversely, suppose $x \in D_{M'}$ then $x^{-1}(R^{+}(\D{M'})) \subseteq R^{+}_{res}(\D{G})$. We can rewrite it as
\[
d_{M^{\theta}}(x, w)^{-1}(d_{M^{\theta}}(x, w)x^{-1})(R^{+}(\D{M'})) \subseteq R^{+}_{res}(\D{G}).
\]
Since $d_{M^{\theta}}(x, w)x^{-1}(R^{+}(\D{M'})) =  w_{M^{\theta}}(x, w)^{-1}w^{-1}(R^{+}(\D{M'})) \subseteq R_{res}(\D{M})$, we must have
\[
d_{M^{\theta}}(x, w)x^{-1}(R^{+}(\D{M'})) \subseteq R_{res}^{+}(\D{M}).
\]
So it is enough to consider
\begin{align*}
d_{M^{\theta}}(x, w)x^{-1}(R^{+}(\D{H}) - R^{+}(\D{M'})) & = w_{M^{\theta}}(x, w)^{-1}w^{-1}(R^{+}(\D{H}) - R^{+}(\D{M'})) \\
& = w_{M^{\theta}}(x, w)^{-1} \big(w^{-1}(R^{+}(\D{H})) - w^{-1}(R^{+}(\D{M'}))\big) \\
& = w_{M^{\theta}}(x, w)^{-1} \big(w^{-1}(R^{+}(\D{H})) - w^{-1}(R^{+}(\D{H}) \cap w(R_{res}(\D{M})))\big) \\
& = w_{M^{\theta}}(x, w)^{-1} \big(w^{-1}(R^{+}(\D{H})) - w^{-1}(R^{+}(\D{H})) \cap R_{res}(\D{M})\big) \\
& = w_{M^{\theta}}(x, w)^{-1} \big(w^{-1}(R^{+}(\D{H})) - R_{res}(\D{M})\big).
\end{align*}
Since $\alpha \in w^{-1}(R^{+}(\D{H})) - R_{res}(\D{M})$ is positive and not in $R_{res}(\D{M})$, then $w_{M^{\theta}}(x, w)^{-1} \alpha$ is also positive. Therefore,
\[
d_{M^{\theta}}(x, w)x^{-1}(R^{+}(\D{H}) - R^{+}(\D{M'})) \subseteq R^{+}_{res}(\D{G}).
\]
This implies $xd_{M^{\theta}}(x, w)^{-1} \in D_{H}$.

\end{proof}

Next, we will modify Proposition~\ref{prop: intersection} and Proposition~\ref{prop: intersection 1} to count \eqref{eq: intersection}.

\begin{proposition}
\label{prop: intersection modified}
For $x \in W_{G^{\theta}}$ satisfying $x(A^{\D{H}}) = A^{\D{H}}$ and $w \in D_{H, M^{\theta}}$,
\[
(x\lif{D}^{-1}_{M^{\theta}} \cap D_{H}) \cap W_{H}wW_{M^{\theta}} = \begin{cases}
                                                                                                       \{xd_{M^{\theta}}(x, w)^{-1}\}, & \text{ if } w \in \lif{D}_{H, M^{\theta}} \text{ and } xd_{M^{\theta}}(x, w)^{-1} \in D_{H} \\
                                                                                                        \emptyset, & \text{ otherwise. } 
                                                                                                        \end{cases}
\]   

\end{proposition}

\begin{proof}
By Proposition~\ref{prop: intersection},
\[
(xD^{-1}_{M^{\theta}} \cap D_{H}) \cap W_{H}wW_{M^{\theta}} = \begin{cases}
                                                                                                       \{xd_{M^{\theta}}(x, w)^{-1}\}, & \text{ if } xd_{M^{\theta}}(x, w)^{-1} \in D_{H} \\
                                                                                                        \emptyset, & \text{ otherwise. } 
                                                                                                        \end{cases}
\]                                                                                                        
So $(x\lif{D}^{-1}_{M^{\theta}} \cap D_{H}) \cap W_{H}wW_{M^{\theta}} \neq \emptyset$ if and only if $xd_{M^{\theta}}(x, w)^{-1} \in D_{H}$ and $d_{M^{\theta}}(x, w) \in \lif{D}_{M^{\theta}}$. By Lemma~\ref{lemma: intersection}, this is equivalent to requiring $xd_{M^{\theta}}(x, w)^{-1} \in D_{H}$ and $w \in \lif{D}_{H, M^{\theta}}$.

\end{proof}

As a consequence, we can restrict ourselves to the set $\lif{D}_{H, M^{\theta}}$ when counting \eqref{eq: intersection}. Since 
\begin{align}
\label{eq: intersection decomposition}
\lif{D}_{H, M^{\theta}} = \bigsqcup_{P' \in \mathcal{P}^{H}} \lif{D}_{M', H, M^{\theta}},
\end{align}
we can further restrict to each $\lif{D}_{M', H, M^{\theta}}$.

\begin{proposition}
\label{prop: intersection 1 modified}
For $x \in W_{G^{\theta}}$ satisfying $x(A^{\D{H}}) = A^{\D{H}}$ and $w \in \lif{D}_{M', H, M^{\theta}}$, 
\[
(x\lif{D}^{-1}_{M^{\theta}} \cap D_{H}) \cap W_{H}wW_{M^{\theta}} \neq \emptyset
\] 
if and only if $x \in D_{M'}$.
\end{proposition} 

\begin{proof}
By definition, $\lif{D}_{M', H, M^{\theta}} \subseteq D_{M', H, M^{\theta}}$. In view of Proposition~\ref{prop: intersection 1}, it suffices to show for $x \in D_{M'}$, 
\[
(x\lif{D}^{-1}_{M^{\theta}} \cap D_{H}) \cap W_{H}wW_{M^{\theta}} \neq \emptyset.
\]
Since in this case
\[
(xD^{-1}_{M^{\theta}} \cap D_{H}) \cap W_{H}wW_{M^{\theta}} \neq \emptyset,
\]
we have $xd_{M^{\theta}}(x, w)^{-1} \in D_{H}$ by Proposition~\ref{prop: intersection}. Then the result follows from Proposition~\ref{prop: intersection modified} immediately.
\end{proof}

\begin{corollary}
\begin{align*}
[\xi_{H}\lif{\xi}_{M^{\theta}}]_{H} = \sum_{P' \in \mathcal{P}^{H}} a_{M', H, M^{\theta}} [\xi_{M'}]_{H}. 
\end{align*}
\end{corollary}

\begin{proof}
Since $a_{M', H, M^{\theta}} = |\lif{D}_{M', H, M^{\theta}}|$, this identity is an easy consequence of \eqref{eq: intersection decomposition} and Proposition~\ref{prop: intersection 1 modified}.
\end{proof}

{\bf Step 3:} In this step, we would like to establish the following two identities:

\begin{align}
\label{eq: identity A}
\sum_{P \in \mathcal{P}^{\theta}} (-1)^{r_{res}(M)} \lif{\xi}_{M^{\theta}} = (-1)^{r_{res}(M^{H})} w^{G}_{-}w^{M^{H}}_{-}
\end{align}

\begin{align}
\label{eq: identity B}
\sum_{P' \in \mathcal{P}^{H}} (-1)^{r(M')}[\xi_{M'}]_{H} = [\xi_{H}w^{G}_{-}w^{M^{H}}_{-}]_{H}
\end{align}

Here $w^{G}_{-}$ (resp. $w^{M^{H}}_{-}$) is the longest element in $W_{G}$ (resp. $W_{M^{H}}$). It is an easy exercise to show $w^{G}_{-} \in W_{G^{\theta}}$ (resp. $w^{M^{H}}_{-} \in W_{(M^{H})^{\theta}}$). Moreover, we have $w^{G}_{-}(A^{\D{G}, \D{\theta}}) = A^{\D{G}, \D{\theta}}$ (resp. $w^{M^{H}}_{-}(A^{\D{G}, \D{\theta}}) = A^{\D{G}, \D{\theta}}$), i.e., $w^{G}_{-}, \,w^{M^{H}}_{-} \in W_{G^{\theta}}^{\Gal{F}}$.

First let us consider \eqref{eq: identity A}. Recall the left hand side of \eqref{eq: identity A} is equal to
\[
LHS. \eqref{eq: identity A} = \sum_{P \in \mathcal{P}^{\theta}} (-1)^{r_{res}(M)} \sum_{w \in \lif{D}_{M^{\theta}}}w,
\]
and we make the following observation.

\begin{lemma}
If $w \in \lif{D}_{M^{\theta}}$, then $w \in W_{G^{\theta}}^{\Gal{F}}$.
\end{lemma}

\begin{proof}
For $w \in \lif{D}_{M^{\theta}}$, we have $w^{-1}(\Delta_{res}(\D{M})) \subseteq R^{+}_{res}(\D{G})$ and $w^{-1}(A_{\D{M}}^{\D{\theta}})^{0} \subseteq A^{\D{H}}$ by the definition. We take any $\sigma \in \Gal{F}$. Since $A^{\D{H}} \subseteq A^{\D{G}, \D{\theta}}$, it is easy to see $\sigma(w) \in W_{M^{\theta}} w$. On the other hand, 
\[
\sigma(w)^{-1}(\Delta_{res}(\D{M})) = \sigma(w^{-1}(\Delta_{res}(\D{M}))) \subseteq \sigma(R^{+}_{res}(\D{G})) = R^{+}_{res}(\D{G}).
\]
So $\sigma(w) \in D_{M^{\theta}}$. By Lemma~\ref{lemma: coset representative}, $\sigma(w) = w$. Hence $w \in W_{G^{\theta}}^{\Gal{F}}$.

\end{proof}

As a consequence, we can restrict the summation on the left hand side of \eqref{eq: identity A} to $W_{G^{\theta}}^{\Gal{F}}$. Moreover, for $w \in W_{G^{\theta}}^{\Gal{F}}$, the condition that $w \in \lif{D}_{M^{\theta}}$ is equivalent to
\[
R^{+}_{res}(\D{M^{H}}) \subseteq w^{-1}(R^{+}_{res}(\D{M})) \subseteq R^{+}_{res}(\D{G}).
\]
So

\begin{align*}
LHS. \eqref{eq: identity A} & = \sum_{P \in \mathcal{P}^{\theta}} (-1)^{r_{res}(M)} \sum_{\substack{w \in W_{G^{\theta}}^{\Gal{F}} \\ R^{+}_{res}(\D{M^{H}}) \subseteq w^{-1}(R^{+}_{res}(\D{M})) \subseteq R^{+}_{res}(\D{G})} } w \\
& = \sum_{w \in W_{G^{\theta}}^{\Gal{F}}}  \big( \sum_{\substack{P \in \mathcal{P}^{\theta} \\ w(R^{+}_{res}(\D{M^{H}})) \subseteq R^{+}_{res}(\D{M}) \subseteq w(R^{+}_{res}(\D{G}))}} (-1)^{r_{res}(M)}\big) \, w
\end{align*}
For $w \in W_{G^{\theta}}^{\Gal{F}}$, we define
\[
I_{w} = \{\alpha \in \Delta_{res}(\D{G}) | n_{\beta \alpha} \neq 0 \text{ for some } \beta \in \Delta_{res}(\D{M^{H}}) \},
\]
where $w \beta = \sum_{\alpha \in \Delta_{res}(\D{G})} n_{\beta\alpha} \, \alpha$. Then we have the following lemma.

\begin{lemma}
For $w \in W_{G^{\theta}}^{\Gal{F}}$, 
\[
I_{w} = w (\Delta_{res}(\D{M^{H}}))
\] 
if and only if 
\[
w(R^{+}_{res}(\D{M^{H}})) \subseteq R^{+}_{res}(\D{M}) \subseteq w(R^{+}_{res}(\D{G}))
\]
for some $P \in \mathcal{P}^{\theta}$.
\end{lemma}

\begin{proof}
If there exists $P \in \mathcal{P}^{\theta}$ such that
\[
w(R^{+}_{res}(\D{M^{H}})) \subseteq R^{+}_{res}(\D{M}) \subseteq w(R^{+}_{res}(\D{G})),
\]
then $I_{w} \subseteq \Delta_{res}(\D{M}) \subseteq w(R^{+}_{res}(\D{G}))$. So $w^{-1}(I_{w}) \subseteq R^{+}_{res}(\D{G})$. We claim $w (\Delta_{res}(\D{M^{H}})) \subseteq \Delta_{res}(\D{G})$. Suppose $\beta \in \Delta_{res}(\D{M^{H}})$, since $w \beta \in R^{+}_{res}(\D{M})$, we can assume $w \beta = \sum_{\alpha \in I_{w}} n_{\beta\alpha} \, \alpha$ where $n_{\beta\alpha} \geqslant 0$. Hence
\[
\beta = w^{-1}(w \beta) = \sum_{\alpha \in I_{w}} n_{\beta\alpha} \, (w^{-1} \alpha).
\]
Since $w^{-1} \alpha \in R^{+}_{res}(\D{G})$ for $\alpha \in I_{w}$, this can only happen when $n_{\beta\alpha} = 0$ except for one simple root, i.e., $w\beta \in \Delta_{res}(\D{G})$. This shows our claim. As a consequence, $I_{w} = w (\Delta_{res}(\D{M^{H}}))$.

Conversely, if $I_{w} = w (\Delta_{res}(\D{M^{H}}))$, we can let $M_{I_{w}}$ be the standard Levi subgroup of $G$ associated with the subset of simple roots $I_{w}$. Then we have
\[
w(R^{+}_{res}(\D{M^{H}})) \subseteq R^{+}_{res}(\D{M}_{I_{w}}) \subseteq w(R^{+}_{res}(\D{G})).
\] 
\end{proof}

In view of this lemma, we can assume 
\(
I_{w} = w (\Delta_{res}(\D{M^{H}})).
\) 
Let $M(w)$ be the standard Levi subgroup of $G$ associated with the subset of $\alpha \in \Delta_{res}(\D{G})$ such that $w^{-1}\alpha \in R^{+}_{res}(\D{G})$. It is clear that $M(w) \supseteq M_{I_{w}}$ under our assumption. Then

\begin{align*}
LHS. \eqref{eq: identity A} & = \sum_{\substack{w \in W_{G^{\theta}}^{\Gal{F}} \\ I_{w} = w (\Delta_{res}(\D{M^{H}}))}}  \big( \sum_{\substack{P \in \mathcal{P}^{\theta} \\ M_{I_{w}} \subseteq M \subseteq M(w)}} (-1)^{r_{res}(M)} \big) \, w 
= \sum_{\substack{w \in W_{G^{\theta}}^{\Gal{F}} \\ I_{w} = w (\Delta_{res}(\D{M^{H}})), \, M_{I_{w}} = M(w)}} (-1)^{r_{res}(M_{I_{w}})} \, w 
\end{align*}
Note $r_{res}(M_{I_{w}}) = r_{res}(M^{H})$, so

\begin{align*}
LHS. \eqref{eq: identity A} = (-1)^{r_{res}(M^{H})} \sum_{\substack{w \in W_{G^{\theta}}^{\Gal{F}} \\ I_{w} = w (\Delta_{res}(\D{M^{H}})), \, M_{I_{w}} = M(w)}} \, w 
\end{align*}
Then \eqref{eq: identity A} follows from the following lemma.

\begin{lemma}
Suppose $w \in W_{G^{\theta}}^{\Gal{F}}$ satisfies $I_{w} = w (\Delta_{res}(\D{M^{H}}))$ and $M_{I_{w}} = M(w)$, then $w = w^{G}_{-}w^{M^{H}}_{-}$.
\end{lemma}

\begin{proof}
Since $(w^{G}_{-})^{2} = (w^{M^{H}}_{-})^{2} = 1$, it is equivalent to show $w^{M^{H}}_{-}w^{-1} = w^{G}_{-}$, i.e., 
\[
w^{M^{H}}_{-}w^{-1}(\Delta_{res}(\D{G})) \subseteq R^{-}_{res}(\D{G}).
\]
Since $w^{-1}(I_{w}) = \Delta_{res}(\D{M^{H}})$, $w^{M^{H}}_{-}w^{-1}(I_{w}) \subseteq R^{-}_{res}(\D{G})$. Since $M_{I_{w}} = M(w)$, 
\[
w^{-1}(\Delta_{res}(\D{G}) - I_{w}) \subseteq R^{-}_{res}(\D{G}).
\]
By $w^{-1}(I_{w}) = \Delta_{res}(\D{M^{H}})$ again, we have $w^{-1}(\Delta_{res}(\D{G}) - I_{w}) \cap R_{res}(\D{M^{H}}) = \emptyset$. Hence
\[
w^{M^{H}}_{-}w^{-1}(\Delta_{res}(\D{G}) - I_{w}) \subseteq R^{-}_{res}(\D{G}).
\]
This finishes the proof.
\end{proof}

Next let us consider \eqref{eq: identity B}. Recall the left hand side of \eqref{eq: identity B} is equal to

\begin{align*}
LHS. \eqref{eq: identity B} = \sum_{P' \in \mathcal{P}^{H}} (-1)^{r(M')} \sum_{\substack{w \in D_{M'} \\ w(A^{\D{H}}) = A^{\D{H}}}} \, w 
\end{align*}
For $w \in W_{G^{\theta}}$ satisfying $w(A^{\D{H}}) = A^{\D{H}}$, we have for any $\sigma \in \Gal{F}$ and $\alpha \in R(\D{H})$, $w^{-1}(\alpha)$ and $w^{-1}(\sigma_{H}(\alpha))$ are both positive or negative, where $\sigma_{H}$ is the Galois action in $\L{H}$. This is because
\[
w^{-1}(\alpha)|_{A^{\D{H}}} = w^{-1}(\alpha|_{A^{\D{H}}}) = w^{-1}(\sigma_{H}(\alpha)|_{A^{\D{H}}}) = w^{-1}(\sigma_{H}(\alpha))|_{A^{\D{H}}} \neq 0.
\]
So the subset of $\alpha \in \Delta(\D{H})$ satisfying $w^{-1}\alpha \in R^{+}_{res}(\D{G})$ determines a standard Levi subgroup $M'(w)$ of $H$. Then

\begin{align*}
LHS. \eqref{eq: identity B} & = \sum_{\substack{w \in W_{G^{\theta}} \\ w(A^{\D{H}}) = A^{\D{H}}}} \big( \sum_{\substack{P' \in \mathcal{P}^{H} \\ M' \subseteq M'(w)}} (-1)^{r(M')} \big) \, w  
= \sum_{\substack{w \in W_{G^{\theta}} \\ w(A^{\D{H}}) = A^{\D{H}} \\ w^{-1}(\Delta(\D{H})) \subseteq R^{-}_{res}(\D{G})}} \, w
\end{align*}
On the other hand, the right hand side of \eqref{eq: identity B} is equal to

\begin{align*}
RHS. \eqref{eq: identity B} = [\xi_{H}w^{G}_{-}]_{H} \cdot w^{M^{H}}_{-}
\end{align*}
One can check easily that $D_{H}w^{G}_{-}$ consists of $w \in W_{G^{\theta}}$ such that $w^{-1}(\Delta(\D{H})) \subseteq R^{-}_{res}(\D{G})$. So

\begin{align*}
RHS. \eqref{eq: identity B} = \big(\sum_{\substack{w \in W_{G^{\theta}} \\ w(A^{\D{H}}) = A^{\D{H}} \\ w^{-1}(\Delta(\D{H})) \subseteq R^{-}_{res}(\D{G})}} \, w \big) w^{M^{H}}_{-}  = \sum_{\substack{w \in W_{G^{\theta}} \\ w(A^{\D{H}}) = A^{\D{H}} \\ w^{-1}(\Delta(\D{H})) \subseteq R^{-}_{res}(\D{G})}} \, w 
\end{align*}
The last equality is due to the fact that for $w \in W_{G^{\theta}}$ satisfying $w(A^{\D{H}}) = A^{\D{H}}$, 
\[
w^{-1}(\Delta(\D{H})) \subseteq R^{-}_{res}(\D{G})
\] 
if and only if 
\[
(ww^{M^{H}}_{-})^{-1} (\Delta(\D{H})) \subseteq R^{-}_{res}(\D{G}).
\] 
One can show this by restricting the roots to $A^{\D{H}}$. Then the proof is completed by comparing the last expressions of RHS. \eqref{eq: identity B} and LHS. \eqref{eq: identity B}.

{\bf Step 4:} We will establish \eqref{eq: combinatorial identity 1} by using the identities \eqref{eq: identity A} and \eqref{eq: identity B}. First, we multiply \eqref{eq: identity A} by $\xi_{H}$, and compare it with \eqref{eq: identity B},
\begin{align*}
\sum_{P \in \mathcal{P}^{\theta}} (-1)^{r_{res}(M)} [\xi_{H}\lif{\xi}_{M^{\theta}}]_{H} = (-1)^{r_{res}(M^{H})} [\xi_{H}w^{G}_{-}w^{M^{H}}_{-}]_{H} = (-1)^{r_{res}(M^{H})} \sum_{P' \in \mathcal{P}^{H}} (-1)^{r(M')}[\xi_{M'}]_{H}
\end{align*}
Then we can use \eqref{eq: algebraic identity} to expand the left hand side,
\begin{align*}
LHS. = \sum_{P \in \mathcal{P}^{\theta}} (-1)^{r_{res}(M)} \sum_{P' \in \mathcal{P}^{H}} a_{M', H, M^{\theta}} [\xi_{M'}]_{H} = \sum_{P' \in \mathcal{P}^{H}} \big( \sum_{P \in \mathcal{P}^{\theta}} (-1)^{r_{res}(M)} a_{M', H, M^{\theta}}  \big) \, [\xi_{M'}]_{H}
\end{align*}
By the linear independence of $[\xi_{M'}]_{H}$, we get 
\[
\sum_{P \in \mathcal{P}^{\theta}} (-1)^{r_{res}(M)} a_{M', H, M^{\theta}} = (-1)^{r_{res}(M^{H}) + r(M')}
\]
for any $P' \in \mathcal{P}^{H}$.

\subsection{Generalized Aubert involution}

We would like to generalize the diagram ~\eqref{diag: twisted compatible with Aubert dual general} to \eqref{diag: twisted compatible with Aubert dual}, \eqref{diag: compatible with Aubert dual} and \eqref{diag: twisted even orthogonal compatible with Aubert dual}. Note we will only show the commutativity of \eqref{diag: twisted compatible with Aubert dual} and \eqref{diag: compatible with Aubert dual} when restricting to distributions associated with elementary parameters. In fact, they are not commutative in general. Let $G$ be a quasisplit symplectic or special orthogonal group. We fix a positive integer $X_{0}$ and write $x_{0} = (X_{0} - 1) / 2$. We also fix a self-dual irreducible unitary supercuspidal representation $\rho$ of $GL(d_{\rho})$.
Let $\mathcal{P}_{d_{\rho}}$ be the set of standard parabolic subgroups $P$ of $G$ whose Levi component $M$ is isomorphic to 
\begin{align*}
GL(a_{1}d_{\rho}) \times \cdots \times GL(a_{l}d_{\rho}) \times G(n - \sum_{i \in [1,l]}a_{i}d_{\rho}).
\end{align*}
Then we can define for $\r \in \Rep(G)$,
\[
inv_{< X_{0}}(\r) := \sum_{P \in \mathcal{P}_{d_{\rho}}} (-1)^{dim A_{M}} \Ind^{G}_{P}(\Jac_{P} (\r)_{< x_{0}}).
\]
It is clear that
\[
[inv_{< X_{0}}(\r)] = \bar{inv}_{< X_{0}}([\r]).
\]
So \eqref{diag: twisted compatible with Aubert dual} is equivalent to
\begin{align*}
\xymatrix{\D{SI}(G) \ar[d]_{inv_{< X_{0}}}  \ar[r] & \D{I}(N^{\theta}) \ar[d]^{inv^{\theta_{N}}_{< X_{0}}} \\
                \D{SI}(G) \ar[r]  & \D{I}(N^{\theta}). }
\end{align*}
To prove this, we can follow the argument for \eqref{diag: twisted compatible with Aubert dual general}. For $P \in \mathcal{P}^{\theta_{N}}_{d_{\rho}}$, we specialize the diagram~\eqref{diag: compatible with twisted endoscopic transfer general} in our case: 
\begin{align}
\label{diag: truncated Jacquet compatible with twisted endoscopic transfer general} 
\xymatrix{\D{SI}(G) \ar[d]_{\+_{w} (\Jac_{P'_{w}})_{< x_{0}}}  \ar[r] & \D{I}(N^{\theta}) \ar[d]^{(\Jac_{P})_{< x_{0}}} \\
                \bigoplus_{w} \D{SI}(M'_{w}) \ar[r]  & \D{I}(M^{\theta}), }
\end{align}
where the sum is restricted to those $w$ satisfying $P'_{w} \in \mathcal{P}_{d_{\rho}}$. Unlike \eqref{diag: compatible with twisted endoscopic transfer general}, the above diagram may not commute in certain cases when we apply it to distributions not associated with elementary parameters. This is the reason that we want to restrict \eqref{diag: twisted compatible with Aubert dual} (similarly \eqref{diag: compatible with Aubert dual}) to distributions associated with elementary parameters.
By \eqref{diag: truncated Jacquet compatible with twisted endoscopic transfer general}, it suffices to show for any $P' \in \mathcal{P}_{d_{\rho}}$,
\begin{align}
\label{eq: combinatorial identity GL(N)}
\sum_{P \in \mathcal{P}^{\theta_{N}}_{d_{\rho}}} (-1)^{dim(A_{P})_{\theta}} a_{M', G, M} = (-1)^{dim A_{P'}}.
\end{align}
By Proposition~\ref{prop: combinatorial identity}, we have
\begin{align*}
\sum_{P \in \mathcal{P}^{\theta_{N}}} (-1)^{dim(A_{P})_{\theta}} a_{M', G, M} = (-1)^{dim A_{P'}}.
\end{align*}
Therefore \eqref{eq: combinatorial identity GL(N)} follows from the simple fact that $a_{M', G, M} = 0$ when $P \notin \mathcal{P}^{\theta_{N}}_{d_{\rho}}$.

The case of \eqref{diag: compatible with Aubert dual} is similar. For \eqref{diag: twisted even orthogonal compatible with Aubert dual}, let $\mathcal{P}^{\theta_{0}}_{d_{\rho}}$ be the set of $\theta_{0}$-stable standard parabolic subgroups in $\mathcal{P}_{d_{\rho}}$. Then we can define for $\r^{\Sigma_{0}} \in \Rep(G^{\Sigma_{0}})$,
\[
inv^{\theta_{0}}_{< X_{0}}(\r^{\Sigma_{0}}) := \sum_{P \in \mathcal{P}^{\theta_{0}}_{d_{\rho}}} (-1)^{dim (A_{M})_{\theta_{0}}} \Ind^{G^{\Sigma_{0}}}_{P^{\Sigma_{0}}}(\Jac_{P^{\Sigma_{0}}} (\r^{\Sigma_{0}})_{< x_{0}}).
\] 
For $P \in \mathcal{P}^{\theta_{0}}_{d_{\rho}}$ and $G(n - \sum_{i \in [1,l]}a_{i}d_{\rho}) \neq SO(2)$, it is clear that $(A_{M})_{\theta_{0}} = A_{M}$ and $\Jac_{P^{\Sigma_{0}}} = \widetilde{\Jac}_{P^{\Sigma_{0}}}$. If $G(n - \sum_{i \in [1,l]}a_{i}d_{\rho}) = SO(2)$, then $\dim (A_{M})_{\theta_{0}} = \dim(A_{M}) - 1$, but the effect of $\widetilde{\Jac}_{P^{\Sigma_{0}}}$ in taking the twisted character also differs from $\Jac_{P^{\Sigma_{0}}}$ by a negative sign. So we have
\[
f_{G}(inv^{\theta_{0}}_{< X_{0}}(\r^{\Sigma_{0}})) = f_{G}(inv_{< X_{0}}(\r^{\Sigma_{0}})), \,\,\,\,\,\,\,\,\,\,\,\,\, f \in C^{\infty}_{c}(G \rtimes \theta_{0}).
\]
As a result, \eqref{diag: twisted even orthogonal compatible with Aubert dual} is equivalent to
\begin{align*}
\xymatrix{\D{SI}(H) \ar[d]_{inv^{H}_{< X_{0}}}  \ar[r] & \D{I}(G^{\theta_{0}}) \ar[d]^{inv^{\theta_{0}}_{< X_{0}}} \\
                \D{SI}(H) \ar[r]  & \D{I}(G^{\theta_{0}}). }
\end{align*}
The rest of the argument is similar to \eqref{diag: twisted compatible with Aubert dual}.

\bibliographystyle{amsalpha}

\bibliography{reps}

\end{document}

%% file: Macros_reps.tex
\newtheorem{theorem}{Theorem}[section]
\newtheorem{lemma}[theorem]{Lemma}
\newtheorem{proposition}[theorem]{Proposition}
\newtheorem{corollary}[theorem]{Corollary}

\newtheorem*{theorem*}{Theorem}
\theoremstyle{remark}
\newtheorem{remark}[theorem]{Remark}
\newtheorem{definition}[theorem]{Definition}
\newtheorem{example}[theorem]{Example}

\numberwithin{equation}{section}

\newcommand{\lif}[1]{\widetilde{#1}}  

\renewcommand{\P}[1]{\Phi(#1)}

\newcommand{\Q}[1]{\Psi(#1)}

\newcommand{\Pbd}[1]{\Phi_{bdd}(#1)}

\newcommand{\Pdt}[1]{\Phi_{2}(#1)}

\newcommand{\cQ}[1]{\bar{\Psi}(#1)}

\newcommand{\cPdt}[1]{\bar{\Phi}_{2}(#1)}

\newcommand{\p}{\phi}

\newcommand{\q}{\psi}

\newcommand{\Pkt}[1]{\Pi_{#1}}

\newcommand{\cPkt}[1]{\bar{\Pi}_{#1}}

\renewcommand{\r}{\pi}

\newcommand{\sH}{\bar{\mathcal{H}}}

\newcommand{\cS}[1]{\bar{S}_{#1}}

\renewcommand{\S}[1]{\mathcal{S}_{#1}}

\newcommand{\+}{\oplus}

\newcommand{\x}{\omega}

\renewcommand{\L}[1]{{}^L#1}

\newcommand{\D}[1]{\widehat{#1}}

\newcommand{\Gal}[1]{\Gamma_{#1}}

\renewcommand{\Im}{\text{Im}\,}

\newcommand{\Ind}{\text{Ind}}

\newcommand{\Cent}{\text{Cent}}

\newcommand{\Two}{\mathbb{Z}_{2}}

\newcommand{\C}{\mathbb{C}}

\newcommand{\Norm}{\text{Norm}}

\newcommand{\Int}{\text{Int}}

\newcommand{\e}{\varepsilon}

\newcommand{\Rep}{\text{Rep}}

\newcommand{\Jac}{\text{Jac}}

\newcommand{\ul}{\underline{l}}

\newcommand{\ueta}{\underline{\eta}}

\newcommand{\Tran}{\text{Tran}}
